\documentclass[10pt,reqno]{amsart}
\usepackage{amssymb}
\usepackage{amscd}
\usepackage{hyperref}

\newtheorem{theorem}{Theorem}[section]
\newtheorem{corollary}[theorem]{Corollary}
\newtheorem{lemma}[theorem]{Lemma}
\newtheorem{proposition}[theorem]{Proposition}

\newtheorem{definition}[theorem]{Definition}
\newtheorem{remark}[theorem]{Remark}

\newtheorem{example}[theorem]{Example}
\numberwithin{equation}{section}
\begin{document}
\title[Index theory on Riemannian foliations]{Index theory for basic Dirac
operators \\
on Riemannian foliations}
\author[J.~Br\"{u}ning]{Jochen Br\"{u}ning}
\address{Institut f\"{u}r Mathematik \\
Humboldt Universit\"{a}t zu Berlin \\
Unter den Linden 6 \\
D-10099 Berlin, Germany}
\email[J.~Br\"{u}ning]{bruening@mathematik.hu-berlin.de}
\author[F. W.~Kamber]{Franz W.~Kamber${}^1$}
\address{Department of Mathematics, University of Illinois \\
1409 W. Green Street \\
Urbana, IL 61801, USA}
\email[F. W.~Kamber]{kamber@math.uiuc.edu}
\author[K.~Richardson]{Ken Richardson}
\address{Department of Mathematics \\
Texas Christian University \\
Fort Worth, Texas 76129, USA}
\email[K.~Richardson]{k.richardson@tcu.edu}
\thanks{${}^1$Supported in part by The National Science Foundation under
Grant DMS-9504084.}
\subjclass[2000]{53C12, 57R30, 58G10}
\keywords{foliation, basic, index, transversally elliptic}
\date{\today }

\begin{abstract}
In this paper we prove a formula for the analytic index of a basic
Dirac-type operator on a Riemannian foliation, solving a problem that has
been open for many years. We also consider more general indices given by
twisting the basic Dirac operator by a representation of the orthogonal
group. The formula is a sum of integrals over blowups of the strata of the
foliation and also involves eta invariants of associated elliptic operators.
As a special case, a Gauss-Bonnet formula for the basic Euler characteristic
is obtained using two independent proofs.
\end{abstract}

\maketitle
\tableofcontents

\section{Introduction}

\vspace{1pt}Let $\left( M,\mathcal{F}\right) $ be a smooth, closed manifold
endowed with a Riemannian foliation. Let $D_{b}^{E}:\Gamma _{b}\left(
M,E^{+}\right) \rightarrow \Gamma _{b}\left( M,E^{-}\right) $ be a basic,
transversally elliptic differential operator acting on the basic sections of
a foliated vector bundle $E$. The basic index $\mathrm{ind}_{b}\left(
D_{b}^{E}\right) $ is known to be a well-defined integer, and it has been an
open problem since the 1980s to write this integer in terms of geometric and
topological invariants. Our main theorem (Theorem \ref{basicIndexTheorem})
expresses $\mathrm{ind}_{b}\left( D_{b}^{E}\right) $ as a sum of integrals
over the different strata of the Riemannian foliation, and it involves the
eta invariant of associated equivariant elliptic operators on spheres normal
to the strata. The result is%
%\begin{multline*}
\begin{align*}
\mathrm{ind}_{b}\left( D_{b}^{E}\right) &= \int_{\widetilde{M_{0}}\diagup 
\overline{\mathcal{F}}}A_{0,b}\left( x\right) ~\widetilde{\left\vert
dx\right\vert }+\sum_{j=1}^{r}\beta \left( M_{j}\right) ~,~ \\
\beta \left( M_{j}\right) &= \frac{1}{2}\sum_{\tau }\frac{1}{n_{\tau } 
\mathrm{rank~}W^{\tau }} \left( -\eta \left( D_{j}^{S+,\tau }\right) +
h\left( D_{j}^{S+,\tau }\right) \right) \int_{\widetilde{M_{j}}\diagup 
\overline{\mathcal{F}}}A_{j,b}^{\tau }\left( x\right) ~\widetilde{\left\vert
dx\right\vert }\ .
\end{align*}
%\end{multline*}%
The notation will be explained later; the integrands $A_{0,b}\left( x\right) 
$ and $A_{j,b}^{\tau }\left( x\right) $ are the familar Atiyah-Singer
integrands corresponding to local heat kernel supertraces of induced
elliptic operators over closed manifolds. Even in the case when the operator 
$D$ is elliptic, this result was not known previously. We emphasize that
every part of the formula is explicitly computable from local information
provided by the operator and foliation. Even the eta invariant of the
operator $D_{j}^{S+,\tau }$ on a sphere is calculated directly from the
principal transverse symbol of the operator $D_{b}^{E}$ at one point of a
singular stratum. The de Rham operator provides an important example
illustrating the computability of the formula, yielding the basic
Gauss-Bonnet Theorem (Theorem \ref{BasicGaussBonnet}).

\vspace{1pt}This new theorem is proved by first writing $\mathrm{ind}%
_{b}\left( D_{b}^{E}\right) $ as the invariant index of a $G$-equivariant,
transversally elliptic operator $\mathcal{D}$ on a $G$-manifold $\widehat{W}$
associated to the foliation, where $G$ is a compact Lie group of isometries.
Using our equivariant index theorem in \cite{BKR}, we obtain an expression
for this index in terms of the geometry and topology of $\widehat{W}$ and
then rewrite this formula in terms of the original data on the foliation.

We note that a recent paper of Gorokhovsky and Lott addresses this
transverse index question on Riemannian foliations in a very special case.
Using a different technique, they %are able to 
prove a formula for the %basic 
index of a basic Dirac operator that is distinct from our formula, in the
case where all the infinitesimal holonomy groups of the foliation are
connected tori and if Molino's commuting sheaf is abelian and has trivial
holonomy (see \cite{GLott}). Our result requires at most mild topological
assumptions on the transverse structure of the strata of the Riemannian
foliation. In particular, the Gauss-Bonnet Theorem for Riemannian foliations
(Theorem \ref{BasicGaussBonnet}) is a corollary and requires no assumptions
on the structure of the Riemannian foliation.

The paper is organized as follows. The definitions of the basic sections,
holonomy-equivariant vector bundles, basic Clifford bundles, and basic
Dirac-type operators are given in Section \ref{PreliminariesSection}. In
Section \ref{EquivariantTheorySection}, we describe the Fredholm properties
of the basic index and show how to construct the $G$-manifold $\widehat{W}$
and the $G$-equivariant operator $\mathcal{D}$, using a generalization of
Molino theory \cite{Mo}. We also use our construction to obtain asymptotic
expansions and eigenvalue asymptotics of transversally elliptic operators on
Riemannian foliations in Section \ref{AsymptoticsSection}, which is of
independent interest. In Section \ref{canonIsotropyBundlesSection}, we
construct bundles associated to representions of the isotropy subgroups of
the $G$-action; these bundles are used in the main theorem. In Section \ref%
{desingularizationSection}, we describe a method of cutting out tubular
neighborhoods of the singular strata of the foliation and doubling the
remainder to produce a Riemannian foliation with fewer strata. We also
deform the operator and metric and determine the effect of this
desingularization operation on the basic index. We recall the equivariant
index theorem in \cite{BKR} in Section \ref{equivIndexSection} and prove the
basic index theorem in Section \ref{basicIndexThmSection}. Finally, we prove
a generalization of this theorem to representation-valued basic indices in
Section \ref{representationValuedSection}.

We illustrate the theorem with a collection of examples. These include
foliations by suspension (Section \ref{suspensionSection}), 
%{representationValuedSection})
a transverse signature (Section \ref{transverseSignSection}), 
%\ref{suspensionSection}) 
and the basic Gauss-Bonnet Theorem (Section \ref{euler}).

One known application of our theorem is Kawasaki's Orbifold Index Theorem (%
\cite{Kawas1}, \cite{Kawas2}). It is known that every orbifold is the leaf
space of a Riemannian foliation, where the leaves are orbits of an
orthogonal group action such that all isotropy subgroups have the same
dimension. In particular, the contributions from the eta invariants in our
transverse signature example (Section \ref{transverseSignSection}) agree
exactly with the contributions from the singular orbifold strata when the
orbifold is four-dimensional.

We thank James Glazebrook, Efton Park and Igor Prokhorenkov for helpful
discussions. The authors would like to thank variously the Mathematisches
Forschungsinstitut Oberwolfach, the Erwin Schr\"{o}dinger International
Institute for Mathematical Physics (ESI), Vienna, the Department for
Mathematical Sciences (IMF) at Aarhus University, the Centre de Recerca Matem%
\`{a}tica (CRM), Barcelona, and the Department of Mathematics at TCU for
hospitality and support during the preparation of this work.

\section{Riemannian foliations and basic Dirac operators\label%
{PreliminariesSection}}

\subsection{Basic definitions\label{basicDefinitionsSection}}

\vspace{1pt}A foliation of codimension $q$ on a smooth manifold $M$ of
dimension $n$ is a natural generalization of a submersion. Any submersion $%
f:M\rightarrow N$ with fiber dimension $p$ induces locally, on an open set $%
U\subset M$, a diffeomorphism $\phi :U\rightarrow \mathbb{R}^{q}\times 
\mathbb{R}^{p}\ni \left( y,x\right) $, where $p+q=n$. A \textbf{foliation} $%
\mathcal{F}$ is a (maximal) atlas $\left\{ \phi _{\alpha }:U_{\alpha
}\rightarrow \mathbb{R}^{q}\times \mathbb{R}^{p}\right\} $ of $M$ such that
the transition functions $\phi _{\alpha }\circ \phi _{\beta }^{-1}:\mathbb{R}%
^{q}\times \mathbb{R}^{p}\rightarrow \mathbb{R}^{q}\times \mathbb{R}^{p}$
preserve the fibers, i.e. they have the form%
\begin{equation*}
\phi _{\alpha }\circ \phi _{\beta }^{-1}\left( y,x\right) =\left( \tau
_{\alpha \beta }\left( y\right) ,\psi _{\alpha \beta }\left( x,y\right)
\right) .
\end{equation*}%
This local description has many equivalent formulations, as expressed in the
famous Frobenius Theorem. Geometrically speaking, $M$ is partitioned into $p$%
-dimensional immersed submanifolds called the \textbf{leaves} of the
foliation; the tangent bundle $T\mathcal{F}$ to the leaves forms an
integrable subbundle of the tangent bundle $TM$.

In the case of a submersion, the normal bundle to $T\mathcal{F}$ is
naturally identified with the tangent bundle of the base, which then forms
the space of leaves. In general, such a description is not possible, since
the space of leaves defined by the obvious equivalence relation does not
form a manifold. Nevertheless, reasonable \textbf{transverse geometry} can
be expressed in terms of the normal bundle $Q:=TM\diagup T\mathcal{F}$ of
the foliation. We are particularly interested in the case of a Riemannian
foliation, which generalizes the concept of a Riemannian submersion. That
is, the horizontal metric $g_{h}$ on the total space of a Riemannian
submersion is the pullback of the metric on the base, such that in any chart 
$\phi $ as above, $g_{h}\left( \frac{\partial }{\partial y_{i}},\frac{%
\partial }{\partial y_{i}}\right) $ depends on the base coordinates $y$
alone. Another way to express this is that $\mathcal{L}_{X}g_{h}=0$ for all
vertical vector fields $X$, where $\mathcal{L}_{X}$ denotes the Lie
derivative. In the case of a foliation, the normal bundle $Q$ is framed by $%
\left\{ \frac{\partial }{\partial y_{j}}\right\} _{j=1}^{q}$, and this
foliation is called \textbf{Riemannian} if it is equipped with a metric $%
g_{Q}$ on $Q$ such that $\mathcal{L}_{X}g_{Q}=0$ for all $X\in C^{\infty
}\left( M,T\mathcal{F}\right) $ (see \cite{Mo}, \cite{T}). For example, a
Riemannian foliation with all leaves compact is a (generalized) Seifert
fibration; in this case the leaf space is an orbifold (\cite[Section 3.6]{Mo}%
). Or if a Lie group of isometries of a Riemannian manifold has orbits of
constant dimension, then the orbits form a Riemannian foliation. A large
class of examples of Riemannian foliations is produced by suspension (see
Section \ref{suspensionSection}).

Consider the exact sequence of vector bundles%
\begin{equation*}
0\rightarrow T\mathcal{F}\rightarrow TM\overset{\pi }{\rightarrow }%
Q\rightarrow 0.
\end{equation*}%
The \textbf{Bott connection} $\nabla ^{Q}$ on the normal bundle $Q$ is
defined as follows. If $s\in C^{\infty }\left( Q\right) $ and if $\pi \left(
Y\right) =s$, then $\nabla _{X}^{Q}s=\pi \left( \left[ X,Y\right] \right) $.
The basic sections of $Q$ are represented by \textbf{basic vector fields},
fields whose flows preserve the foliation. Alternately, a section $V$ of $Q$
is called a basic vector field if for every $X\in C^{\infty }\left( T%
\mathcal{F}\right) $, $\left[ X,V\right] \in C^{\infty }\left( T\mathcal{F}%
\right) $ (see \cite{KT2} or \cite{Mo}).

A differential form $\omega $ on $M$ is \textbf{basic} if locally it is a
pullback of a form on the base. Equivalently, $\omega $ is basic if for
every vector field $X$ tangent to the leaves, $i_{X}\omega =0$ and $%
i_{X}(d\omega )=0$, where $i_{X}$ denotes interior product with $X$. If we
extend the Bott connection to a connection $\nabla ^{\Lambda ^{\ast }Q^{\ast
}}$ on $\Lambda ^{\ast }Q^{\ast }$, a section $\omega $ of $\Lambda ^{\ast
}Q^{\ast }$ is basic if and only if $\nabla _{X}^{\Lambda ^{\ast }Q^{\ast
}}\omega =0$ for all $X$ tangent to $\mathcal{F}$. The exterior derivative
of a basic form is again basic, so the basic forms are a subcomplex $\Omega
^{\ast }\left( M,\mathcal{F}\right) $ of the de Rham complex $\Omega ^{\ast
}\left( M\right) $. The cohomology of this subcomplex is the \textbf{basic
cohomology} $H^{\ast }\left( M,\mathcal{F}\right) $.

\subsection{Foliated vector bundles}

\label{BasicConnectionsSection}We now review some standard definitions (see 
\cite{KT2} and \cite{Mo}). Let $G$ be a compact Lie group. With notation as
above, we say that a principal $G$--bundle $P\rightarrow \left( M,\mathcal{F}%
\right) $ is a \textbf{foliated principal bundle} if it is equipped with a
foliation $\mathcal{F}_{P}$ (the \textbf{lifted foliation}) such that the
distribution $T\mathcal{F}_{P}$ is invariant under the right action of $G$,
is transversal to the tangent space to the fiber, and projects to $T\mathcal{%
F}$. A connection $\omega $ on $P$ is called \textbf{adapted to }$\mathcal{F}%
_{P}$ if the associated horizontal distribution contains $T\mathcal{F}_{P}$.
An adapted connection $\omega $ is called a \textbf{basic connection} if it
is basic as a $\mathfrak{g}$-valued form on $\left( P,\mathcal{F}_{P}\right) 
$. Note that in \cite{KT2} the authors showed that basic connections always
exist on a foliated principal bundle over a Riemannian foliation.

Similarly, a vector bundle $E\rightarrow \left( M,\mathcal{F}\right) $ is 
\textbf{foliated} if $E$ is associated to a foliated principal bundle $%
P\rightarrow \left( M,\mathcal{F}\right) $ via a representation $\rho $ from 
$G$ to $O\left( k\right) $ or $U\left( k\right) $. Let $\Omega \left(
M,E\right) $ denote the space of forms on $M$ with coefficients in $E$. If a
connection form $\omega $ on $P$ is adapted, then we say that an associated
covariant derivative operator $\nabla ^{E}$ on $\Omega \left( M,E\right) $
is \emph{\ }\textbf{adapted} to the foliated bundle. We say that $\nabla
^{E} $ is a \textbf{basic} connection on $E$ if in addition the associated
curvature operator $\left( \nabla ^{E}\right) ^{2}$ satisfies $i_{X}\left(
\nabla ^{E}\right) ^{2}=0$ for every $X\in T\mathcal{F}$, where $i_{X}$
denotes the interior product with $X$. Note that $\nabla ^{E}$ is basic if $%
\omega $ is basic.

Let $C^{\infty }\left( E\right) $ denote the smooth sections of $E$, and let 
$\nabla ^{E}$ denote a basic connection on $E$. We say that a section $%
s:M\rightarrow E$ is a \textbf{basic section}\emph{\ }if and only if $\nabla
_{X}^{E}s=0$ for all $X\in T\mathcal{F}$. Let $\allowbreak C_{b}^{\infty
}\left( E\right) $ denote the space of basic sections of $E$. We will make
use of the fact that we can give $E$ a metric such that $\nabla ^{E}$ is a $%
\allowbreak $ metric basic connection.

The \textbf{holonomy groupoid} $G_{\mathcal{F}}$ of $\left( M,\mathcal{F}%
\right) $ (see \cite{W}) is the set of ordered triples $\left( x,y,\left[
\gamma \right] \right) $, where $x$ and $y$ are points of a leaf $L$ and $%
\left[ \gamma \right] $ is an equivalence class of piecewise smooth paths in 
$L$ starting at $x$ and ending at $y$; two such paths $\alpha $ and $\beta $
are equivalent if and only if $\beta ^{-1}\alpha $ has trivial holonomy.
Multiplication is defined by $\left( y,z,\left[ \alpha \right] \right) \cdot
\left( x,y,\left[ \beta \right] \right) =\left( x,z,\left[ \alpha \beta %
\right] \right) $, where $\alpha \beta $ refers to the curve starting at $x$
and ending at $z$ that is the concatenation of $\beta $ and $\alpha $.
Because $\left( M,\mathcal{F}\right) $ is Riemannian, $G_{\mathcal{F}}$ is
endowed with the structure of a smooth $\left( n+p\right) $--dimensional
manifold (see \cite{W}), where $n$ is the dimension of $M$ and $p$ is the
dimension of the foliation.

We say that a vector bundle $E\rightarrow M$ is $G_{\mathcal{F}}$\textbf{%
--equivariant }if there is an action of the holonomy groupoid on the fibers.
Explicitly, if the action of $g=\left( x,y,\left[ \gamma \right] \right) $
is denoted by $T_{g}$, then $T_{g}:E_{x}\rightarrow E_{y}$ is a linear
transformation. The transformations $\left\{ T_{g}\right\} $ satisfy $%
T_{g}T_{h}=T_{g\cdot h}$ for every $g,h\in G_{\mathcal{F}}$ for which $%
g\cdot h$ is defined, and we require that the map $g\longmapsto T_{g}$ is
smooth. In addition, we require that for any unit $u=\left( x,x,\left[
\alpha \right] \right) $ (that is, such that the holonomy of $\alpha $ is
trivial), $T_{u}:E_{x}\rightarrow E_{x}$ is the identity.

We say that a section $s:M\rightarrow E$ is \textbf{holonomy--invariant} if
for every $g=\left( x,y,\left[ \gamma \right] \right) \in G_{\mathcal{F}}$, $%
T_{g}s\left( x\right) =s\left( y\right) $.

\begin{remark}
Every $G_{\mathcal{F}}$--equivariant vector bundle $E\rightarrow \left( M,%
\mathcal{F}\right) $ is a foliated vector bundle, because the action of the
holonomy groupoid corresponds exactly to parallel translation along the
leaves. If the partial connection is extended to a basic connection on $E$,
we see that the notions of basic sections and holonomy--invariant sections
are the same.

On the other hand, suppose that $E\rightarrow \left( M,\mathcal{F}\right) $
is a foliated vector bundle that is equipped with a basic connection. It is
not necessarily true that parallel translation can be used to give $E$ the
structure of a $G_{\mathcal{F}}$--equivariant vector bundle. For example,
let $\alpha $ be an irrational multiple of $2\pi $, and consider $E=\left[
0,2\pi \right] \times \left[ 0,2\pi \right] \times \mathbb{C}\diagup \left(
0,\theta ,z\right) \sim \left( 2\pi ,\theta ,e^{i\alpha }z\right) $, which
is a Hermitian line bundle over the torus $S^{1}\times S^{1}$, using the
obvious product metric. The natural flat connection for $E$ over the torus
is a basic connection for the product foliation $\mathcal{F}=\left\{
L_{\theta }\right\} $, where $L_{\theta }=\left\{ \left( \phi ,\theta
\right) \,|\,\phi \in S^{1}\right\} $. However, one can check that parallel
translation cannot be used to make a well-defined action of $G_{\mathcal{F}}$
on the fibers.
\end{remark}

An example of a $G_{\mathcal{F}}$--equivariant vector bundle is the normal
bundle $Q$, given by the exact sequence of vector bundles%
\begin{equation*}
0\rightarrow T\mathcal{F}\rightarrow TM\overset{\pi }{\rightarrow }%
Q\rightarrow 0.
\end{equation*}%
The Bott connection $\nabla ^{Q}$ on $Q$ is a $\allowbreak $metric basic
connection. (Recall that if $s\in C^{\infty }\left( Q\right) $ and if $\pi
\left( Y\right) =s$, then $\nabla _{X}^{Q}s=\pi \left( \left[ X,Y\right]
\right) $.) The basic sections of $Q$ are represented by \textbf{basic
vector fields}, fields whose flows preserve the foliation. Alternately, a
section $V$ of $Q$ is called a basic vector field if for every $X\in
C^{\infty }\left( T\mathcal{F}\right) $, $\left[ X,V\right] \in C^{\infty
}\left( T\mathcal{F}\right) $ (see \cite{KT2} or \cite{Mo}).

\begin{lemma}
\label{connectionsection}Let $E\rightarrow \left( M,\mathcal{F}\right) $ be
a foliated vector bundle with a basic connection $\nabla ^{E}$. Let $V\in
C^{\infty }\left( Q\right) $ be a basic vector field, and let $%
s:M\rightarrow E$ be a basic section. Then $\nabla _{V}^{E}s$ is a basic
section of $E$.
\end{lemma}

\begin{proof}
For any $X\in \allowbreak C^{\infty }\left( T\mathcal{F}\right) $, $\left[
X,V\right] \in \allowbreak C^{\infty }\left( T\mathcal{F}\right) $, so that $%
\nabla _{X}^{E}s=\nabla _{\left[ X,V\right] }^{E}s=0$. Thus, 
\begin{eqnarray*}
\nabla _{X}^{E}\nabla _{V}^{E}s &=&\left( \nabla _{X}^{E}\nabla
_{V}^{E}-\nabla _{V}^{E}\nabla _{X}^{E}-\nabla _{\left[ X,V\right]
}^{E}\right) s \\
&=&\allowbreak \left( \nabla ^{E}\right) ^{2}\left( X,V\right) s=0,
\end{eqnarray*}%
since $\nabla ^{E}$ is basic.
\end{proof}

Another example of a foliated vector bundle is the exterior bundle $%
\bigwedge Q^{\ast }$; the induced connection from the Bott connection on $Q$
is a $\allowbreak $metric basic connection. The set of basic sections of
this vector bundle is the set of basic forms $\Omega \left( M,\mathcal{F}%
\right) $, which is defined in the ordinary way in Section \ref%
{basicDefinitionsSection}. It is routine to check that these two definitions
of basic forms are equivalent.

\subsection{Basic Clifford bundles\label{BasicCliffordBundlesSection}}

Identifying $Q\ $with the normal bundle of the Riemannian foliation $\left(
M,\mathcal{F}\right) $, we form the bundle of Clifford algebras $\allowbreak 
\mathbb{C}\mathrm{l}\left( Q\right) =\mathrm{Cl}\left( Q\right) \otimes 
\mathbb{C}$ over $M$.

\begin{definition}
Let $E$ be a bundle of $\allowbreak \mathbb{C}\mathrm{l}\left( Q\right) $
--modules over a Riemannian foliation $\left( M,\mathcal{F}\right) $. Let $%
\nabla $ denote the Levi--Civita connection on $M,$ which restricts to a $%
\allowbreak $metric basic connection on $Q$. Let $h=\left( \cdot ,\cdot
\right) $ be a Hermitian metric on $E$, and let $\nabla ^{E}$ be a
connection on $E$. Let the action of an element $\xi \in \mathbb{C}\mathrm{l}%
\left( Q_{x}\right) $ on $v\in E_{x}$ be denoted by $c\left( \xi \right) v$.
We say that $\left( E,h,\nabla ^{E}\right) $ is a \textbf{basic Clifford
bundle} if

\begin{enumerate}
\item The bundle $E\rightarrow \left( M,\mathcal{F}\right) $ is foliated.

\item The connection $\nabla ^{E}$ is a $\allowbreak $metric basic
connection.

\item For every $\xi \in Q_{x}$, $c\left( \xi \right) $ is skew-adjoint on $%
E_{x}$.

\item For every $X\in C^{\infty }\left( TM\right) ,Y\in C^{\infty }\left(
Q\right) ,$ and $s\in C^{\infty }\left( E\right) ,$ 
\begin{equation*}
\nabla _{X}^{E}\left( c\left( Y\right) s\right) =c\left( \nabla _{X}Y\right)
s+c\left( Y\right) \nabla _{X}^{E}\left( s\right) .
\end{equation*}
\end{enumerate}
\end{definition}

\begin{lemma}
Let $\left( E,h,\nabla ^{E}\right) $ be a basic Clifford module over $\left(
M,\mathcal{F}\right) $. Let $V\in C^{\infty }\left( Q\right) $ be a basic
vector field, and let $s:M\rightarrow E$ be a basic section. Then $c\left(
V\right) s$ is a basic section of $E$.
\end{lemma}

\begin{proof}
If $\nabla _{X}^{E}s=0$ and $\nabla _{X}V=0$ for every $X\in \allowbreak
C^{\infty }\left( T\mathcal{F}\right) $, then 
\begin{equation*}
\nabla _{X}^{E}\left( c\left( V\right) s\right) =c\left( \nabla _{X}V\right)
\ s+c\left( V\right) \nabla _{X}^{E}\left( s\right) =0.
\end{equation*}
\end{proof}

\subsection{Basic Dirac operators\label{BasicDirac-TypeOperatorsSection}}

\begin{definition}
Let $\left( E,\left( \cdot ,\cdot \right) ,\nabla ^{E}\right) $ be a basic
Clifford bundle. The \textbf{transversal Dirac operator }$D_{\mathrm{tr}%
}^{E} $ is the composition of the maps 
\begin{equation*}
C^{\infty }\left( E\right) \overset{\left( \nabla ^{E}\right) ^{\mathrm{%
\mathrm{tr}}}}{\rightarrow }C^{\infty }\left( Q^{\ast }\otimes E\right) 
\overset{\cong }{\rightarrow }C^{\infty }\left( Q\otimes E\right) \overset{c}%
{\rightarrow }C^{\infty }\left( E\right) ,
\end{equation*}%
where the operator $\left( \nabla ^{E}\right) ^{\mathrm{tr}}$ is the obvious
projection of $\nabla ^{E}:C^{\infty }\left( E\right) \rightarrow C^{\infty
}\left( T^{\ast }M\otimes E\right) $ and the ismorphism $\cong $ is induced
via the holonomy--invariant metric on $Q$.
\end{definition}

If $\left\{ e_{1},...,e_{q}\right\} $ is an orthonormal basis of $Q$, we
have that 
\begin{equation*}
D_{\mathrm{tr}}^{E}=\sum_{j=1}^{q}c\left( e_{j}\right) \nabla _{e_{j}}^{E}.
\end{equation*}%
Let $p:T^{\ast }M\rightarrow M$ be the projection. The restriction of the
principal symbol $\sigma \left( D_{\mathrm{tr}}^{E}\right) :T^{\ast
}M\rightarrow \mathrm{End}\left( p^{\ast }E\right) $ to $Q^{\ast }$ is
denoted $\sigma ^{\mathrm{tr}}\left( D_{\mathrm{tr}}^{E}\right) $, and it is
given by 
\begin{equation*}
\sigma ^{\mathrm{tr}}\left( D_{\mathrm{tr}}^{E}\right) \left( \xi \right)
=c\left( \xi ^{\#}\right) .
\end{equation*}%
Since this map is invertible for $\xi \in Q^{\ast }\setminus 0$ , we say
that $D_{\mathrm{tr}}^{E}$ is \textbf{transversally elliptic}.

\begin{lemma}
The operator $D_{\mathrm{tr}}^{E}$ restricts to a map on the subspace $%
\allowbreak C_{b}^{\infty }\left( E\right) $.
\end{lemma}

\begin{proof}
Suppose that $s:M\rightarrow E$ is a basic section, so that $\nabla
_{X}^{E}s=0$ for every $X\in C^{\infty }\left( T\mathcal{F}\right) $. Near a
point $x$ of $M,$ choose an orthonormal frame field $\left(
e_{1},...,e_{q}\right) $ of $Q$ consisting of basic fields. Then 
\begin{eqnarray*}
\nabla _{X}^{E}\left( D_{\mathrm{tr}}^{E}\left( s\right) \right)
&=&\sum_{j=1}^{q}\nabla _{X}^{E}\left( c\left( e_{j}\right) \nabla
_{e_{j}}^{E}s\right) \\
&=&\sum_{j=1}^{q}c\left( e_{j}\right) \nabla _{X}^{E}\left( \nabla
_{e_{j}}^{E}s\right) ,
\end{eqnarray*}
since each $e_{j}$ is basic, and the result is zero by Lemma \ref%
{connectionsection}.
\end{proof}

We now calculate the formal adjoint of $D_{\mathrm{tr}}^{E}$ on $%
C_{b}^{\infty }\left( E\right) $. Letting $\left( s_{1},s_{2}\right) $
denote the pointwise inner product of sections of $E$ and choosing an
orthonormal frame field $\left( e_{1},...,e_{q}\right) $ of $Q$ consisting
of basic fields, we have that 
\begin{eqnarray*}
\left( D_{\mathrm{tr}}^{E}s_{1},s_{2}\right) -\left( s_{1},D_{\mathrm{tr}%
}^{E}s_{2}\right) &=&\sum_{j=1}^{q}\left( c\left( e_{j}\right) \nabla
_{e_{j}}^{E}s_{1},s_{2}\right) -\left( s_{1},c\left( e_{j}\right) \nabla
_{e_{j}}^{E}s_{2}\right) \\
&=&\sum_{j=1}^{q}\left( c\left( e_{j}\right) \nabla
_{e_{j}}^{E}s_{1},s_{2}\right) +\left( c\left( e_{j}\right) s_{1},\nabla
_{e_{j}}^{E}s_{2}\right) \\
&=&\sum_{j=1}^{q}\left( \nabla _{e_{j}}^{E}\left( c\left( e_{j}\right)
s_{1}\right) ,s_{2}\right) -\left( c\left( \nabla _{e_{j}}^{\bot
}e_{j}\right) s_{1},s_{2}\right) +\left( c\left( e_{j}\right) s_{1},\nabla
_{e_{j}}^{E}s_{2}\right) \\
&=&\left( \sum_{j=1}^{q}\nabla _{e_{j}}^{\bot }\left( c\left( e_{j}\right)
s_{1},s_{2}\right) \right) -\left( c\left( \sum_{j=1}^{q}\nabla
_{e_{j}}^{\bot }e_{j}\right) s_{1},s_{2}\right) \\
&=&-\sum_{j=1}^{q}\nabla _{e_{j}}^{\bot }i_{e_{j}}\omega +\omega \left(
\sum_{j=1}^{q}\nabla _{e_{j}}^{\bot }e_{j}\right) ,
\end{eqnarray*}%
where $\omega $ is the basic form defined by $\omega \left( X\right)
=-\left( c\left( X\right) s_{1},s_{2}\right) $ for $X\in C^{\infty }\left(
Q\right) $. Continuing, 
\begin{eqnarray*}
\left( D_{\mathrm{tr}}^{E}s_{1},s_{2}\right) -\left( s_{1},D_{\mathrm{tr}%
}^{E}s_{2}\right) &=&-\sum_{j=1}^{q}\nabla _{e_{j}}^{\bot }i_{e_{j}}\omega
+\omega \left( \sum_{j=1}^{q}\nabla _{e_{j}}^{\bot }e_{j}\right) \\
&=&-\sum_{j=1}^{q}\left( i_{e_{j}}\nabla _{e_{j}}^{\bot }+i_{\nabla
_{e_{j}}^{\bot }e_{j}}\right) \omega +\omega \left( \sum_{j=1}^{q}\nabla
_{e_{j}}^{\bot }e_{j}\right) \\
&=&-\sum_{j=1}^{q}i_{e_{j}}\nabla _{e_{j}}^{\bot }\omega
\end{eqnarray*}

Note we have been using the normal Levi-Civita connection $\nabla ^{\bot }$.
If we (locally) complete the normal frame field to an orthonormal frame
field $\left\{ e_{1},...,e_{n}\right\} $ for $TM$ near $x\in M$. Letting $%
\nabla ^{M}=$ $\nabla ^{\bot }+\nabla ^{\mathrm{\tan }}$ be the Levi-Civita
connection on $\Omega \left( M\right) $, the divergence of a general basic
one-form $\beta $ is 
\begin{eqnarray*}
\delta \beta &=&-\sum_{j=1}^{n}i_{e_{j}}\nabla _{e_{j}}^{M}\beta \\
&=&-\sum_{j=1}^{n}i_{e_{j}}\nabla _{e_{j}}^{\bot }\beta
+-\sum_{j=1}^{n}i_{e_{j}}\nabla _{e_{j}}^{\mathrm{\tan }}\beta \\
&=&-\sum_{j=1}^{q}i_{e_{j}}\nabla _{e_{j}}^{\bot }\beta
+-\sum_{j>q}^{n}i_{e_{j}}\nabla _{e_{j}}^{\mathrm{\tan }}\beta
\end{eqnarray*}%
Letting $\beta =\sum_{k=1}^{q}\beta _{k}e_{k}^{\ast },$ then each $\beta
_{k} $ is basic and 
\begin{eqnarray*}
\delta \beta &=&-\sum_{j=1}^{q}i_{e_{j}}\nabla _{e_{j}}^{\bot }\beta
-\sum_{j>q}^{n}i_{e_{j}}\nabla _{e_{j}}^{\mathrm{\tan }}\left(
\sum_{k=1}^{q}\beta _{k}e_{k}^{\ast }\right) \\
&=&-\sum_{j=1}^{q}i_{e_{j}}\nabla _{e_{j}}^{\bot }\beta
-\sum_{k=1}^{q}\sum_{j>q}^{n}\beta _{k}i_{e_{j}}\nabla _{e_{j}}^{\mathrm{%
\tan }}\left( e_{k}^{\ast }\right) \\
&=&-\sum_{j=1}^{q}i_{e_{j}}\nabla _{e_{j}}^{\bot }\beta
-\sum_{k=1}^{q}\sum_{j>q}\beta _{k}i_{e_{j}}\left( \sum_{m>q}\left( \nabla
_{e_{j}}^{M}\left( e_{k}^{\ast }\right) ,e_{m}^{\ast }\right) e_{m}^{\ast
}\right) \\
&=&-\sum_{j=1}^{q}i_{e_{j}}\nabla _{e_{j}}^{\bot }\beta
+\sum_{k=1}^{q}\sum_{j>q}\beta _{k}i_{e_{j}}\left( \sum_{m>q}\left( \nabla
_{e_{j}}^{M}\left( e_{m}^{\ast }\right) ,e_{k}^{\ast }\right) e_{m}^{\ast
}\right) \\
&=&-\sum_{j=1}^{q}i_{e_{j}}\nabla _{e_{j}}^{\bot }\beta
+\sum_{k=1}^{q}\sum_{j>q}\beta _{k}\left( \sum_{m>q}\left( \nabla
_{e_{j}}^{M}\left( e_{j}^{\ast }\right) ,e_{k}^{\ast }\right) \right) \\
&=&-\sum_{j=1}^{q}i_{e_{j}}\nabla _{e_{j}}^{\bot }\beta +i_{H}\beta ,
\end{eqnarray*}%
where $H$ is the mean curvature vector field of the foliation. Thus, for
every basic one-form $\beta $, 
\begin{equation*}
-\sum_{j=1}^{q}i_{e_{j}}\nabla _{e_{j}}^{\bot }\beta =\delta \beta
-i_{H}\beta .
\end{equation*}%
Applying this result to the form $\omega $ defined above, we have 
\begin{eqnarray*}
\left( D_{\mathrm{tr}}^{E}s_{1},s_{2}\right) -\left( s_{1},D_{\mathrm{tr}%
}^{E}s_{2}\right) &=&-\sum_{j=1}^{q}i_{e_{j}}\nabla _{e_{j}}^{\bot }\omega \\
&=&\delta \omega -i_{H}\omega \\
&=&\delta \omega +\left( c\left( H\right) s_{1},s_{2}\right) \\
&=&\delta \omega -\left( s_{1},c\left( H\right) s_{2}\right)
\end{eqnarray*}%
Next, letting $P:L^{2}\left( \Omega \left( M\right) \right) \rightarrow
L^{2}\left( \Omega _{b}\left( M,\mathcal{F}\right) \right) $ denote the
orthogonal projection onto the closure of basic forms in $L^{2}\left( \Omega
\left( M\right) \right) $, we observe that $\delta _{b}=P\delta $ is the
adjoint of $d_{b}$, the restriction of the exterior derivative to basic
forms. Using the results of \cite{PaRi}, $P$ maps smooth forms to smooth
basic forms, and the projection of the smooth function $\left( s_{1},c\left(
H\right) s_{2}\right) $ is simply $\left( s_{1},c\left( H_{b}\right)
s_{2}\right) ,$ where $H_{b}$ is the vector field $P\left( H^{\flat }\right)
^{\sharp }$, the basic projection of the mean curvature vector field. If we
had originally chosen our bundle-like metric to have basic mean curvature,
which is always possible by \cite{D}, then $H_{b}=H$. In any case, the right
hand side of the formula above is a basic function, so that 
\begin{equation*}
\left( D_{\mathrm{tr}}^{E}s_{1},s_{2}\right) -\left( s_{1},D_{\mathrm{tr}%
}^{E}s_{2}\right) =\delta _{b}\omega -\left( s_{1},c\left( H_{b}\right)
s_{2}\right) .
\end{equation*}%
We conclude:

\begin{proposition}
The formal adjoint of the transversal Dirac operator is $\left( D_{\mathrm{tr%
}}^{E}\right) ^{\ast }=D_{\mathrm{tr}}^{E}-c\left( H_{b}\right) $.
\end{proposition}

\begin{definition}
The \textbf{basic Dirac operator} associated to a basic Clifford module $%
\left( E,\left( \cdot ,\cdot \right) ,\nabla ^{E}\right) $ over a Riemannian
foliation $\left( M,\mathcal{F}\right) $ with bundle-like metric is%
\begin{equation*}
D_{b}^{E}=D_{\mathrm{tr}}^{E}-\frac{1}{2}c\left( H_{b}\right) :C_{b}^{\infty
}\left( E\right) \rightarrow C_{b}^{\infty }\left( E\right) .
\end{equation*}
\end{definition}

\begin{remark}
Note that the formal adjoint of $D_{b}^{E}$ is $\left( D_{\mathrm{tr}%
}^{E}\right) ^{\ast }+\frac{1}{2}c\left( H_{b}\right) =D_{\mathrm{tr}}^{E}-%
\frac{1}{2}c\left( H_{b}\right) =D_{b}^{E}$. Thus, $D_{b}^{E}$ is formally
sel-adjoint. In \cite{HabRi}, the researchers showed that the eigenvalues of 
$D_{b}^{E}$ are independent of the choice of the bundle-like metric that
restricts to the given transverse metric of the Riemannian foliation.
\end{remark}

Let $\left( E,\left( \cdot ,\cdot \right) ,\nabla ^{E}\right) $ be a basic
Clifford bundle over the Riemannian foliation $\left( M,\mathcal{F}\right) $
, and let $D_{b}^{E}:\allowbreak C_{b}^{\infty }\left( E\right) \rightarrow
\allowbreak C_{b}^{\infty }\left( E\right) $ be the associated basic Dirac
operator. Assume that $E=E^{+}\oplus E^{-}$, with $D_{b}^{E\pm }:\allowbreak
C_{b}^{\infty }\left( E^{\pm }\right) \rightarrow \allowbreak C_{b}^{\infty
}\left( E^{\mp }\right) $. Let $\left( D_{b}^{E\pm }\right) ^{\ast
}:L_{b}^{2}\left( E^{\mp }\right) \rightarrow L_{b}^{2}\left( E^{\pm
}\right) $ denote the adjoint of $D_{b}^{E\pm }$.

\begin{definition}
The \textbf{analytic basic index }of $D_{b}^{E}$ is 
\begin{equation*}
\mathrm{ind}_{b}\left( D_{b}^{E}\right) =\dim \ker \left.
D_{b}^{E+}\right\vert _{L^{2}\left( \allowbreak C_{b}^{\infty }\left(
E^{+}\right) \right) }-\dim \ker \left. \left( D_{b}^{E-}\right) ^{\ast
}\right\vert _{L^{2}\left( \allowbreak C_{b}^{\infty }\left( E^{-}\right)
\right) }.
\end{equation*}
\end{definition}

\begin{remark}
\label{finitedim} At this point, it is not clear that these dimensions are
finite. We demonstrate this fact inside this section.
\end{remark}

\subsection{Examples\label{diracex}}

The standard examples of ordinary Dirac operators are the spin$^{c}$ Dirac
operator, the de Rham operator, the signature operator, and the Dolbeault
operator. Transversally elliptic analogues of these operators and their
corresponding basic indices are typical examples of basic Dirac operators.

Suppose that the normal bundle $Q=TM\diagup T\mathcal{F}\rightarrow M$ of
the Riemannian foliation $\left( M,\mathcal{F}\right) $ is spin$^{c}$. Then
there exists a foliated Hermitian basic Clifford bundle $\left( S,\left(
\cdot ,\cdot \right) ,\nabla ^{S}\right) $ over $M$ such that for all $x\in
M $, $S_{x}$ is isomorphic to the standard spinor representation of the
Clifford algebra $\mathbb{C}\mathrm{l}\left( Q_{x}\right) $ (see \cite{LM}).
The associated basic Dirac operator $\not{\partial}_{b}^{S}$ is called a 
\textbf{basic spin}$^{c}$\textbf{\ Dirac operator}. The meaning of the
integer $\mathrm{ind}_{b}\left( \not{\partial}_{b}^{S}\right) $ is not
clear, but it is an obstruction to some transverse curvature and other
geometric conditions (see \cite{GlK}, \cite{Ju}, \cite{Kord}, \cite{HabRi}).

Suppose $\mathcal{F}$ has codimension $q$. The\textbf{\ basic Euler
characteristic} is defined as 
\begin{equation*}
\chi \left( M,\mathcal{F}\right) =\sum_{k=0}^{q}(-1)^{k}\dim H^{k}\left( M,%
\mathcal{F}\right) ,
\end{equation*}%
provided that all of the basic cohomology groups $H^{k}\left( M,\mathcal{F}%
\right) $ are finite-dimensional. Although $H^{0}\left( M,\mathcal{F}\right) 
$ and $H^{1}\left( M,\mathcal{F}\right) $ are always finite-dimensional,
there are foliations for which higher basic cohomology groups can be
infinite-dimensional. For example, in \cite{Gh}, the author gives an example
of a flow on a 3-manifold for which $H^{2}\left( M,\mathcal{F}\right) $ is
infinite-dimensional. There are various proofs that the basic cohomology of
a Riemannian foliation on a closed manifold is finite-dimensional; see for
example \cite{EKHS} for the original proof using spectral sequence
techniques or \cite{KT} and \cite{PaRi} for proofs using a basic version of
the Hodge theorem.

It is possible to express the basic Euler characteristic as the index of an
operator. Let $d_{b}$ denote the restriction of the exterior derivative $d$
to basic forms over the Riemannian foliation $\left( M,\mathcal{F}\right) $
with bundle-like metric, and let $\delta _{b}$ be the adjoint of $d_{b}$. It
can be shown that $\delta _{b}$ is the restriction of the operator $P\delta $
to basic forms, where $\delta $ is the adjoint of $d$ on all forms and $P$
is the $L^{2}$ -orthogonal projection of the space of forms onto the space
of basic forms. For general foliations, this is not a smooth operator, but
in the case of Riemannian foliations, $P$ maps smooth forms to smooth basic
forms (see \cite{PaRi}), and $P\delta $ is a differential operator. In
perfect analogy to the fact that the index of the de Rham operator 
\begin{equation*}
d+\delta :\Omega ^{\mathrm{\mathrm{even}}}\left( M\right) \rightarrow \Omega
^{\mathrm{\mathrm{odd}}}\left( M\right)
\end{equation*}%
is the ordinary Euler characteristic, it can be shown that the basic index
of the differential operator $d+P\delta $, that is the index of 
\begin{equation*}
D_{b}^{\prime }=d_{b}+\delta _{b}:\Omega _{b}^{\mathrm{\mathrm{even}}}\left(
M,\mathcal{F}\right) \rightarrow \Omega _{b}^{\mathrm{\mathrm{odd}}}\left( M,%
\mathcal{F}\right) ,
\end{equation*}%
is the basic Euler characteristic. The same proof works; this time we must
use the basic version of the Hodge theorem (see \cite{EKHS}, \cite{KT}, and 
\cite{PaRi}). Note that the equality of the basic index remains valid for
nonRiemannian foliations; however, the Fredholm property fails in many
circumstances. It is interesting to note that the operator $d_{b}+\delta
_{b} $ fails to be transversally elliptic in some examples of nonRiemannian
foliations.

The principal symbol of $D_{b}^{\prime }$ is as follows. We define the
Clifford multiplication of $\mathbb{C}\mathrm{l}\left( Q\right) $ on the
bundle $\wedge ^{\ast }Q^{\ast }$ by the action%
\begin{equation*}
v\cdot ~=\left( v^{\flat }\wedge \right) -\left( v\lrcorner \right)
\end{equation*}%
for any vector $v\in N\mathcal{F}\cong Q$. With the standard connection and
inner product defined by the metric on $Q$, the bundle $\wedge ^{\ast
}Q^{\ast }$ is a basic Clifford bundle. The corresponding basic Dirac
operator, called the \textbf{basic de Rham operator} on basic forms,
satisfies 
\begin{equation*}
D_{b}=d+\delta _{b}-\frac{1}{2}\left( \kappa _{b}\wedge +\kappa
_{b}\lrcorner \right) =D_{b}^{\prime }-\frac{1}{2}\left( \kappa _{b}\wedge
+\kappa _{b}\lrcorner \right) .
\end{equation*}%
The kernel of this operator represents the twisted basic cohomology classes,
the cohomology of basic forms induced by the differential $\widetilde{d}$
defined as%
\begin{equation*}
\widetilde{d}=d-\frac{1}{2}\kappa _{b}\wedge ~.
\end{equation*}%
See \cite{HabRi2} for an extended discussion of twisted basic cohomology,
the basic de Rham operator, and its properties. We have $\mathrm{ind}%
_{b}\left( D_{b}^{\prime }\right) =\mathrm{ind}_{b}\left( D_{b}\right) $
because they differ by a zeroth order operator (see the Fredholm properties
of the basic index in Section \ref{BasicIndexAndDiagramSection} below), and
thus%
\begin{equation*}
\mathrm{ind}_{b}\left( D_{b}^{\prime }\right) =\mathrm{ind}_{b}\left(
D_{b}\right) =\chi \left( M,\mathcal{F}\right) ,
\end{equation*}%
the basic Euler characteristic of the complex of basic forms.

\section{Fredholm properties and equivariant theory\label%
{EquivariantTheorySection}}

\subsection{Molino theory and properties of the basic index\label%
{BasicIndexAndDiagramSection}}

Let $\widehat{M}\overset{p}{\longrightarrow }M$ denote the principal bundle
of ordered pairs of frames $\left( \phi _{x},\psi _{x}\right) $ over $x\in M$%
, where $\phi _{x}:\mathbb{R}^{q}\rightarrow N_{x}\mathcal{F}$ is an
isometry and $\psi :\mathbb{C}^{k}\rightarrow E_{x}$ is a complex isometry.
This is a principal $G$--bundle, where $G\cong O\left( q\right) \times
U\left( k\right) $, and it comes equipped with a natural metric connection $%
\nabla $ associated to the Riemannian and Hermitian structures of $%
E\rightarrow M$. The foliated vector bundles $Q\rightarrow \left( M,\mathcal{%
F}\right) $ and $E\rightarrow \left( M,\mathcal{F}\right) $ naturally give $%
\widehat{M}$ the structure of a foliated principal bundle with lifted
foliation $\widehat{\mathcal{F}}$. Transferring the normalized, biinvariant
metric on $G$ to the fibers and using the connection $\nabla $, we define a
natural metric $\left( \cdot ,\cdot \right) _{\widehat{M}}$ on $\widehat{M}$
that is locally a product. The connection $\nabla ^{E}$ pulls back to a
basic connection $\nabla ^{p^{\ast }E}$ on $p^{\ast }E$; the horizontal
subbundle $\mathcal{H}p^{\ast }E$ of $Tp^{\ast }E$ is the inverse image of
the horizontal subbundle $\mathcal{H}E\subset TE$ under the natural map $%
Tp^{\ast }E\rightarrow TE$. It is clear that the metric is bundle-like for
the lifted foliation $\widehat{\mathcal{F}}$.

Observe that the foliation $\widehat{\mathcal{F}}$ is transversally
parallelizable, meaning that the normal bundle of the lifted foliation is
parallelizable by $\widehat{\mathcal{F}}$-basic vector fields. To see this,
we use a modification of the standard construction of the parallelism of the
frame bundle of a manifold (see \cite[p.82]{Mo} for this construction in the
case where the principle bundle is the bundle of transverse orthonormal
frames). Let $G=O\left( q\right) \times U\left( k\right) $, and let $\theta $
denote the $\mathbb{R}^{q}$--valued solder form of $\widehat{M}\rightarrow M$%
. Given the pair of frames $z=\left( \phi ,\psi \right) $ where $\phi :%
\mathbb{R}^{q}\rightarrow N_{p\left( z\right) }\mathcal{F}$ and $\psi :%
\mathbb{C}^{k}\rightarrow E_{p\left( z\right) }$ and given $X_{z}\in T_{z}%
\widehat{M}$, we define $\theta \left( X_{z}\right) =\phi ^{-1}\left( \pi
^{\bot }p_{\ast }X_{z}\right) $, where $\pi ^{\bot }:T_{p\left( z\right)
}M\rightarrow N_{p\left( z\right) }\mathcal{F}$ is the orthogonal
projection. Let $\omega $ denote the $\mathfrak{o}\left( q\right) \oplus 
\mathfrak{u}\left( k\right) $--valued connection one-form. Let $\left\{
e_{1},...,e_{q}\right\} $ be the standard orthonormal basis of $\mathbb{\ R}%
^{q}$ , and let $\left\{ E_{j}\right\} _{j=1}^{\dim G}$ denote a fixed
orthonormal basis of $\mathfrak{o}\left( q\right) \oplus \mathfrak{u}\left(
k\right) $. We uniquely define the vector fields $V_{1},...,V_{q}$, $%
\overline{E_{1}},...,\overline{E_{\dim G}}$ on $\widehat{M}$ by the
conditions

\begin{enumerate}
\item $V_{i}\in N_{z}\widehat{\mathcal{F}},$ $\overline{E_{j}}\in N_{z}%
\widehat{\mathcal{F}}$ for every $i,j$.

\item $\omega \left( V_{i}\right) =0,$ $\omega \left( \overline{E_{j}}%
\right) =E_{j}$ for every $i,j$.

\item $\theta \left( V_{i}\right) =e_{i},$ $\theta \left( \overline{E_{j}}%
\right) =0$ for every $i,j$.
\end{enumerate}

Then the set of $\widehat{\mathcal{F}}$--basic vector fields $\left\{
V_{1},...,V_{q},\overline{E_{1}},...,\overline{E_{\dim G}}\right\} $ is a
transverse parallelism on $\left( \widehat{M},\widehat{\mathcal{F}}\right) $
associated to the connection $\nabla $. By the fact that $\left( \widehat{M},%
\widehat{\mathcal{F}}\right) $ is Riemannian and the structure theorem of
Molino \cite[Chapter 4]{Mo}, the leaf closures of $\left( \widehat{M},%
\widehat{\mathcal{F}}\right) $ are the fibers of a Riemannian submersion $%
\widehat{\pi }:\widehat{M}\rightarrow \widehat{W}$.

Next, we show that the bundle $p^{\ast }E\rightarrow \widehat{M}$ is $G_{%
\widehat{\mathcal{F}}}$--equivariant. An element of the foliation groupoid $%
G_{\widehat{\mathcal{F}}}$ is a triple of the form $\left( y,z,\left[ \cdot %
\right] \right) $, where $y$ and $z$ are points of a leaf of $\widehat{%
\mathcal{F}}$ and $\left[ \cdot \right] $ is the set of all piecewise smooth
curves starting at $y$ and ending at $z$, since all such curves are
equivalent because the holonomy is trivial on $\widehat{M}$. The basic
connection on $E$ induces a $G_{\widehat{\mathcal{F}}}$--action on $p^{\ast
}E$, defined as follows. Given a vector $\left( y,v\right) \in \left(
p^{\ast }E\right) _{y}$ so that $y\in \allowbreak \allowbreak \allowbreak 
\widehat{M}$, $v\in E_{p\left( y\right) }$, we define the action of $%
\widehat{g}=$ $\left( y,z,\left[ \cdot \right] \right) $ by $S_{\widehat{g}%
}\left( y,v\right) =\left( z,P_{\gamma }v\right) $, where $\gamma $ is any
piecewise smooth curve from $p\left( y\right) $ to $p\left( z\right) $ in
the leaf containing $p\left( y\right) $ that lifts to a leafwise curve in $%
\allowbreak \allowbreak \allowbreak \widehat{M}$ from $y$ to $z$ and where $%
P_{\gamma }$ denotes parallel translation in $E$ along the curve $\gamma $.
It is easy to check that this action makes $p^{\ast }E$ into a $G_{\widehat{%
\mathcal{F}}}$--equivariant, foliated vector bundle. The pullback $p^{\ast }$
maps basic sections of $E$ to basic sections of $p^{\ast }E$. Also, the $%
O\left( \allowbreak q\right) \times U\left( k\right) $--action on $\left( 
\widehat{M},{\mathcal{\ }}\widehat{\mathcal{F}}\right) $ induces an action
of $O\left( \allowbreak q\right) \times U\left( k\right) $ on $p^{\ast }E$
that preserves the basic sections.

Observe that if $s\in \allowbreak C_{b}^{\infty }\left( E\right) $, then $%
p^{\ast }s$ is a basic section of $p^{\ast }E$ that is $O\left( \allowbreak
q\right) \times U\left( k\right) $--invariant. Conversely, if $\widehat{s}%
\in C^{\infty }\left( p^{\ast }E\right) $ is $O\left( \allowbreak q\right)
\times U\left( k\right) $--invariant, then $\widehat{s}=p^{\ast }s$ for some 
$s\in \allowbreak C^{\infty }\left( E\right) $. Next, suppose $\widehat{s}%
=p^{\ast }s$ is $O\left( \allowbreak q\right) \times U\left( k\right) $%
--invariant and basic. Given any vector $X\in T_{p\left( y\right) }\mathcal{F%
}$ and its horizontal lift $\widetilde{X}\in T_{y}\widehat{\mathcal{F}}$, we
have 
\begin{equation*}
0=\nabla _{\widetilde{X}}^{p^{\ast }E}\widehat{s}=\nabla _{\widetilde{X}%
}^{p^{\ast }E}p^{\ast }s=p^{\ast }\nabla _{X}^{E}s,
\end{equation*}%
so that $s$ is also basic. We have shown that $\allowbreak C_{b}^{\infty
}\left( E\right) $ is isomorphic to $\allowbreak C_{b}^{\infty }\left( 
\widehat{M},p^{\ast }E\right) ^{O\left( \allowbreak q\right) \times U\left(
k\right) }$.

We now construct a Hermitian vector bundle $\mathcal{E}$ over $\widehat{W}$,
similar to the constructions in \cite{RiTransv} and \cite{EK}. Given $w\in 
\widehat{W}$ and the corresponding leaf closure $\widehat{\pi }^{-1}\left(
w\right) \in \widehat{M}$ , consider a basic section $s\in \allowbreak
C_{b}^{\infty }\left( \widehat{M},p^{\ast }E\right) $ restricted to $%
\widehat{\pi }^{-1}\left( w\right) $. Given any $y\in \widehat{M}$, the
vector $s\left( y\right) $ uniquely determines $s$ on the entire leaf
closure by parallel transport, because the section is smooth. Similarly,
given a vector $v_{y}\in \left( p^{\ast }E\right) _{y}$, there exists a
basic section $s\in \allowbreak C_{b}^{\infty }\left( \widehat{M},p^{\ast
}E\right) $ such that $s\left( y\right) =v_{y}$, because there is no
obstruction to extending, by the following argument. Given a basis $\left\{
b_{1},...,b_{k}\right\} $ of $\mathbb{C}^{k}$, we define the $k$ linearly
independent, basic sections $s_{k}$ of $p^{\ast }E$ by $s_{j}\left( \left(
\phi ,\psi \right) \right) =\psi \left( b_{j}\right) \in \left( p^{\ast
}E\right) _{\left( \phi ,\psi \right) }=E_{p\left( \left( \phi ,\psi \right)
\right) }$ . Thus, given a local frame $\left\{ v_{j}\right\} $ for $p^{\ast
}E$ on a $\widehat{\mathcal{F}}$ --transversal submanifold near $y$, there
is a unique extension of this frame to be a frame consisting of basic
sections on a tubular neighborhood of the leaf closure containing $y$; in
particular a vector may be extended to be a basic section of $p^{\ast }E$.
We now define $\mathcal{E}_{w}=\allowbreak C_{b}^{\infty }\left( \widehat{M}%
,p^{\ast }E\right) \diagup \sim _{w},$ where two basic sections $s$, $%
s^{\prime }:\widehat{M}\rightarrow p^{\ast }E$ are equivalent($s\sim
_{w}s^{\prime }$) if $s\left( y\right) =s^{\prime }\left( y\right) $ for
every $y\in \widehat{\pi }^{-1}\left( w\right) $. By the reasoning above, $%
\mathcal{E}_{w}$ is a complex vector space whose dimension is equal to the
complex rank of $p^{\ast }E\rightarrow \widehat{M}$. Alternately, we could
define $\mathcal{E}_{w}$ to be the vector space of $\widehat{\mathcal{F}}$
-basic sections of $p^{\ast }E$ restricted to the leaf closure $\widehat{\pi 
}^{-1}\left( w\right) $. The union $\cup _{w\in \widehat{W}}$ $\mathcal{E}%
_{w}$ forms a smooth, complex vector bundle $\mathcal{E}$ over $\widehat{W}$%
; local trivializations of $\mathcal{E}$ are given by local, basic framings
of the trivial bundle $p^{\ast }E\rightarrow \widehat{M}$. We remark that in
the constructions of \cite{RiTransv} and \cite{EK}, the vector bundle was
lifted to the transverse orthonormal frame bundle $\widehat{M}$, and in that
case the corresponding bundle $\mathcal{E}$ in those papers could have
smaller rank than $E$.

We let the invertible $\Phi :C^{\infty }\left( \widehat{W},\mathcal{E}%
\right) \rightarrow \allowbreak C_{b}^{\infty }\left( \widehat{M},p^{\ast }E,%
\widehat{\mathcal{F}}\right) $ be the almost tautological map defined as
follows. Given a section $\widehat{s}$ of $\mathcal{E}$, its value at each $%
w\in \widehat{W}$ is an equivalence class $\left[ s\right] _{w}$ of basic
sections. We define for each $y\in \widehat{\pi }^{-1}\left( w\right) $, 
\begin{equation*}
\Phi \left( \widehat{s}\right) \left( y\right) =s\left( y\right) \in \left(
p^{\ast }E\right) _{y}.
\end{equation*}%
By the continuity of the basic section $s$, the above is independent of the
choice of this basic section in the equivalence class. By the definition of $%
\Phi $ and of the trivializations of $\mathcal{E}$, it is clear that $\Phi $
is a smooth map. Also, the $G=O\left( q\right) \times U\left( k\right) $
action on basic sections of $p^{\ast }E$ pushes forward to a $G$ action on
sections of $\mathcal{E}$. We have the following commutative diagram, with $%
W=M\diagup \overline{\mathcal{F}}=\widehat{W}\diagup G$ the leaf closure
space of $\left( M,\mathcal{F}\right) $.

\begin{equation*}
\begin{array}{ccccccc}
p^{\ast }E &  &  &  & \mathcal{E} &  &  \\ 
& \searrow &  &  & \downarrow &  &  \\ 
G & \hookrightarrow & \left( \widehat{M},\widehat{\mathcal{F}}\right) & 
\overset{\widehat{\pi }}{\longrightarrow } & \widehat{W} &  &  \\ 
&  & \downarrow ^{p} & \circlearrowleft & \downarrow \,\, &  &  \\ 
E & \rightarrow & \left( M,\mathcal{F}\right) & \longrightarrow & W &  & 
\end{array}%
\end{equation*}

Observe that we have the necessary data to construct the basic Dirac
operator on sections of $p^{\ast }E$ over $\widehat{M}$ corresponding to the
pullback foliation $p^{\ast }\mathcal{F}$ on $\widehat{M}$. The connection $%
\nabla ^{p^{\ast }E}$ is a basic connection with respect to this Riemannian
foliation, and the normal bundle $N\left( p^{\ast }\mathcal{F}\right) $
projects to the normal bundle $Q=N\mathcal{F}$, so that the action of $%
\allowbreak \mathbb{C}\mathrm{l}\left( Q\right) $ on $E$ lifts to an action
of $\mathbb{C}\mathrm{l}\left( N\left( p^{\ast }\mathcal{F}\right) \right) $
on $p^{\ast }E$. Using this basic Clifford bundle structure, we construct
the transversal Dirac-type operator $D_{\mathrm{\mathrm{tr}},p^{\ast
}}^{p^{\ast }E}$ and the basic Dirac-type operator $D_{b,p^{\ast }}^{p^{\ast
}E}$ on $\allowbreak C_{b}^{\infty }\left( \widehat{M},p^{\ast }E,p^{\ast }%
\mathcal{F}\right) $; we add the subscript $p^{\ast }$ to emphasize that we
are working with $p^{\ast }\mathcal{F}$ rather than the lifted foliation.
Observe that $\allowbreak C_{b}^{\infty }\left( \widehat{M},p^{\ast
}E,p^{\ast }\mathcal{F}\right) =\allowbreak C_{b}^{\infty }\left( \widehat{M}%
,p^{\ast }E,\widehat{\mathcal{F}}\right) ^{G}\subset \allowbreak
C_{b}^{\infty }\left( \widehat{M},p^{\ast }E,\widehat{\mathcal{F}}\right) $.
It is clear from the construction that $p^{\ast }$ is an isomorphism from $%
\allowbreak C_{b}^{\infty }\left( M,E\right) $ to $\allowbreak C_{b}^{\infty
}\left( \widehat{M},p^{\ast }E,p^{\ast }\mathcal{F}\right) $ and $p^{\ast
}\circ D_{b}^{E}=D_{b,p^{\ast }}^{p^{\ast }E}\circ p^{\ast }$.

We define the operator $\mathcal{D}:C^{\infty }\left( \widehat{W},\mathcal{E}%
\right) \rightarrow C^{\infty }\left( \widehat{W},\mathcal{E}\right) $ by 
\begin{equation*}
\mathcal{D}=\Phi ^{-1}\circ \left( D_{\mathrm{\mathrm{tr}},p^{\ast
}}^{p^{\ast }E}-\frac{1}{2}c\left( \widehat{H_{b}}\right) \right) \circ \Phi
,
\end{equation*}%
where $\widehat{H_{b}}$ is the basic mean curvature of the pullback
foliation, which is merely the horizontal lift of $H_{b}$. Let $\mathcal{D}%
^{G}$ denote the restriction of $\mathcal{D}$ to $C^{\infty }\left( \widehat{%
W},\mathcal{E}\right) ^{G}$. Note that 
\begin{equation*}
\Phi :C^{\infty }\left( \widehat{W},\mathcal{E}\right) ^{G}\rightarrow
\allowbreak C_{b}^{\infty }\left( \widehat{M},p^{\ast }E,\widehat{\mathcal{F}%
}\right) ^{G}
\end{equation*}%
is an isomorphism. Observe that the Hermitian metric on $p^{\ast }E$ induces
a well-defined Hermitian metric on $\mathcal{E}$ that is invariant under the
action of $G$.

Assume that $E=E^{+}\oplus E^{-}$ with $D_{b}^{E\pm }:\allowbreak
C_{b}^{\infty }\left( M,E^{\pm }\right) \rightarrow \allowbreak
C_{b}^{\infty }\left( M,E^{\mp }\right) $. We define $D_{b,p^{\ast }}^{\pm }$
to be the restrictions 
\begin{equation*}
\left( D_{\mathrm{\mathrm{tr}},p^{\ast }}^{p^{\ast }E}-\frac{1}{2}c\left( 
\widehat{H_{b}}\right) \right) :C_{b}^{\infty }\left( \widehat{M},p^{\ast
}E^{\pm },p^{\ast }\mathcal{F}\right) \rightarrow \allowbreak C_{b}^{\infty
}\left( \widehat{M},p^{\ast }E^{\mp },p^{\ast }\mathcal{F}\right) ,
\end{equation*}%
We define the bundles $\mathcal{E}^{\pm }$ and the operator 
\begin{equation}
\mathcal{D}^{+}=\Phi ^{-1}\circ \left( D_{\mathrm{\mathrm{tr}},p^{\ast
}}^{p^{\ast }E}-\frac{1}{2}c\left( \widehat{H_{b}}\right) \right) \circ \Phi
:C^{\infty }\left( \widehat{W},\mathcal{E}^{+}\right) \rightarrow C^{\infty
}\left( \widehat{W},\mathcal{E}^{-}\right)  \label{newDFormula}
\end{equation}%
in an analogous way. We now have the following result.

\begin{proposition}
\label{NewD}Let $D_{b}^{E+}:\allowbreak C_{b}^{\infty }\left( M,E^{+}\right)
\rightarrow \allowbreak C_{b}^{\infty }\left( M,E^{-}\right) $ be a basic
Dirac operator for the rank $k$ complex vector bundle $E=E^{+}\oplus E^{-}$
over the transversally oriented Riemannian foliation $\left( M,\mathcal{F}%
\right) ,$ and let $G=O\left( q\right) \times U\left( k\right) $. Then 
\begin{equation*}
\mathrm{ind}_{b}\left( D_{b}^{E+}\right) =\mathrm{ind}\left( \mathcal{D}%
^{G}\right) ,
\end{equation*}%
where $\mathrm{ind}\left( \mathcal{D}^{G}\right) $ refers to the index of
the transversally elliptic operator $\mathcal{D}^{+}$ restricted to $G$
--invariant sections (equivalently, the supertrace of the invariant part of
the virtual representation--valued equivariant index of the operator $%
\mathcal{D}$). It is not necessarily the case that the adjoint $\mathcal{D}%
^{-}:C^{\infty }\left( \widehat{W},\mathcal{E}^{-}\right) \rightarrow
C^{\infty }\left( \widehat{W},\mathcal{E}^{+}\right) $ coincides with $%
\left. \Phi ^{-1}\circ \left( D_{\mathrm{\mathrm{tr}},p^{\ast }}^{p^{\ast
}E}-\frac{1}{2}c\left( \widehat{H_{b}}\right) \right) \circ \Phi \right\vert
_{C^{\infty }\left( \widehat{W},\mathcal{E}^{-}\right) }$, but the principal
transverse symbols of $\mathcal{D}^{+}$ and $\mathcal{D}^{-}$ evaluated on a
normal space to an orbit in $\widehat{W}$ correspond with the restriction of
the principal transverse symbol of $D_{b}^{E+}$ and $D_{b}^{E-}$ restricted
to the normal space to a leaf closure in $M$.
\end{proposition}

\begin{proof}
The kernels satisfy 
\begin{eqnarray*}
\ker (D_{b}^{E+}) &\cong &\ker \left( p^{\ast }\circ D_{b}^{E+}\right) \\
&\cong &\ker \left( \left. D_{b,p^{\ast }}^{+}\circ p^{\ast }\right\vert
_{\allowbreak C_{b}^{\infty }\left( M,E^{+}\right) }\right) \\
&\cong &\ker \left( \left. \left( D_{\mathrm{\mathrm{tr}},p^{\ast
}}^{p^{\ast }E}-\frac{1}{2}c\left( \widehat{H_{b}}\right) \right) \circ
p^{\ast }\right\vert _{\allowbreak C_{b}^{\infty }\left( M,E^{+}\right)
}\right) \\
&\cong &\ker \left( \left. \Phi ^{-1}\circ \left( D_{\mathrm{\mathrm{tr}}%
,p^{\ast }}^{p^{\ast }E}-\frac{1}{2}c\left( \widehat{H_{b}}\right) \right)
\circ \Phi \right\vert _{C^{\infty }\left( \widehat{W},\mathcal{E}%
^{+}\right) ^{G}}\right) \\
&\cong &\ker \left( \mathcal{D}^{+G}\right) ,
\end{eqnarray*}%
the kernel of the operator restricted to $G$ --invariant sections. Next,
while $D_{b}^{E-}$ is the adjoint of $D_{b}^{E+}$ with respect to the $L^{2}$%
-inner product on the closure of the space of basic sections of $E$, it is
not necessarily true that the adjoint of $\mathcal{D}^{+}=\left. \Phi
^{-1}\circ \left( D_{\mathrm{\mathrm{tr}},p^{\ast }}^{p^{\ast }E}-\frac{1}{2}%
c\left( \widehat{H_{b}}\right) \right) \circ \Phi \right\vert _{C^{\infty
}\left( \widehat{W},\mathcal{E}^{+}\right) }$ is $\left. \Phi ^{-1}\circ
\left( D_{\mathrm{\mathrm{tr}},p^{\ast }}^{p^{\ast }E}-\frac{1}{2}c\left( 
\widehat{H_{b}}\right) \right) \circ \Phi \right\vert _{C^{\infty }\left( 
\widehat{W},\mathcal{E}^{-}\right) }$, because although the operators have
the same principal transverse symbol, the volumes of the orbits on $\widehat{%
W}$ need not coincide with the volumes of the leaf closures on $M$, at least
with the metric on $\widehat{W}$ that we have chosen. However, it is
possible to choose a different metric, similar to that used in \cite[Theorem
3.3]{RiTransv}, so by using the induced $L^{2}$-metric on invariant sections
of $\mathcal{E}$ over $\widehat{W}$ and the $L^{2}$ metric on basic sections
of $E$ on $M$, $\Phi $ is an isometry. Specifically, let $\phi :\widehat{W}%
\rightarrow \mathbb{R}$ be the smooth positive function defined by $\phi
\left( w\right) =\mathrm{vol}\left( \widehat{\pi }^{-1}\left( w\right)
\right) $. Let $d_{\widehat{W}}$ be the dimension of $\widehat{W}$. We
determine a new metric $g^{\prime }$ on $\widehat{W}$ by conformally
multiplying the original metric $g$ on $\widehat{W}$ by $\phi ^{2/d_{%
\widehat{W}}}\in C^{\infty }\left( \widehat{W}\right) $, so that the volume
form on $\widehat{W}$ is multiplied by $\phi $. Note that $\phi \left(
w\right) \mathrm{vol}_{g}\left( \mathcal{O}_{w}\right) =\mathrm{vol}\left( 
\overline{L}\right) $ by the original construction, where $\overline{L}%
=p\left( \widehat{\pi }^{-1}\left( w\right) \right) $ is the leaf closure
corresponding to the orbit $\mathcal{O}_{w}=wG\subset \widehat{W}$. By using
the new metric $g^{\prime }$ on $\widehat{W}$, we see that $\Phi $ extends
to an $L^{2}$-isometry and that $G$ still acts by isometries on $\widehat{W}$%
. Then%
\begin{eqnarray*}
\ker \left( D_{b}^{E-}\right) &\cong &\ker \left( \left. \Phi ^{-1}\circ
\left( D_{\mathrm{\mathrm{tr}},p^{\ast }}^{p^{\ast }E}-\frac{1}{2}c\left( 
\widehat{H_{b}}\right) \right) \circ \Phi \right\vert _{C^{\infty }\left( 
\widehat{W},\mathcal{E}^{-}\right) ^{G}}\right) \\
&\cong &\ker \left( \mathcal{D}^{G,\mathrm{adj}^{\prime }}\right) ,
\end{eqnarray*}%
where the superscript $\mathrm{adj}^{\prime }$ refers to the adjoint with
respect to the $L^{2}$ metrics $C^{\infty }\left( \widehat{W},\mathcal{E}%
^{\pm }\right) ^{G}$ induced by $g^{\prime }$. Therefore, the analytic basic
index satisfies%
\begin{equation*}
\mathrm{ind}_{b}\left( D_{b}^{E+}\right) =\mathrm{ind}^{\prime }\left( 
\mathcal{D}^{G}\right) ,
\end{equation*}%
where $\mathrm{ind}^{\prime }\left( \mathcal{D}^{G}\right) $ is the analytic
index of the transversally elliptic operator $\mathcal{D}$ restricted to $G$%
-invariant sections, with adjoint calculated with respect to the choice of
metric $g^{\prime }$. Because the restriction of $\mathcal{D}$ to $G$%
-invariant sections is a Fredholm operator (see \cite{A}), $\mathrm{ind}%
^{\prime }\left( \mathcal{D}^{G}\right) $ is independent of the choice of
metric.
\end{proof}

The Fredholm properties of the equivariant index of transversally elliptic
operators (see \cite{A}) imply the following.

\begin{corollary}
\label{FredholmPropertiesCorollary}In the notation of Proposition \ref{NewD}%
, the analytic basic index$\mathrm{\ ind}_{b}\left( D_{b}^{E+}\right) $ is a
well-defined integer. Further, it is invariant under smooth deformations of
the basic operator and metrics that preserve the invertibility of the
principal symbol $\sigma \left( \xi _{x}\right) $ of $D_{b}^{E+}$ for every $%
x\in M$, but only for $\xi _{x}\in \overline{Q}_{x}^{\ast }=\left(
T_{x}M\diagup T_{x}\overline{L}_{x}\right) ^{\ast }$, the dual to the normal
space to the leaf closure through $x$.
\end{corollary}

Note that if $f$ is a smooth function on $\widehat{W}$ such that $df_{w}$ is
an element of the dual space to the normal bundle to the orbit space at $%
w\in \widehat{W}$, and if $\widehat{s}$ is a smooth section of $\mathcal{E}%
^{+}$ , then 
\begin{eqnarray*}
\left[ \mathcal{D}^{+},f\right] \widehat{s} &=&\Phi ^{-1}c\left( d\left( 
\widehat{\pi }^{\ast }f\right) ^{\#}\right) \Phi \left( \widehat{s}\right) \\
&=&\Phi ^{-1}c\left( \left( \widehat{\pi }^{\ast }df\right) ^{\#}\right)
\Phi \left( \widehat{s}\right) \\
&=&\widehat{c}\left( df^{\#}\right) \widehat{s}.
\end{eqnarray*}%
This implies that $\mathcal{D}^{+}$ is a Dirac operator on sections of $%
\mathcal{E}^{+}$, since $\left( \mathcal{D}^{+}\right) ^{\ast }\mathcal{D}%
^{+}$ is a generalized Laplacian. The analogous result is true for $\mathcal{%
D}^{-}$~.

It is possible use the Atiyah--Segal Theorem (\cite{ASe}) to compute $%
\mathrm{ind}^{G}\left( \mathcal{D}^{+}\right) $, but only in the case where $%
\mathcal{D}$ is a genuinely elliptic operator. Recall that if $D^{\prime }$
is an elliptic operator on a compact, connected manifold $M$ that is
equivariant with respect to the action of a compact Lie group $G^{\prime }$,
then $G^{\prime }$ represents on both finite--dimensional vector spaces $%
\ker D^{\prime }$ and $\ker D^{\prime \ast }$ in a natural way. For $g\in
G^{\prime }$, 
\begin{equation*}
\mathrm{ind}_{g}\left( D^{\prime }\right) :=\mathrm{tr}\left( g|\ker
D^{\prime }\right) -\mathrm{tr}\left( g|\ker D^{\prime \ast }\right) .
\end{equation*}%
where $dg$ is the normalized, bi-invariant measure on $G^{\prime }$. The
Atiyah-Segal Theorem computes this index in terms of an integral over the
fixed point set of $g$. We will use our equivariant index theorem in \cite%
{BKR} to evaluate $\mathrm{ind}^{G}\left( \mathcal{D}^{+}\right) $ in terms
of geometric invariants of the operator restricted to the strata of the
foliation.

\begin{remark}
\label{SometimesNotUseU(k)Remark}Under the additional assumption that $%
E\rightarrow \left( M,\mathcal{F}\right) $ is $G_{\mathcal{F}}$%
--equivariant, the pullback of $E$ to the transverse orthonormal frame
bundle is already transversally parallelizable. Thus, it is unnecessary in
this case to pull back again to the unitary frame bundle. We may then
replace $G$ by $O\left( q\right) $, and $\widehat{W}$ is the so-called 
\textbf{basic manifold} of the foliation, as in the standard construction in 
\cite{Mo}.
\end{remark}

\begin{remark}
\label{TransversallyOrientableCaseRemark}If $\left( M,\mathcal{F}\right) $
is in fact transversally orientable, we may replace $O\left( q\right) $ with 
$SO\left( q\right) $ and work with the bundle of oriented orthonormal frames.
\end{remark}

\subsection{The asymptotic expansion of the trace of the basic heat kernel 
\label{AsymptoticsSection}}

In this section, we will state some results concerning the spectrum of the
square of a basic Dirac-type operator and the heat kernel corresponding to
this operator, which are corollaries of the work in the previous section and
are of independent interest.

\begin{proposition}
\label{discrete spectrum} Let $D_{b}^{E+}:\allowbreak C_{b}^{\infty }\left(
M,E^{+}\right) \rightarrow \allowbreak C_{b}^{\infty }\left( M,E^{-}\right) $
be a basic Dirac operator for the rank $k$ complex vector bundle $%
E=E^{+}\oplus E^{-}$ over the transversally oriented Riemannian foliation $%
\left( M,\mathcal{F}\right) ,$ and let $\left( D_{b}^{E+}\right) ^{\mathrm{\
adj}}$ be the adjoint operator. Then the operators 
\begin{eqnarray*}
L^{+} &=&\left( D_{b}^{E+}\right) ^{\mathrm{adj}}D_{b}^{E+}, \\
L^{-} &=&D_{b}^{E+}\left( D_{b}^{E+}\right) ^{\mathrm{adj}}
\end{eqnarray*}
\newline
are essentially self--adjoint, and their spectrum consists of nonnegative
real eigenvalues with finite multiplicities. Further, the operators $L^{\pm
} $ have the same positive spectrum, including multiplicities.
\end{proposition}

\begin{proof}
By (\ref{newDFormula}) and the proof of Proposition \ref{NewD}, the
operators $L^{+}$ and $L^{-}$ are conjugate to essentially self-adjoint,
second order, $G$-equivariant, transversally elliptic operators on $\widehat{%
W}$.
\end{proof}

The basic heat kernel $K_{b}(t,x,y)$ for $L$ is a continuous section of $%
E\boxtimes E^{\ast }$ over $\mathbb{R}_{>0}\times M\times M$ that is $C^{1}$
with respect to $t$, $C^{2}$ with respect to $x$ and $y$, and satisfies, for
any vector $e_{y}\in E_{y}$, 
\begin{eqnarray*}
\left( \frac{\partial }{\partial t}+L_{x}\right) K_{b}(t,x,y)e_{y} &=&0 \\
\lim_{t\rightarrow 0^{+}}\int_{M}K_{b}(t,x,y)\,s(y)\,dV(y) &=&s(x)
\end{eqnarray*}%
for every continuous basic section $s:M\rightarrow E$. 
%For simplicity, we will assume that 
The principal transverse symbol of $L$ satisfies $\sigma \left( L\right)
\left( \xi _{x}\right) =\left\vert \xi _{x}\right\vert ^{2}I_{x}$ for every $%
\xi _{x}\in N_{x}\mathcal{F}$, where $I_{x}:E_{x}\rightarrow E_{x}$ is the
identity 
%In particular this is true if $L$ is the square of a basic Dirac--type 
operator. The existence of the basic heat kernel has already been shown in 
\cite{Ri}. Let $\overline{q}$ be the codimension of the leaf closures of $(M,%
\mathcal{F})$ with maximal dimension. The following theorems are
consequences of \cite{BrH1}, \cite{BrH2} and the conjugacy mentioned in the
proof of the proposition above.

\begin{theorem}
\label{asymptotics1} Under the assumptions in Proposition~\ref{discrete
spectrum}, let $0<\lambda _{0}^{b}\leq \lambda _{1}^{b}\leq \lambda
_{2}^{b}\leq ...$ be the eigenvalues of $L\left\vert _{\allowbreak
C_{b}^{\infty }\left( E\right) }\right. $, counting multiplicities. Then the
spectral counting function $N_{b}\left( \lambda \right) $ satisfies the
asymptotic formula 
\begin{eqnarray*}
N_{b}\left( \lambda \right) &:&=\#\left\{ \left. \lambda _{m}^{b}\right\vert
\,\lambda _{m}^{b}<\lambda \right\} \\
&\sim &\frac{\mathrm{rank}\left( \mathcal{E}\right) V_{tr}}{\left( 4\pi
\right) ^{\overline{q}/2}\Gamma \left( \frac{\overline{q}}{2}+1\right) }%
\lambda ^{\overline{q}/2}.
\end{eqnarray*}
\end{theorem}

\begin{theorem}
\label{asymptotics2} Under the assumptions in Proposition~\ref{discrete
spectrum}, the heat operators $e^{-tL^{+}}$ and $e^{-tL^{-}}$ are trace
class, and they satisfy the following asymptotic expansions. Then, as $%
t\rightarrow 0$, 
\begin{equation*}
\mathrm{Tr}e^{-tL^{\pm }}=K_{b}^{\pm }(t)\sim \frac{1}{t^{\overline{q}/2}}%
\left( a_{0}^{\pm }+\sum_{\substack{ {j\geq 1}  \\ 0\leq k<K_{0}}}%
a_{j,k}^{\pm }t^{j/2}(\log t)^{k}\right) ,
\end{equation*}%
where $K_{0}$ is less than or equal to the number of different dimensions of
closures of infinitesimal holonomy groups of the leaves of $\mathcal{F}$.
\end{theorem}

\begin{remark}
\emph{(The basic zeta function and determinant of the generalized basic
Laplacian)} We remark that due to the singular asymptotics lemma of \cite{BS}%
, we have that $a_{jk}=0$ for $j\leq \overline{q}$, $k>0$. We conjecture
that all the logarithmic terms vanish. Note that the fact that $a_{\overline{%
q}k}=0$ for $k>0$ implies that the corresponding zeta function $\zeta
_{L}\left( z\right) $ is regular at $z=0$, so that the regularized
determinant of $L$ may be defined.
\end{remark}

\subsection{Stratifications of $G$-manifolds and Riemannian foliations\label%
{stratification}}

In the following, we will describe some standard results from the theory of
Lie group actions and Riemannian foliations (see \cite{Bre}, \cite{Kaw}, 
\cite{Mo}). Such $G$-manifolds and Riemannian foliations are stratified
spaces, and the stratification can be described explicitly. In the following
discussion, we often have in mind that $\widehat{W}$ is the basic manifold
corresponding to $\left( M,\mathcal{F}\right) $ described in the last
section, but in fact the ideas apply to any Lie group $G$ acting on a
smooth, closed, connected manifold $\widehat{W}$. In the context of this
paper, either $G$ is $O\left( q\right) $, $SO\left( q\right) $, or the
product of one of these with $U\left( k\right) $. We will state the results
for general $G$ and then specialize to the case of Riemannian foliations $%
\left( M,\mathcal{F}\right) $ and the associated basic manifold. We also
emphasize that our stratification of the foliation may be finer than that
described in \cite{Mo}, because in addition we consider the action of the
holonomy on the relevant vector bundle when identifying isotropy types.

Given a $G$-manifold $\widehat{W}$ and $w\in \widehat{W}$, an orbit $%
\mathcal{O}_{w}=\left\{ gw:g\in G\right\} $ is naturally diffeomorphic to $%
G/H_{w}$, where $H_{w}=\{g\in G\,|wg=w\}$ is the (closed) isotropy or
stabilizer subgroup. In the foliation case, the group $H_{w}$ is isomorphic
to the structure group corresponding to the principal bundle $p:\widehat{\pi 
}^{-1}(w)\rightarrow \overline{L}$, where $\overline{L}$ is the leaf closure 
$p\left( \widehat{\pi }^{-1}(w)\right) $ in $M$. Given a subgroup $H$ of $G$%
, let $\left[ H\right] $ denote the conjugacy class of $H$. The isotropy
type of the orbit $\mathcal{O}_{x}$ is defined to be the conjugacy class $%
\left[ H_{w}\right] $ , which is well--defined independent of $w\in \mathcal{%
O}_{x}$. On any such $G$-manifold, there are a finite number of orbit types,
and there is a partial order on the set of orbit types. Given subgroups $H$
and $K$ of $G$, we say that $\left[ H\right] \leq $ $\left[ K\right] $ if $H$
is conjugate to a subgroup of $K$, and we say $\left[ H\right] <$ $\left[ K%
\right] $ if $\left[ H\right] \leq $ $\left[ K\right] $ and $\left[ H\right]
\neq $ $\left[ K\right] $. We may enumerate the conjugacy classes of
isotropy subgroups as $\left[ G_{0}\right] ,...,\left[ G_{r}\right] $ such
that $\left[ G_{i}\right] \leq \left[ G_{j}\right] $ implies that $i\leq j$.
It is well-known that the union of the principal orbits (those with type $%
\left[ G_{0}\right] $) form an open dense subset $\widehat{W}_{0}=\widehat{W}%
\left( \left[ G_{0}\right] \right) $ of the manifold $\widehat{W}$, and the
other orbits are called \textbf{singular}. As a consequence, every isotropy
subgroup $H$ satisfies $\left[ G_{0}\right] \leq \left[ H\right] $. Let $%
\widehat{W}_{j}$ denote the set of points of $\widehat{W}$ of orbit type $%
\left[ G_{j}\right] $ for each $j$; the set $\widehat{W}_{j}$ is called the 
\textbf{stratum} corresponding to $\left[ G_{j}\right] $. If $\left[ G_{j}%
\right] \leq \left[ G_{k}\right] $, it follows that the closure of $\widehat{%
W}_{j}$ contains the closure of $\widehat{W}_{k}$. A stratum $\widehat{W}%
_{j} $ is called a \textbf{minimal stratum} if there does not exist a
stratum $\widehat{W}_{k}$ such that $\left[ G_{j}\right] <\left[ G_{k}\right]
$ (equivalently, such that $\overline{\widehat{W}_{k}}\subsetneq \overline{%
\widehat{W}_{j}}$). It is known that each stratum is a $G$-invariant
submanifold of $\widehat{W}$, and in fact a minimal stratum is a closed (but
not necessarily connected) submanifold. Also, for each $j$, the submanifold $%
\widehat{W}_{\geq j}:=\bigcup\limits_{\left[ G_{k}\right] \geq \left[ G_{j}%
\right] }\widehat{W}_{k}$ is a closed, $G$-invariant submanifold.

Now, given a proper, $G$-invariant submanifold $S$ of $\widehat{W}$ and $%
\varepsilon >0$, let $T_{\varepsilon }(S)$ denote the union of the images of
the exponential map at $s$ for $s\in S$ restricted to the open ball of
radius $\varepsilon $ in the normal bundle at $S$. It follows that $%
T_{\varepsilon }(S)$ is also $G$-invariant. If $\widehat{W}_{j}$ is a
stratum and $\varepsilon $ is sufficiently small, then all orbits in $%
T_{\varepsilon }\left( \widehat{W}_{j}\right) \setminus \widehat{W}_{j}$ are
of type $\left[ G_{k}\right] $, where $\left[ G_{k}\right] <\left[ G_{j}%
\right] $. This implies that if $j<k$, $\overline{\widehat{W}_{j}}\cap 
\overline{\widehat{W}_{k}}\neq \varnothing $, and $\widehat{W}_{k}\subsetneq 
\overline{\widehat{W}_{j}}$, then $\overline{\widehat{W}_{j}}$ and $%
\overline{\widehat{W}_{k}}$ intersect at right angles, and their
intersection consists of more singular strata (with isotropy groups
containing conjugates of both $G_{k}$ and $G_{j}$).

Fix $\varepsilon >0$. We now decompose $\widehat{W}$ as a disjoint union of
sets $\widehat{W}_{0}^{\varepsilon },\dots ,\widehat{W}_{r}^{\varepsilon }$.
If there is only one isotropy type on $\widehat{W}$, then $r=0$, and we let $%
\widehat{W}_{0}^{\varepsilon }=\Sigma _{0}^{\varepsilon }=\widehat{W}_{0}=%
\widehat{W}$. Otherwise, for $j=r,r-1,...,0$, let $\varepsilon
_{j}=2^{j}\varepsilon $, and let%
\begin{equation}
{\ }\Sigma _{j}^{\varepsilon }=\widehat{W}_{j}\setminus \overline{%
\bigcup_{k>j}\widehat{W}_{k}^{\varepsilon }},~~~\widehat{W}_{j}^{\varepsilon
}=T_{\varepsilon _{j}}\left( \widehat{W}_{j}\right) \setminus \overline{%
\bigcup_{k>j}\widehat{W}_{k}^{\varepsilon }},  \label{Mjepsilon}
\end{equation}%
Thus, 
\begin{equation*}
{\ }T_{\varepsilon }\left( \Sigma _{j}^{\varepsilon }\right) \subset 
\widehat{W}_{j}^{\varepsilon },~\Sigma _{j}^{\varepsilon }\subset \widehat{W}%
_{j}.
\end{equation*}

We now specialize to the foliation case. Let $\left( M,\mathcal{F}\right) $
be a Riemannian foliation, and let $E\rightarrow M$ be a foliated Hermitian
vector bundle over $M$ (defined in Section \ref{BasicConnectionsSection}).
Let the $G$-bundle $\widehat{M}\rightarrow M$ be either the orthonormal
transverse frame bundle of $\left( M,\mathcal{F}\right) $ or the bundle of
ordered pairs $\left( \alpha ,\beta \right) $, with $\alpha $ a orthonormal
transverse frame and $\beta $ an orthonormal frame of $E$ with respect to
the Hermitian inner product on $E$, as in Section \ref%
{BasicIndexAndDiagramSection}. In the former case, $\widehat{M}\overset{p}{%
\longrightarrow }M$ is an $O\left( q\right) $-bundle, and in the latter
case, $\widehat{M}$ is an $O\left( q\right) \times U\left( k\right) $%
-bundle. We also note that in the case where $\left( M,\mathcal{F}\right) $
is transversally oriented, we may replace $O\left( q\right) $ with $SO\left(
q\right) $ and choose oriented transverse frames. In Section \ref%
{BasicIndexAndDiagramSection}, we showed that the foliation $\mathcal{F}$
lifts to a foliation $\widehat{\mathcal{F}}$ on $\widehat{M}$, and the
lifted foliation is transversally parallelizable. We chose a natural metric
on $\widehat{M}$, as explained in Section \ref{BasicIndexAndDiagramSection}.
By Molino theory (\cite{Mo}), the leaf closures of $\widehat{\mathcal{F}}$
are diffeomorphic and have no holonomy; they form a Riemannian fiber bundle $%
\widehat{M}\overset{\widehat{\pi }}{\longrightarrow }\widehat{W}$ over what
is called the \textbf{basic manifold} $\widehat{W}$, on which the group $G$
acts by isometries.

We identify the spaces $\widehat{W}_{i}^{\varepsilon }$, $\widehat{W}%
_{i}^{\varepsilon }\diagup G$, and $\overline{\widehat{W}_{i}^{\varepsilon }}%
\diagup G$ with the corresponding $M_{i}^{\varepsilon }$, $%
M_{i}^{\varepsilon }\diagup \overline{\mathcal{F}}$, and $\overline{%
M_{i}^{\varepsilon }}\diagup \overline{\mathcal{F}}$ on $M$ via the
correspondence 
\begin{equation*}
p\left( \widehat{\pi }^{-1}\left( G\text{--orbit on }\widehat{W}\right)
\right) =\text{leaf closure of }\left( M,\mathcal{F}\right) .
\end{equation*}%
The following result is contained in \cite{Mo}, which is a consequence of
Riemannian foliation theory and the decomposition theorems of $G$-manifolds
(see \cite{Kaw}). However, we note that the decomposition described below
may be finer than that described in Molino, as the bundle $E\rightarrow M$
is used in our construction to construct the basic manifold, and the group
acting may be larger than the orthogonal group. The action of the holonomy
on the bundle may participate in the decomposition of the foliation.

\begin{lemma}
\label{foliationDecomposition}Let $\left( M,\mathcal{F}\right) $ be a
Riemannian foliation with bundle-like metric. Let $\overline{\mathcal{F}}$
denote the (possibly) singular foliation by leaf closures of $\mathcal{F}$.
We let $M_{j}=M\left( \left[ G_{j}\right] \right) =p\left( \widehat{\pi }%
^{-1}\left( \widehat{W}\left( \left[ G_{j}\right] \right) \right) \right) $, 
$M_{i}^{\varepsilon }=p\left( \widehat{\pi }^{-1}\left( \widehat{W}%
_{i}^{\varepsilon }\right) \right) $ with $\widehat{W}$, $\widehat{W}\left( %
\left[ G_{j}\right] \right) $, $\widehat{W}_{i}^{\varepsilon }$ defined as
above on the basic manifold. Note that $M_{j}$ is a stratum on $M$
corresponding to the union of all leaf closures whose structure group of the
principal bundle $p:\overline{\widehat{L}}\rightarrow \overline{L}$ is in $%
\left[ G_{j}\right] $, where $\overline{\widehat{L}}$ is a leaf closure of $%
\widehat{M}$ that projects to $\overline{L}$. It follows that all the leaf
closures in $M_{j}$ have the same dimension. Then we have, for every $i\in
\{1,\ldots ,r\}$ and sufficiently small $\varepsilon >0$:

\begin{enumerate}
\item $\displaystyle M=\coprod_{i=0}^{r}M_{i}^{\varepsilon }$ (disjoint
union).

\item $M_{i}^{\varepsilon }$ is a union of leaf closures.

\item The manifold $M_{i}^{\varepsilon }$ is diffeomorphic to the interior
of a compact manifold with corners; the leaf closure space space $%
M_{i}^{\varepsilon }\diagup \overline{\mathcal{F}}\cong \widehat{W}%
_{i}^{\varepsilon }\diagup O\left( q\right) $ is a smooth manifold that is
isometric to the interior of a triangulable, compact manifold with corners.
The same is true for each $\Sigma _{i}^{\varepsilon }$, $M_{i}\diagup 
\overline{\mathcal{F}}$.

\item If $\left[ G_{j}\right] $ is the isotropy type of an orbit in $%
M_{i}^{\varepsilon }$, then $j\leq i$ and $\left[ G_{j}\right] \leq \left[
G_{i}\right] $.

\item The distance between the submanifold $M_{j}$ and $M_{i}^{\varepsilon }$
for $j>i$ is at least $\varepsilon $.
\end{enumerate}
\end{lemma}

\begin{remark}
The lemma above remains true if at each stage $T_{\varepsilon }\left(
M_{j}\right) $ is replaced by any sufficiently small open neighborhood of $%
M_{j}$ that contains $T_{\varepsilon }\left( M_{j}\right) $, that is a $%
\mathcal{F}$-saturated, and whose closure is a manifold with corners.
\end{remark}

\begin{remark}
The additional frames of $E$ have no effect on the stratification of $M$;
the corresponding $M_{i}^{\varepsilon }$, $M_{i}$ are identical whether or
not the bundle $\widehat{M}\rightarrow M$ is chosen to be the $O\left(
q\right) $-bundle or the $O\left( q\right) \times U\left( k\right) $-bundle.
However, the isotropy subgroups and basic manifold $\widehat{W}$ are
different and depend on the structure of the bundle $E$.
\end{remark}

\begin{definition}
With notation as in this section, suppose that $\left[ H\right] $ is a
maximal isotropy type with respect to the partial order $\leq $. Then the
closed, saturated submanifold $M\left( \left[ H\right] \right) $ is called a 
\textbf{minimal stratum} of the foliation $\left( M,\mathcal{F}\right) $.
\end{definition}

\subsection{Fine components and canonical isotropy bundles\label%
{canonIsotropyBundlesSection}}

First we review some definitions from \cite{BKR} and \cite{KRi} concerning
manifolds $X$ on which a compact Lie group $G$ acts by isometries with
single orbit type $\left[ H\right] $. Let $X^{H}$ be the fixed point set of $%
H$, and for $\alpha \in \pi _{0}\left( X^{H}\right) $, let $X_{\alpha }^{H}$
denote the corresponding connected component of $X^{H}$.

\begin{definition}
\label{componentRelGDefn}We denote $X_{\alpha }=GX_{\alpha }^{H}$, and $%
X_{\alpha }$ is called a \textbf{component of} $X$ \textbf{relative to} $G$.
\end{definition}

\begin{remark}
The space $X_{\alpha }$ is not necessarily connected, but it is the inverse
image of a connected component of $G\diagdown X=N\diagdown X^{H}$ under the
projection $X\rightarrow G\diagdown X$. Also, note that $X_{\alpha
}=X_{\beta }$ if there exists $n\in N$ such that $nX_{\alpha }^{H}=X_{\beta
}^{H}$. If $X$ is a closed manifold, then there are a finite number of
components of $X$ relative to $G$.
\end{remark}

We now introduce a decomposition of a $G$-bundle $E\rightarrow X$ over a $G$%
-space with single orbit type $\left[ H\right] $ that is a priori finer than
the normalized isotypical decomposition. Let $E_{\alpha }$ be the
restriction $\left. E\right\vert _{X_{\alpha }^{H}}$. For $\sigma
:H\rightarrow U\left( W_{\sigma }\right) $ an irreducible unitary
representation, let $\sigma ^{n}:H\rightarrow U\left( W_{\sigma }\right) $
be the irreducible representation defined by%
\begin{equation*}
\sigma ^{n}\left( h\right) =\sigma \left( n^{-1}hn\right) .
\end{equation*}%
Let $\widetilde{N}_{\left[ \sigma \right] }=\left\{ n\in N:\left[ \sigma ^{n}%
\right] \text{~is~equivalent~to}~\left[ \sigma \right] ~\right\} $ . If the
isotypical component 
\begin{equation*}
E_{x}^{\left[ \sigma \right] }:=i_{\sigma }\left( \mathrm{Hom}_{H}\left(
W_{\sigma },E_{x}\right) \otimes W_{\sigma }\right)
\end{equation*}%
is nontrivial, then it is invariant under the subgroup $\widetilde{N}%
_{\alpha ,\left[ \sigma \right] }\subseteq \widetilde{N}_{\left[ \sigma %
\right] }$ that leaves in addition the connected component $X_{\alpha }^{H}$
invariant; again, this subgroup has finite index in $N$. The isotypical
components transform under $n\in N$ as%
\begin{equation*}
n:E_{\alpha }^{\left[ \sigma \right] }\overset{\cong }{\longrightarrow }E_{%
\overline{n}\left( \alpha \right) }^{\left[ \sigma ^{n}\right] }~,
\end{equation*}%
where $\overline{n}$ denotes the residue class class of $n\in N$ in $%
N\diagup \widetilde{N}_{\alpha ,\left[ \sigma \right] }~$. Then a
decomposition of $E$ is obtained by `inducing up' the isotypical components $%
E_{\alpha }^{\left[ \sigma \right] }$ from $\widetilde{N}_{\alpha ,\left[
\sigma \right] }$ to $N$. That is, 
\begin{equation*}
E_{\alpha ,\left[ \sigma \right] }^{N}=N\times _{\widetilde{N}_{\alpha ,%
\left[ \sigma \right] }}E_{\alpha }^{\left[ \sigma \right] }
\end{equation*}%
is a bundle containing $\left. E_{\alpha }^{\left[ \sigma \right]
}\right\vert _{X_{\alpha }^{H}}$ . This is an $N$-bundle over $NX_{\alpha
}^{H}\subseteq X^{H}$, and a similar bundle may be formed over each distinct 
$NX_{\beta }^{H}$, with $\beta \in \pi _{0}\left( X^{H}\right) $. Further,
observe that since each bundle $E_{\alpha ,\left[ \sigma \right] }^{N}$ is
an $N$-bundle over $NX_{\alpha }^{H}$, it defines a unique $G$ bundle $%
E_{\alpha ,\left[ \sigma \right] }^{G}$.

\begin{definition}
\label{fineComponentDefinition}The $G$-bundle $E_{\alpha ,\left[ \sigma %
\right] }^{G}$ over the submanifold $X_{\alpha }$ is called a \textbf{fine
component} or the \textbf{fine component of }$E\rightarrow X$ \textbf{%
associated to }$\left( \alpha ,\left[ \sigma \right] \right) $.
\end{definition}

If $G\diagdown X$ is not connected, one must construct the fine components
separately over each $X_{\alpha }$. If $E$ has finite rank, then $E$ may be
decomposed as a direct sum of distinct fine components over each $X_{\alpha
} $. In any case, $E_{\alpha ,\left[ \sigma \right] }^{N}$ is a finite
direct sum of isotypical components over each $X_{\alpha }^{H}$.

\begin{definition}
\label{FineDecompositionDefinition}The direct sum decomposition of $\left.
E\right\vert _{X_{\alpha }}$ into subbundles $E^{b}$ that are fine
components $E_{\alpha ,\left[ \sigma \right] }^{G}$ for some $\left[ \sigma %
\right] $, written 
\begin{equation*}
\left. E\right\vert _{X_{\alpha }}=\bigoplus\limits_{b}E^{b}~,
\end{equation*}%
is called the \textbf{refined isotypical decomposition} (or \textbf{fine
decomposition}) of $\left. E\right\vert _{X_{\alpha }}$.
\end{definition}

In the case where $G\diagdown X$ is connected, the group $\pi _{0}\left(
N\diagup H\right) $ acts transitively on the connected components $\pi
_{0}\left( X^{H}\right) $, and thus $X_{\alpha }=X$. We comment that if $%
\left[ \sigma ,W_{\sigma }\right] $ is an irreducible $H$-representation
present in $E_{x}$ with $x\in X_{\alpha }^{H}$, then $E_{x}^{\left[ \sigma %
\right] }$ is a subspace of a distinct $E_{x}^{b}$ for some $b$. The
subspace $E_{x}^{b}$ also contains $E_{x}^{\left[ \sigma ^{n}\right] }$ for
every $n$ such that $nX_{\alpha }^{H}=X_{\alpha }^{H}$~.

\begin{remark}
\label{constantMultiplicityRemark}Observe that by construction, for $x\in
X_{\alpha }^{H}$ the multiplicity and dimension of each $\left[ \sigma %
\right] $ present in a specific $E_{x}^{b}$ is independent of $\left[ \sigma %
\right] $. Thus, $E_{x}^{\left[ \sigma ^{n}\right] }$ and $E_{x}^{\left[
\sigma \right] }$ have the same multiplicity and dimension if $nX_{\alpha
}^{H}=X_{\alpha }^{H}$~.
\end{remark}

\begin{remark}
The advantage of this decomposition over the isotypical decomposition is
that each $E^{b}$ is a $G$-bundle defined over all of $X_{\alpha }$, and the
isotypical decomposition may only be defined over $X_{\alpha }^{H}$.
\end{remark}

\begin{definition}
\label{adaptedDefn}Now, let $E$ be a $G$-equivariant vector bundle over $X$,
and let $E^{b}~$be a fine component as in Definition \ref%
{fineComponentDefinition} corresponding to a specific component $X_{\alpha
}=GX_{\alpha }^{H}$ of $X$ relative to $G$. Suppose that another $G$-bundle $%
W$ over $X_{\alpha }$ has finite rank and has the property that the
equivalence classes of $G_{y}$-representations present in $E_{y}^{b},y\in
X_{\alpha }$ exactly coincide with the equivalence classes of $G_{y}$%
-representations present in $W_{y}$, and that $W$ has a single component in
the fine decomposition. Then we say that $W$ is \textbf{adapted} to $E^{b}$.
\end{definition}

\begin{lemma}
\label{AdaptedToAnyBundleLemma}In the definition above, if another $G$%
-bundle $W$ over $X_{\alpha }$ has finite rank and has the property that the
equivalence classes of $G_{y}$-representations present in $E_{y}^{b},y\in
X_{\alpha }$ exactly coincide with the equivalence classes of $G_{y}$%
-representations present in $W_{y}$, then it follows that $W$ has a single
component in the fine decomposition and hence is adapted to $E^{b}$. Thus,
the last phrase in the corresponding sentence in the above definition is
superfluous.
\end{lemma}

\begin{proof}
Suppose that we choose an equivalence class $\left[ \sigma \right] $ of $H$%
-representations present in $W_{x}$, $x\in X_{\alpha }^{H}$. Let $\left[
\sigma ^{\prime }\right] $ be any other equivalence class; then, by
hypothesis, there exists $n\in N$ such that $nX_{\alpha }^{H}=X_{\alpha
}^{H} $ and $\left[ \sigma ^{\prime }\right] =\left[ \sigma ^{n}\right] $.
Then, observe that $nW_{x}^{\left[ \sigma \right] }=W_{nx}^{\left[ \sigma
^{n}\right] }=W_{x}^{\left[ \sigma ^{n}\right] }$, with the last equality
coming from the rigidity of irreducible $H$-representations. Thus, $W$ is
contained in a single fine component, and so it must have a single component
in the fine decomposition.
\end{proof}

In what follows, we show that there are naturally defined finite-dimensional
vector bundles that are adapted to any fine components. Once and for all, we
enumerate the irreducible representations $\left\{ \left[ \rho _{j},V_{\rho
_{j}}\right] \right\} _{j=1,2,...}$ of $G$. Let $\left[ \sigma ,W_{\sigma }%
\right] $ be any irreducible $H$-representation. Let $G\times _{H}W_{\sigma
} $ be the corresponding homogeneous vector bundle over the homogeneous
space $G\diagup H$. Then the $L^{2}$-sections of this vector bundle
decompose into irreducible $G$-representations. In particular, let $\left[
\rho _{j_{0}},V_{\rho _{j_{0}}}\right] $ be the equivalence class of
irreducible representations that is present in $L^{2}\left( G\diagup
H,G\times _{H}W_{\sigma }\right) $ and that has the lowest index $j_{0}$.
Then Frobenius reciprocity implies%
\begin{equation*}
0\neq \mathrm{Hom}_{G}\left( V_{\rho _{j_{0}}},L^{2}\left( G\diagup
H,G\times _{H}W_{\sigma }\right) \right) \cong \mathrm{Hom}_{H}\left( V_{%
\mathrm{\mathrm{Res}}\left( \rho _{j_{0}}\right) },W_{\sigma }\right) ,
\end{equation*}%
so that the restriction of $\rho _{j_{0}}$ to $H$ contains the $H$%
-representation $\left[ \sigma \right] $. Now, for a component $X_{\alpha
}^{H}$ of $X^{H}$, with $X_{\alpha }=GX_{\alpha }^{H}$ its component in $X$
relative to $G$, the trivial bundle%
\begin{equation*}
X_{\alpha }\times V_{\rho _{j_{0}}}
\end{equation*}%
is a $G$-bundle (with diagonal action) that contains a nontrivial fine
component $W_{\alpha ,\left[ \sigma \right] }$ containing $X_{\alpha
}^{H}\times \left( V_{\rho _{j_{0}}}\right) ^{\left[ \sigma \right] }$.

\begin{definition}
\label{canonicalIsotropyBundleDefinition}We call $W_{\alpha ,\left[ \sigma %
\right] }\rightarrow X_{\alpha }$ the \textbf{canonical isotropy }$G$\textbf{%
-bundle associated to }$\left( \alpha ,\left[ \sigma \right] \right) \in \pi
_{0}\left( X^{H}\right) \times \widehat{H}$. Observe that $W_{\alpha ,\left[
\sigma \right] }$ depends only on the enumeration of irreducible
representations of $G$, the irreducible $H$-representation $\left[ \sigma %
\right] $ and the component $X_{\alpha }^{H}$. We also denote the following
positive integers associated to $W_{\alpha ,\left[ \sigma \right] }$:

\begin{itemize}
\item $m_{\alpha ,\left[ \sigma \right] }=\dim \mathrm{Hom}_{H}\left(
W_{\sigma },W_{\alpha ,\left[ \sigma \right] ,x}\right) =\dim \mathrm{Hom}%
_{H}\left( W_{\sigma },V_{\rho _{j_{0}}}\right) $ (the \textbf{associated
multiplicity}), independent of the choice of $\left[ \sigma ,W_{\sigma }%
\right] $ present in $W_{\alpha ,\left[ \sigma \right] ,x}$ , $x\in
X_{\alpha }^{H}$ (see Remark \ref{constantMultiplicityRemark}).

\item $d_{\alpha ,\left[ \sigma \right] }=\dim W_{\sigma }$(the \textbf{%
associated representation dimension}), independent of the choice of $\left[
\sigma ,W_{\sigma }\right] $ present in $W_{\alpha ,\left[ \sigma \right]
,x} $ , $x\in X_{\alpha }^{H}$.

\item $n_{\alpha ,\left[ \sigma \right] }=\frac{\mathrm{rank}\left(
W_{\alpha ,\left[ \sigma \right] }\right) }{m_{\alpha ,\left[ \sigma \right]
}d_{\alpha ,\left[ \sigma \right] }}$ (the \textbf{inequivalence number}),
the number of inequivalent representations present in $W_{\alpha ,\left[
\sigma \right] ,x}$ , $x\in X_{\alpha }^{H}$.
\end{itemize}
\end{definition}

\begin{remark}
Observe that $W_{\alpha ,\left[ \sigma \right] }=W_{\alpha ^{\prime },\left[
\sigma ^{\prime }\right] }$ if $\left[ \sigma ^{\prime }\right] =\left[
\sigma ^{n}\right] $ for some $n\in N$ such that $nX_{\alpha }^{H}=X_{\alpha
^{\prime }}^{H}~$.
\end{remark}

The lemma below follows immediately from Lemma \ref{AdaptedToAnyBundleLemma}.

\begin{lemma}
\label{canIsotropyGbundleAdaptedExists}Given any $G$-bundle $E\rightarrow X$
and any fine component $E^{b}$ of $E$ over some $X_{\alpha }=GX_{\alpha
}^{H} $, there exists a canonical isotropy $G$-bundle $W_{\alpha ,\left[
\sigma \right] }$ adapted to $E^{b}\rightarrow X_{\alpha }$.
\end{lemma}

An example of another foliated bundle over a component of a stratum $M_{j}$
is the bundle defined as follows.

\begin{definition}
\label{W-tau-twistDefinition}Let $E\rightarrow M$ be any foliated vector
bundle. Let $\Sigma _{\alpha _{j}}=\widehat{\pi }\left( p^{-1}\left(
M_{j}\right) \right) $ be the corresponding component of the stratum
relative to $G$ on the basic manifold $\widehat{W}$ (see Section \ref%
{stratification}), and let $W^{\tau }\rightarrow \Sigma _{\alpha _{j}}$ be a
canonical isotropy bundle (Definition \ref{canonicalIsotropyBundleDefinition}%
). Consider the bundle $\widehat{\pi }^{\ast }W^{\tau }\otimes p^{\ast
}E\rightarrow p^{-1}\left( M_{j}\right) $, which is foliated and basic for
the lifted foliation restricted to $p^{-1}\left( M_{j}\right) $. This
defines a new foliated bundle $E^{\tau }\rightarrow M_{j}$ by letting $%
E_{x}^{\tau }$ be the space of $G$-invariant sections of $\widehat{\pi }%
^{\ast }W^{\tau }\otimes p^{\ast }E$ restricted to $p^{-1}\left( x\right) $.
We call this bundle \textbf{the} $W^{\tau }$\textbf{-twist of} $E\rightarrow
M_{j}$.
\end{definition}

\section{Desingularization of the foliation\label{desingularizationSection}}

\subsection{Topological Desingularization\label{BlowUpSection}}

Assume that $\left( M,\mathcal{F}\right) $ is a Riemannian foliation, with
principal stratum $M_{0}$ and singular strata $M_{1},...,M_{r}$
corresponding to isotropy types $\left[ G_{0}\right] $, $\left[ G_{1}\right] 
$, $\left[ G_{2}\right] $, ..., $\left[ G_{r}\right] $ on the basic
manifold, as explained in Section \ref{stratification}. We will construct a
new Riemannian foliation $\left( N,\mathcal{F}_{N}\right) $ that has a
single stratum (of type $\left[ G_{0}\right] $) and that is a branched cover
of $M$, branched over the singular strata. A distinguished fundamental
domain of $M_{0}$ in $N$ is called the desingularization of $M$ and is
denoted $\widetilde{M}$. This process closely parallels the process of
desingularizing a $G$-manifold, which is described in \cite{BKR}.

Recall the setup from Section \ref{stratification}. We are given $%
E\rightarrow M$, a foliated Hermitian vector bundle over $M$, and the bundle 
$\widehat{M}\overset{p}{\longrightarrow }M$ is the bundle of ordered pairs $%
\left( \alpha ,\beta \right) $ with structure group $G=O\left( q\right)
\times U\left( k\right) $, with $\alpha $ a orthonormal transverse frame and 
$\beta $ an orthonormal frame of $E$ with respect to the Hermitian inner
product on $E$, as in Section \ref{BasicIndexAndDiagramSection}; in many
cases the principal bundle may be reduced to a bundle with smaller structure
group. The foliation $\mathcal{F}$ lifts to a foliation $\widehat{\mathcal{F}%
}$ on $\widehat{M}$, and the lifted foliation is transversally
parallelizable. We chose the natural metric on $\widehat{M}$ as in Section %
\ref{BasicIndexAndDiagramSection}. By Molino theory (\cite{Mo}), the leaf
closures of $\widehat{\mathcal{F}}$ are diffeomorphic, have no holonomy, and
form a Riemannian fiber bundle $\widehat{M}\overset{\widehat{\pi }}{%
\longrightarrow }\widehat{W}$ over the basic manifold $\widehat{W}$, on
which the group $G$ acts by isometries. The $G$-orbits on $\widehat{W}$ and
leaf closures of $\left( M,\mathcal{F}\right) $ are identified via the
correspondence 
\begin{equation*}
p\left( \widehat{\pi }^{-1}\left( G\text{--orbit on }\widehat{W}\right)
\right) =\text{leaf closure of }\left( M,\mathcal{F}\right) .
\end{equation*}

A sequence of modifications is used to construct $N$ and $\widetilde{M}%
\subset N$. Let $M_{j}$ be a minimal stratum. Let $T_{\varepsilon }\left(
M_{j}\right) $ denote a tubular neighborhood of radius $\varepsilon $ around 
$M_{j}$, with $\varepsilon $ chosen sufficiently small so that all leaf
closures in $T_{\varepsilon }\left( M_{j}\right) \setminus M_{j}$ correspond
to isotropy types $\left[ G_{k}\right] $, where $\left[ G_{k}\right] <\left[
G_{j}\right] $. Let 
\begin{equation*}
N^{1}=\left( M\setminus T_{\varepsilon }\left( M_{j}\right) \right) \cup
_{\partial T_{\varepsilon }\left( M_{j}\right) }\left( M\setminus
T_{\varepsilon }\left( M_{j}\right) \right)
\end{equation*}%
be the manifold constructed by gluing two copies of $\left( M\setminus
T_{\varepsilon }\left( M_{j}\right) \right) $ smoothly along the boundary.
Since the $T_{\varepsilon }\left( M_{j}\right) $ is saturated (a union of
leaves), the foliation lifts to $N^{1}$. Note that the strata of the
foliation $\mathcal{F}^{1}$ on $N^{1}$ correspond to strata in $M\setminus
T_{\varepsilon }\left( M_{j}\right) $. If $M_{k}\cap \left( M\setminus
T_{\varepsilon }\left( M_{j}\right) \right) $ is nontrivial, then the
stratum corresponding to isotropy type $\left[ G_{k}\right] $ on $N^{1}$ is 
\begin{equation*}
N_{k}^{1}=\left( M_{k}\cap \left( M\setminus T_{\varepsilon }\left(
M_{j}\right) \right) \right) \cup _{\left( M_{k}\cap \partial T_{\varepsilon
}\left( M_{j}\right) \right) }\left( M_{k}\cap \left( M\setminus
T_{\varepsilon }\left( M_{j}\right) \right) \right) .
\end{equation*}%
Thus, $\left( N^{1},\mathcal{F}^{1}\right) $ is a foliation with one fewer
stratum than $\left( M,\mathcal{F}\right) $, and $M\setminus M_{j}$ is
diffeomorphic to one copy of $\left( M\setminus T_{\varepsilon }\left(
M_{j}\right) \right) $, denoted $\widetilde{M}^{1}$ in $N^{1}$. One may
radially modify metrics so that a bundle-like metric on $\left( M,\mathcal{F}%
\right) $ transforms to a bundle-like metric on $\left( N^{1},\mathcal{F}%
^{1}\right) $. In fact, $N^{1}$ is a branched double cover of $M$, branched
over $M_{j}$. If the leaf closures of $\left( N^{1},\mathcal{F}^{1}\right) $
correspond to a single orbit type, then we set $N=N^{1}$ and $\widetilde{M}=%
\widetilde{M}^{1}$. If not, we repeat the process with the foliation $\left(
N^{1},\mathcal{F}^{1}\right) $ to produce a new Riemannian foliation $\left(
N^{2},\mathcal{F}^{2}\right) $ with two fewer strata than $\left( M,\mathcal{%
F}\right) $ and that is a $4$-fold branched cover of $M$. Again, $\widetilde{%
M}^{2}$ is a fundamental domain of $\widetilde{M}^{1}\setminus \left\{ \text{%
a minimal stratum}\right\} $, which is a fundamental domain of $M$ with two
strata removed. We continue until $\left( N,\mathcal{F}_{N}\right) =\left(
N^{r},\mathcal{F}^{r}\right) $ is a Riemannian foliation with all leaf
closures corresponding to orbit type $\left[ G_{0}\right] $ and is a $2^{r}$%
-fold branched cover of $M$, branched over $M\setminus M_{0}$. We set $%
\widetilde{M}=\widetilde{M}^{r}$, which is a fundamental domain of $M_{0}$
in $N$.

Further, one may independently desingularize $M_{\geq j}$, since this
submanifold is itself a closed $G$-manifold. If $M_{\geq j}$ has more than
one connected component, we may desingularize all components simultaneously.
The isotropy type corresponding to all leaf closures of $\widetilde{M_{\geq
j}}$ is $\left[ G_{j}\right] $, and $\widetilde{M_{\geq j}}\diagup \overline{%
\mathcal{F}}$ is a smooth (open) manifold.

\subsection{Modification of the metric and differential operator}

We now more precisely describe the desingularization. If $\left( M,\mathcal{F%
}\right) $ is equipped with a basic, transversally elliptic differential
operator on sections of a foliated vector bundle over $M$, then this data
may be pulled back to the desingularization $\widetilde{M}$. Given the
bundle and operator over $N^{j}$, simply form the invertible double of the
operator on $N^{j+1}$, which is the double of the manifold with boundary $%
N^{j}\setminus T_{\varepsilon }\left( \Sigma \right) $, where $\Sigma $ is a
minimal stratum on $N^{j}$.

Specifically, we modify the bundle-like metric radially so that there exists
sufficiently small $\varepsilon >0$ such that the (saturated) tubular
neighborhood $B_{4\varepsilon }\Sigma $ of $\Sigma $ in $N^{j}$ is isometric
to a ball of radius $4\varepsilon $ in the normal bundle $N\Sigma $. In
polar coordinates, this metric is $ds^{2}=dr^{2}+d\sigma ^{2}+r^{2}d\theta
_{\sigma }^{2}$, with $r\in \left( 0,4\varepsilon \right) $, $d\sigma ^{2}$
is the metric on $\Sigma $, and $d\theta _{\sigma }^{2}$ is the metric on $%
S\left( N_{\sigma }\Sigma \right) $, the unit sphere in $N_{\sigma }\Sigma $%
; note that $d\theta _{\sigma }^{2}$ is isometric to the Euclidean metric on
the unit sphere. We simply choose the horizontal metric on $B_{4\varepsilon
}\Sigma $ to be the pullback of the metric on the base $\Sigma $, the fiber
metric to be Euclidean, and we require that horizontal and vertical vectors
be orthogonal. We do not assume that the horizontal distribution is
integrable. We that the metric constructed above is automatically
bundle-like for the foliation.

Next, we replace $r^{2}$ with $f\left( r\right) =\left[ \psi \left( r\right) %
\right] ^{2}$ in the expression for the metric, where $\psi \left( r\right) $
is increasing, is a positive constant for $0\leq r\leq \varepsilon $, and $%
\psi \left( r\right) =r$ for $2\varepsilon \leq r\leq 3\varepsilon $. Then
the metric is cylindrical for $r<\varepsilon $.

In our description of the modification of the differential operator, we will
need the notation for the (external) product of differential operators.
Suppose that $F\hookrightarrow X\overset{\pi }{\rightarrow }B$ is a fiber
bundle that is locally a metric product. Given an operator $A_{1,x}:\Gamma
\left( \pi ^{-1}\left( x\right) ,E_{1}\right) \rightarrow \Gamma \left( \pi
^{-1}\left( x\right) ,F_{1}\right) $ that is locally given as a differential
operator $A_{1}:\Gamma \left( F,E_{1}\right) \rightarrow \Gamma \left(
F,F_{1}\right) $ and $A_{2}:\Gamma \left( B,E_{2}\right) \rightarrow \Gamma
\left( B,F_{2}\right) $ on Hermitian bundles, we define the product 
\begin{equation*}
A_{1,x}\ast A_{2}:\Gamma \left( X,\left( E_{1}\boxtimes E_{2}\right) \oplus
\left( F_{1}\boxtimes F_{2}\right) \right) \rightarrow \Gamma \left(
X,\left( F_{1}\boxtimes E_{2}\right) \oplus \left( E_{1}\boxtimes
F_{2}\right) \right)
\end{equation*}%
as the unique linear operator that satisfies locally 
\begin{equation*}
A_{1,x}\ast A_{2}=\left( 
\begin{array}{ll}
A_{1}\boxtimes \mathbf{1} & -\mathbf{1}\boxtimes A_{2}^{\ast } \\ 
\mathbf{1}\boxtimes A_{2} & A_{1}^{\ast }\boxtimes \mathbf{1}%
\end{array}%
\right)
\end{equation*}%
on sections of 
\begin{equation*}
\left( 
\begin{array}{l}
E_{1}\boxtimes E_{2} \\ 
F_{1}\boxtimes F_{2}%
\end{array}%
\right)
\end{equation*}%
of the form $\left( 
\begin{array}{l}
u_{1}\boxtimes u_{2} \\ 
v_{1}\boxtimes v_{2}%
\end{array}%
\right) $, where $u_{1}\in \Gamma \left( F,E_{1}\right) $, $u_{2}\in \Gamma
\left( B,E_{2}\right) $, $v_{1}\in \Gamma \left( F,F_{1}\right) $, $v_{2}\in
\Gamma \left( B,E_{2}\right) $. This coincides with the product in various
versions of K-theory (see, for example, \cite{A}, \cite[pp. 384ff]{LM}),
which is used to define the Thom Isomorphism in vector bundles.

Let $D=D^{+}:\Gamma \left( N^{j},E^{+}\right) \rightarrow \Gamma \left(
N^{j},E^{-}\right) $ be the given first order, transversally elliptic, $%
\mathcal{F}^{j}$-basic differential operator. Let $\Sigma \ $be a minimal
stratum of $N^{j}$. We assume for the moment that $\Sigma $ has codimension
at least two. We modify the bundle radially so that the foliated bundle $E$
over $B_{4\varepsilon }\left( \Sigma \right) $ is a pullback of the bundle $%
\left. E\right\vert _{\Sigma }\rightarrow \Sigma $. We assume that near $%
\Sigma $, after a foliated homotopy $D^{+}$ can be written on $%
B_{4\varepsilon }\left( \Sigma \right) $ locally as the product 
\begin{equation}
D^{+}=\left( D_{N}\ast D_{\Sigma }\right) ^{+},  \label{Dsplitting}
\end{equation}%
where $D_{\Sigma }\ $is a transversally elliptic, basic, first order
operator on the stratum $\left( \Sigma ,\left. \mathcal{F}\right\vert
_{\Sigma }\right) $, and $D_{N}$ is a basic, first order operator on $%
B_{4\varepsilon }\left( \Sigma \right) $ that is elliptic on the fibers. If $%
r$ is the distance from $\Sigma $, we write $D_{N}$ in polar coordinates as%
\begin{equation*}
D_{N}=Z\left( \nabla _{\partial _{r}}^{E}+\frac{1}{r}D^{S}\right)
\end{equation*}%
where $Z=-i\sigma \left( D_{N}\right) \left( \partial _{r}\right) $ is a
local bundle isomorphism and the map $D^{S}$ is a purely first order
operator that differentiates in the unit normal bundle directions tangent to 
$S_{x}\Sigma $.

We modify the operator $D_{N}$ on each Euclidean fiber of $N\Sigma \overset{%
\pi }{\rightarrow }\Sigma $ by adjusting the coordinate $r$ and function $%
\frac{1}{r}$ so that $D_{N}\ast D_{\Sigma }$ is converted to an operator on
a cylinder; see \cite[Section 6.3.2]{BKR} for the precise details. The
result is a $G$-manifold $\widetilde{M}^{j}$ with boundary $\partial 
\widetilde{M}^{j}$, a $G$-vector bundle $\widetilde{E}^{j}$, and the induced
operator $\widetilde{D}^{j}$, all of which locally agree with the original
counterparts outside $B_{\varepsilon }\left( \Sigma \right) $. We may double 
$\widetilde{M}^{j}$ along the boundary $\partial \widetilde{M}^{j}$ and
reverse the chirality of $\widetilde{E}^{j}$ as described in \cite[Ch. 9]%
{Bo-Wo}. Doubling produces a closed manifold $N^{j}$ with foliation $%
\mathcal{F}^{j}$, a foliated bundle $E^{j}$, and a first-order transversally
elliptic differential operator $D^{j}$. This process may be iterated until
all leaf closures are principal. The case where some strata have codimension 
$1$ is addressed in the following paragraphs.\label{codimOneSection}

We now give the definitions for the case when there is a minimal stratum $%
\Sigma $ of codimension $1$. Only the changes to the argument are noted.
This means that the isotropy subgroup $H$ corresponding to $\Sigma $
contains a principal isotropy subgroup of index two. If $r$ is the distance
from $\Sigma $, then $D_{N}$ has the form%
\begin{equation*}
D_{N}=Z\left( \nabla _{\partial _{r}}^{E}+\frac{1}{r}D^{S}\right) =Z\nabla
_{\partial _{r}}^{E}
\end{equation*}%
where $Z=-i\sigma \left( D_{N}\right) \left( \partial _{r}\right) $ is a
local bundle isomorphism and the map $D^{S}=0$.

In this case, there is no reason to modify the metric inside $B_{\varepsilon
}\left( \Sigma \right) $. The \textquotedblleft
desingularization\textquotedblright\ of $M$ along $\Sigma $ is the manifold
with boundary $\widetilde{M}=M\diagdown B_{\delta }\left( \Sigma \right) $
for some $0<\delta <\varepsilon $; the singular stratum is replaced by the
boundary $\partial \widetilde{M}=S_{\delta }\left( \Sigma \right) $, which
is a two-fold cover of $\Sigma $ and whose normal bundle is necessarily
oriented (via $\partial _{r}$). The double $M^{\prime }$ is identical to the
double of $\widetilde{M}$ along its boundary, and $M^{\prime }$ contains one
less stratum.

\subsection{Discussion of operator product assumption\label%
{productAssumptionSection}}

We now explain specific situations that guarantee that, after a foliated
homotopy, $D^{+}$ may be written locally as a product of operators as in (%
\ref{Dsplitting}) over the tubular neighborhhood $B_{4\varepsilon }\left(
\Sigma \right) $ over a singular stratum $\Sigma $. This demonstrates that
this assumption is not overly restrictive. We also emphasize that one might
think that this assumption places conditions on the curvature of the normal
bundle $N\Sigma $; however, this is not the case for the following reason.
The condition is on the foliated homotopy class of the principal transverse
symbol of $D$. The curvature of the bundle only effects the zeroth order
part of the symbol. For example, if $Y\rightarrow X$ is any fiber bundle
over a spin$^{c}$ manifold $X$ with fiber $F$, then a Dirac-type operator $D$
on $Y$ has the form $D=\partial _{X}\ast D_{F}+Z$, where $D_{F}$ is a family
of fiberwise Dirac-type operators, $\partial _{X}$ is the spin$^{c}$ Dirac
operator on $X$, and $Z$ is a bundle endomorphism.

First, we show that if $D^{+}$ is a transversal Dirac operator at points of $%
\Sigma $, and if either $\Sigma $ is spin$^{c}$ or its normal bundle $%
N\Sigma \rightarrow \Sigma $ is (fiberwise) spin$^{c}$, then it has the
desired form. Moreover, we also remark that certain operators, like those
resembling transversal de Rham operators, always satisfy this splitting
condition with no assumptions on $\Sigma $.

Let $N\mathcal{F}$ be normal bundle of the foliation$\mathcal{F}_{\Sigma
}=\left. \mathcal{F}\right\vert _{\Sigma }$, and let $N\Sigma $ be the
normal bundle of $\Sigma $ in $M$. Then the principal transverse symbol of $%
D^{+}$ (evaluated at $\xi \in N_{x}^{\ast }\mathcal{F}_{\Sigma }\oplus
N_{x}^{\ast }\Sigma $) at points $x\in \Sigma $ takes the form of a constant
multiple of Clifford multiplication. That is, we assume there is an action $%
c $ of $\mathbb{C}\mathrm{l}\left( N\mathcal{F}_{\Sigma }\oplus N\Sigma
\right) $ on $E$ and a Clifford connection $\nabla $ on $E$ such that the
local expression for $D$ is given by the composition%
\begin{equation*}
\Gamma \left( E\right) \overset{\nabla }{\rightarrow }\Gamma \left( E\otimes
T^{\ast }M\right) \overset{\mathrm{proj}}{\rightarrow }\Gamma \left(
E\otimes \left( N^{\ast }\mathcal{F}_{\Sigma }\oplus N^{\ast }\Sigma \right)
\right) \overset{\cong }{\rightarrow }\Gamma \left( E\otimes \left( N%
\mathcal{F}_{\Sigma }\oplus N\Sigma \right) \right) \overset{c}{\rightarrow }%
\Gamma \left( E\right) .
\end{equation*}

The principal transverse symbol $\sigma \left( D^{+}\right) $ at $\xi
_{x}\in T_{x}^{\ast }\Sigma $ is%
\begin{equation*}
\sigma \left( D^{+}\right) \left( \xi _{x}\right) =\sum_{j=1}^{q^{\prime
}}ic\left( \xi _{x}\right) :E_{x}^{+}\rightarrow E_{x}^{-}
\end{equation*}%
Suppose $N\Sigma $ is spin$^{c}$; then there exists a vector bundle $%
S=S^{+}\oplus S^{-}\rightarrow \Sigma $ that is an irreducible
representation of $\mathbb{C}\mathrm{l}\left( N\Sigma \right) $ over each
point of $\Sigma $, and we let $E^{\Sigma }=\mathrm{End}_{\mathbb{C} \mathrm{%
l}\left( N\Sigma \right) }\left( E\right) $ and have%
\begin{equation*}
E\cong S\widehat{\otimes }E^{\Sigma }
\end{equation*}%
as a graded tensor product, such that the action of $\mathbb{C}\mathrm{l}%
\left( N\mathcal{F}_{\Sigma }\oplus N\Sigma \right) \cong \mathbb{C}\mathrm{l%
}\left( N\Sigma \right) \widehat{\otimes }\mathbb{C}\mathrm{l}\left( N%
\mathcal{F}_{\Sigma }\right) $ (as a graded tensor product) on $E^{+}$
decomposes as%
\begin{equation*}
\left( 
\begin{array}{cc}
c\left( x\right) \otimes \mathbf{1} & -\mathbf{1}\otimes c\left( y\right)
^{\ast } \\ 
\mathbf{1}\otimes c\left( y\right) & c\left( x\right) ^{\ast }\otimes 
\mathbf{1}%
\end{array}%
\right) :\left( 
\begin{array}{c}
S^{+}\otimes E^{\Sigma +} \\ 
S^{-}\otimes E^{\Sigma -}%
\end{array}%
\right) \rightarrow \left( 
\begin{array}{c}
S^{-}\otimes E^{\Sigma +} \\ 
S^{+}\otimes E^{\Sigma -}%
\end{array}%
\right)
\end{equation*}%
(see \cite{ASi}, \cite{LM}). If we let the operator $\partial ^{N}$ denote
the spin$^{c}$ transversal Dirac operator on sections of $\pi ^{\ast
}S\rightarrow N\Sigma $, and let $D_{\Sigma }$ be the transversal Dirac
operator defined by the action of $\mathbb{C}\mathrm{l}\left( N\mathcal{F}%
_{\Sigma }\right) $ on $E^{\Sigma }$, then we have%
\begin{equation*}
D^{+}=\left( \partial ^{N}\ast D_{\Sigma }\right) ^{+}
\end{equation*}%
up to zero$^{\mathrm{th}}$ order terms (coming from curvature of the fiber).

The same argument works if instead we have that the bundle $N\mathcal{F}%
_{\Sigma }\rightarrow \Sigma $ is spin$^{c}$. In this case a spin$^{c}$
Dirac operator $\partial ^{\Sigma }$ on sections of a complex spinor bundle
over $\Sigma $ is transversally elliptic to the foliation $\mathcal{F}%
_{\Sigma }$, and we have a formula of the form%
\begin{equation*}
D^{+}=\left( D_{N}\ast \partial ^{\Sigma }\right) ^{+},
\end{equation*}%
again up to zeroth order terms.

Even if $N\Sigma \rightarrow \Sigma $ and $N\mathcal{F}_{\Sigma }\rightarrow
\Sigma $ are not spin$^{c}$, many other first order operators have
splittings as in Equation (\ref{Dsplitting}). For example, if $D^{+}$ is a
transversal de Rham operator from even to odd forms, then $D^{+}$ is the
product of de Rham operators in the $N\Sigma $ and $N\mathcal{F}_{\Sigma }$
directions.

In \cite{GLott}, where a formula for the basic index is derived, the
assumptions dictate that every isotropy subgroup is a connected torus, which
implies that $N\Sigma \rightarrow \Sigma $ automatically carries a vertical
almost complex structure and is thus spin$^{c}$, so that the splitting
assumption is automatically satisfied in their paper as well.

\section{The equivariant index theorem\label{equivIndexSection}}

We review some facts about equivariant index theory and in particular make
note of \cite[Theorem 9.2]{BKR}. Suppose that a compact Lie group $G$ acts
by isometries on a compact, connected Riemannian manifold $\widehat{W}$. In
the following sections of the paper, we will be particularly interested in
the case where $\widehat{W}$ is the basic manifold associated to $\left( M,%
\mathcal{F}\right) $ and $G=O\left( q\right) $. Let $E=E^{+}\oplus E^{-}$ be
a graded, $G$-equivariant Hermitian vector bundle over $\widehat{W}$. We
consider a first order $G$-equivariant differential operator $D=D^{+}:$ $%
\Gamma \left( \widehat{W},E^{+}\right) \rightarrow \Gamma \left( \widehat{W}%
,E^{-}\right) $ that is transversally elliptic, and let $D^{-}$ be the
formal adjoint of $D^{+}$.

The group $G$ acts on $\Gamma \left( \widehat{W},E^{\pm }\right) $ by $%
\left( gs\right) \left( x\right) =g\cdot s\left( g^{-1}x\right) $, and the
(possibly infinite-dimensional) subspaces $\ker \left( D^{+}\right) $ and $%
\ker \left( D^{-}\right) $ are $G$-invariant subspaces. Let $\rho
:G\rightarrow U\left( V_{\rho }\right) $ be an irreducible unitary
representation of $G$, and let $\chi _{\rho }=\mathrm{tr}\left( \rho \right) 
$ denote its character. Let $\Gamma \left( \widehat{W},E^{\pm }\right)
^{\rho }$ be the subspace of sections that is the direct sum of the
irreducible $G$-representation subspaces of $\Gamma \left( \widehat{W}%
,E^{\pm }\right) $ that are unitarily equivalent to the representation $\rho 
$. It can be shown that the extended operators 
\begin{equation*}
\overline{D}_{\rho ,s}:H^{s}\left( \Gamma \left( \widehat{W},E^{+}\right)
^{\rho }\right) \rightarrow H^{s-1}\left( \Gamma \left( \widehat{W}%
,E^{-}\right) ^{\rho }\right)
\end{equation*}%
are Fredholm and independent of $s$, so that each irreducible representation
of $G$ appears with finite multiplicity in $\ker D^{\pm }$ (see \cite{BKR}).
Let $a_{\rho }^{\pm }\in \mathbb{Z}_{\geq 0}$ be the multiplicity of $\rho $
in $\ker \left( D^{\pm }\right) $.

The study of index theory for such transversally elliptic operators was
initiated by M. Atiyah and I. Singer in the early 1970s (\cite{A}). The
virtual representation-valued index of $D$ is given by 
\begin{equation*}
\mathrm{ind}^{G}\left( D\right) :=\sum_{\rho }\left( a_{\rho }^{+}-a_{\rho
}^{-}\right) \left[ \rho \right] ,
\end{equation*}%
where $\left[ \rho \right] $ denotes the equivalence class of the
irreducible representation $\rho $. The index multiplicity is 
\begin{equation*}
\mathrm{ind}^{\rho }\left( D\right) :=a_{\rho }^{+}-a_{\rho }^{-}=\frac{1}{%
\dim V_{\rho }}\mathrm{ind}\left( \left. D\right\vert _{\Gamma \left( 
\widehat{W},E^{+}\right) ^{\rho }\rightarrow \Gamma \left( \widehat{W}%
,E^{-}\right) ^{\rho }}\right) .
\end{equation*}%
In particular, if $\rho _{0}$ is the trivial representation of $G$, then 
\begin{equation*}
\mathrm{ind}^{\rho _{0}}\left( D\right) =\mathrm{ind}\left( \left.
D\right\vert _{\Gamma \left( \widehat{W},E^{+}\right) ^{G}\rightarrow \Gamma
\left( \widehat{W},E^{-}\right) ^{G}}\right) ,
\end{equation*}%
where the superscript $G$ implies restriction to $G$-invariant sections.

There is a clear relationship between the index multiplicities and Atiyah's
equivariant distribution-valued index $\mathrm{ind}_{g}\left( D\right) $;
the multiplicities determine the distributional index, and vice versa. The
space $\Gamma \left( \widehat{W},E^{\pm }\right) ^{\rho }$ is a subspace of
the $\lambda _{\rho }$-eigenspace of $C$. The virtual character $\mathrm{ind}%
_{g}\left( D\right) $ is given by (see \cite{A}) 
\begin{eqnarray*}
\mathrm{ind}_{g}\left( D\right) &:&=\text{\textquotedblleft }\mathrm{tr}%
\left( \left. g\right\vert _{\ker D^{+}}\right) -\mathrm{tr}\left( \left.
g\right\vert _{\ker D^{-}}\right) \text{\textquotedblright } \\
&=&\sum_{\rho }\mathrm{ind}^{\rho }\left( D\right) \chi _{\rho }\left(
g\right) .
\end{eqnarray*}%
Note that the sum above does not in general converge, since $\ker D^{+}$ and 
$\ker D^{-}$ are in general infinite-dimensional, but it does make sense as
a distribution on $G$. That is, if $dg$ is the normalized, biinvariant Haar
measure on $G$, and if $\phi =\beta +\sum c_{\rho }\chi _{\rho }\in
C^{\infty }\left( G\right) $, with $\beta $ orthogonal to the subspace of
class functions on $G$, then 
\begin{eqnarray*}
\mathrm{ind}_{\ast }\left( D\right) \left( \phi \right) &=&\text{%
\textquotedblleft }\int_{G}\phi \left( g\right) ~\overline{\mathrm{ind}%
_{g}\left( D\right) }~dg\text{\textquotedblright } \\
&=&\sum_{\rho }\mathrm{ind}^{\rho }\left( D\right) \int \phi \left( g\right)
~\overline{\chi _{\rho }\left( g\right) }~dg=\sum_{\rho }\mathrm{ind}^{\rho
}\left( D\right) c_{\rho },
\end{eqnarray*}%
an expression which converges because $c_{\rho }$ is rapidly decreasing and $%
\mathrm{ind}^{\rho }\left( D\right) $ grows at most polynomially as $\rho $
varies over the irreducible representations of $G$. From this calculation,
we see that the multiplicities determine Atiyah's distributional index.
Conversely, let $\alpha :G\rightarrow U\left( V_{\alpha }\right) $ be an
irreducible unitary representation. Then 
\begin{equation*}
\mathrm{ind}_{\ast }\left( D\right) \left( \chi _{\alpha }\right)
=\sum_{\rho }\mathrm{ind}^{\rho }\left( D\right) \int \chi _{\alpha }\left(
g\right) \overline{\chi _{\rho }\left( g\right) }\,dg=\mathrm{ind}^{\alpha
}D,
\end{equation*}%
so that complete knowledge of the equivariant distributional index is
equivalent to knowing all of the multiplicities $\mathrm{ind}^{\rho }\left(
D\right) $. Because the operator $\left. D\right\vert _{\Gamma \left( 
\widehat{W},E^{+}\right) ^{\rho }\rightarrow \Gamma \left( \widehat{W}%
,E^{-}\right) ^{\rho }}$ is Fredholm, all of the indices $\mathrm{ind}%
^{G}\left( D\right) $ , $\mathrm{ind}_{g}\left( D\right) $, and $\mathrm{ind}%
^{\rho }\left( D\right) $ depend only on the stable homotopy class of the
principal transverse symbol of $D$.

The equivariant index theorem (\cite[Theorem 9.2]{BKR}) expresses $\mathrm{%
ind}^{\rho }\left( D\right) $ as a sum of integrals over the different
strata of the action of $G$ on $\widehat{W}$, and it involves the eta
invariant of associated equivariant elliptic operators on spheres normal to
the strata. The result is%
\begin{eqnarray*}
\mathrm{ind}^{\rho }\left( D\right) &=&\int_{G\diagdown \widetilde{\widehat{W%
}_{0}}}A_{0}^{\rho }\left( x\right) ~\widetilde{\left\vert dx\right\vert }%
~+\sum_{j=1}^{r}\beta \left( \Sigma _{\alpha _{j}}\right) ~, \\
\beta \left( \Sigma _{\alpha _{j}}\right) &=&\frac{1}{2\dim V_{\rho }}%
\sum_{b\in B}\frac{1}{n_{b}\mathrm{rank~}W^{b}}{\Huge (}-\eta \left(
D_{j}^{S+,b}\right) \\
&&+h\left( D_{j}^{S+,b}\right) {\Huge )}\int_{G\diagdown \widetilde{\Sigma
_{\alpha _{j}}}}A_{j,b}^{\rho }\left( x\right) ~\widetilde{\left\vert
dx\right\vert }~,
\end{eqnarray*}%
(The notation is explained in \cite{BKR}; the integrands $A_{0}^{\rho
}\left( x\right) $ and $A_{j,b}^{\rho }\left( x\right) $ are the familar
Atiyah-Singer integrands corresponding to local heat kernel supertraces of
induced elliptic operators over closed manifolds.)

\section{The basic index theorem\label{basicIndexThmSection}}

Suppose that $E$ is a foliated $\mathbb{C}\mathrm{l}\left( Q\right) $ module
with basic $\mathbb{C}\mathrm{l}\left( Q\right) $ connection $\nabla ^{E}$
over a Riemannian foliation $\left( M,\mathcal{F}\right) $. Let

\begin{equation*}
D_{b}^{E}:\Gamma _{b}\left( E^{+}\right) \rightarrow \Gamma _{b}\left(
E^{-}\right)
\end{equation*}%
be the corresponding basic Dirac operator, with basic index $\mathrm{ind}%
_{b}\left( D_{b}^{E}\right) $.

In what follows, if $U$ denotes an open subset of a stratum of $\left( M,%
\mathcal{F}\right) $, $U^{\prime }$ denotes the desingularization of $U$
very similar to that in Section \ref{desingularizationSection}, and $%
\widetilde{U}$ denotes the fundamental domain of $U$ inside $U^{\prime }$.
We assume that near each component $M_{j}$ of a singular stratum of $\left(
M,\mathcal{F}\right) $, $D_{b}^{E}$ is homotopic (through basic,
transversally elliptic operators) to the product $D_{N}\ast D_{M_{j}}$,
where $D_{N}$ is an $\mathcal{F}$-basic, first order differential operator
on a tubular neighborhood of $\Sigma _{\alpha _{j}}$ that is elliptic and
Zhas constant coefficients on the fibers and $D_{M_{j}}\ $is a global
transversally elliptic, basic, first order operator on the Riemannian
foliation $\left( M_{j},\mathcal{F}\right) $. In polar coordinates, the
fiberwise elliptic operator $D_{N}$ may be written 
\begin{equation*}
D_{N}=Z_{j}\left( \nabla _{\partial _{r}}^{E}+\frac{1}{r}D_{j}^{S}\right) ~,
\end{equation*}%
where $r$ is the distance from $M_{j}$, where $Z_{j}$ is a local bundle
isometry (dependent on the spherical parameter), the map $D_{j}^{S}$ is a
family of purely first order operators that differentiates in directions
tangent to the unit normal bundle of $M_{j}$.

\begin{theorem}
(Basic Index Theorem for Riemannian foliations)\label{basicIndexTheorem} Let 
$M_{0}$ be the principal stratum of the Riemannian foliation $\left( M,%
\mathcal{F}\right) $, and let $M_{1}$, ... , $M_{r}$ denote all the
components of all singular strata, corresponding to $O\left( q\right) $%
-isotropy types $\left[ G_{1}\right] $, ... ,$\left[ G_{r}\right] $ on the
basic manifold. With notation as in the discussion above, we have 
%\begin{multline*}
\begin{align*}
\mathrm{ind}_{b}\left( D_{b}^{E}\right) & =\int_{\widetilde{M_{0}}\diagup 
\overline{\mathcal{F}}}A_{0,b}\left( x\right) ~\widetilde{\left\vert
dx\right\vert }+\sum_{j=1}^{r}\beta \left( M_{j}\right) ~,~ \\
\beta \left( M_{j}\right) & =\frac{1}{2}\sum_{\tau }\frac{1}{n_{\tau }%
\mathrm{rank~}W^{\tau }}\left( -\eta \left( D_{j}^{S+,\tau }\right) +h\left(
D_{j}^{S+,\tau }\right) \right) \int_{\widetilde{M_{j}}\diagup \overline{%
\mathcal{F}}}A_{j,b}^{\tau }\left( x\right) ~\widetilde{\left\vert
dx\right\vert }\ ,
\end{align*}%
%
%\end{multline*}%
where the sum is over all components of singular strata and over all
canonical isotropy bundles $W^{\tau }$, only a finite number of which yield
nonzero $A_{j,b}^{\tau }$, and where

\begin{enumerate}
\item $A_{0,b}\left( x\right) $ is the Atiyah-Singer integrand, the local
supertrace of the ordinary heat kernel associated to the elliptic operator
induced from $\widetilde{D_{b}^{E}}$ (a desingularization of $D_{b}^{E}$) on
the quotient $\widetilde{M_{0}}\diagup \overline{\mathcal{F}}$, where the
bundle $E$ is replaced by the space of basic sections of over each leaf
closure;

\item $\eta \left( D_{j}^{S+,b}\right) $ and $h\left( D_{j}^{S+,b}\right) $
are the equivariant eta invariant and dimension of the equivariant kernel of
the $G_{j}$-equivariant operator $D_{j}^{S+,b}$ (defined in a similar way as
in \cite[formulas (6.3), (6.4), (6.7)]{BKR});

\item $A_{j,b}^{\tau }\left( x\right) $ is the local supertrace of the
ordinary heat kernel associated to the elliptic operator induced from $%
\left( \mathbf{1}\otimes D_{M_{j}}\right) ^{\prime }$ (blown-up and doubled
from $\mathbf{1}\otimes D_{M_{j}}$, the twist of $D_{M_{j}}$ by the
canonical isotropy bundle $W^{\tau }$ from Definition \ref%
{W-tau-twistDefinition}) on the quotient $\widetilde{M_{j}}\diagup \overline{%
\mathcal{F}}$, where the bundle is replaced by the space of basic sections
over each leaf closure; and

\item $n_{\tau }$ is the number of different inequivalent $G_{j}$%
-representation types present in a typical fiber of $W^{\tau }$.
\end{enumerate}
\end{theorem}

\begin{proof}
Using Proposition \ref{NewD}, we have%
\begin{equation*}
\mathrm{ind}_{b}\left( D_{b}^{E}\right) =\mathrm{ind}\left( \mathcal{D}%
^{G}\right) ,
\end{equation*}%
where $\mathcal{D}=\mathcal{D}^{+}$ is defined in (\ref{newDFormula}). Let $%
\Sigma _{\alpha _{1}},...,\Sigma _{\alpha _{r}}$ denote the components of
the strata of the basic manifold $\widehat{W}$ relative to the $G$-action
corresponding to the components $M_{1},...,M_{r}$. Near each $\Sigma
_{\alpha _{j}}$, we write $\mathcal{D}=D_{N}\ast D^{\alpha _{j}}$, and write 
$D_{N}=Z_{j}\left( \nabla _{\partial _{r}}^{E}+\frac{1}{r}D_{j}^{S}\right) $
in polar coordinates. By the Invariant Index Theorem \cite[Theorem 9.6]{BKR}%
, a special case of the Equivariant Index Theorem stated in the last
section, we have%
\begin{eqnarray*}
\mathrm{ind}\left( \mathcal{D}^{G}\right) &=&\int_{G\diagdown \widetilde{%
\widehat{W}_{0}}}A_{0}^{G}\left( x\right) ~\widetilde{\left\vert
dx\right\vert }~+\sum_{j=1}^{r}\beta \left( \Sigma _{\alpha _{j}}\right) ~,
\\
\beta \left( \Sigma _{\alpha _{j}}\right) &=&\frac{1}{2}\sum_{\tau \in B}%
\frac{1}{n_{\tau }\mathrm{rank~}W^{\tau }}\left( -\eta \left( D_{j}^{S+,\tau
}\right) +h\left( D_{j}^{S+,\tau }\right) \right) \int_{G\diagdown 
\widetilde{\Sigma _{\alpha _{j}}}}A_{j,\tau }^{G}\left( x\right) ~\widetilde{%
\left\vert dx\right\vert }~,
\end{eqnarray*}%
where $\tau \in B$ only if $W^{\tau }$ corresponds to irreducible isotropy
representations whose duals are present in $E^{\alpha _{j}}$, the bundle on
which $D^{\alpha _{j}}$ acts. First, $G\diagdown \widetilde{\widehat{W}_{0}}=%
\widetilde{M_{0}}\diagup \overline{\mathcal{F}}$, and $G\diagdown \widetilde{%
\Sigma _{\alpha _{j}}}=\widetilde{M_{j}}\diagup \overline{\mathcal{F}}$. By
definition, $A_{0}^{G}\left( x\right) $ is the Atiyah-Singer integrand, the
local supertrace of the ordinary heat kernel associated to the elliptic
operator induced from $\mathcal{D}^{\prime }$ (blown-up and doubled from $%
\mathcal{D}$) on the quotient $G\diagdown \widehat{W}_{0}^{\prime }$, where
the bundle $\mathcal{E}\rightarrow \widehat{W}$ is replaced by the bundle of
invariant sections of $\mathcal{E}$ over each orbit (corresponding to a
point of $G\diagdown \widetilde{\widehat{W}_{0}}$). This is precisely the
the space of basic sections of over the corresponding leaf closure (point of 
$\widetilde{M_{0}}\diagup \overline{\mathcal{F}}$), and the operator is the
same as $\widetilde{D_{b}^{E}}$ by construction. Similarly, $A_{j,\tau }^{G} 
$ is the local supertrace of the ordinary heat kernel associated to the
elliptic operator induced from $\left( \mathbf{1}\otimes D^{\alpha
_{j}}\right) ^{\prime }$ (blown-up and doubled from $\mathbf{1}\otimes
D^{\alpha _{j}}$, the twist of $D^{\alpha _{j}}$ by the canonical isotropy
bundle $W^{\tau }\rightarrow \Sigma _{\alpha _{j}}$ ) on the quotient $%
G\diagdown \Sigma _{\alpha _{j}}^{\prime }$, where the bundle is replaced by
the space of invariant sections over each orbit. Again, this part of the
formula is exactly that shown in the statement of the theorem. The
quantities $-\eta \left( D_{j}^{S+,\tau }\right) +h\left( D_{j}^{S+,\tau
}\right) $ in the equivariant and basic formulas are the same, since the
spherical operator on the normal bundle to the stratum in the basic manifold
is the same as the spherical operator defined on the normal bundle to the
stratum of the Riemannian foliation. The theorem follows. z
\end{proof}

\section{The representation-valued basic index theorem\label%
{representationValuedSection}}

In order to retain the complete information given by Atiyah's distributional
index of the transversal differential operator $\mathcal{D}$, we need to
consider the equivariant indices $\mathrm{ind}^{\rho }\left( \mathcal{D}%
\right) $ associated to any irreducible representation $\rho $ of $O\left(
q\right) $.

\begin{definition}
The \textbf{representation-valued basic index} of the transversal Dirac
operator $D_{\mathrm{tr}}^{E}$ is defined as 
\begin{equation*}
\mathrm{ind}_{b}^{\rho }\left( D_{\mathrm{tr}}^{E}\right) =\mathrm{ind}%
^{\rho }\left( \mathcal{D}\right) .
\end{equation*}
\end{definition}

Using \cite[Theorem 9.2]{BKR}, we have the following result. The proof is no
different than that of Theorem \ref{basicIndexTheorem}.

\begin{theorem}
(Representation-valued Basic Index Theorem for Riemannian foliations) \label%
{representationValuedTheorem}Let $M_{0}$ be the principal stratum of the
Riemannian foliation $\left( M,\mathcal{F}\right) $, and let $M_{1}$, ... , $%
M_{r}$ denote all the components of all singular strata, corresponding to $%
O\left( q\right) $-isotropy types $\left[ G_{1}\right] $, ... ,$\left[ G_{r}%
\right] $ on the basic manifold. With notation as in the previous section,
we have 
\begin{align*}
\mathrm{ind}_{b}^{\rho }\left( D_{\mathrm{tr}}^{E}\right) &=\int_{%
\widetilde {M_{0}}\diagup \overline{\mathcal{F}}}A_{0}^{\rho }\left(
x\right) ~\widetilde{\left\vert dx\right\vert }+\sum_{j=1}^{r}\beta \left(
M_{j}\right) ~, \\
\beta \left( M_{j}\right) &= \frac{1}{2}\sum_{\tau }\frac{1}{n_{\tau } 
\mathrm{rank~}W^{\tau }} \left( - \eta \left( D_{j}^{S+,\tau }\right) +
h\left( D_{j}^{S+,\tau }\right) \right) \int_{\widetilde{M_{j}}\diagup 
\overline{\mathcal{F}}}A_{j,\tau }^{\rho }\left( x\right) ~\widetilde{%
\left\vert dx\right\vert },
\end{align*}%
where $A_{0}^{\rho }\left( x\right) $ and $A_{j,\tau }^{\rho }\left(
x\right) $ are the local Atiyah-Singer integrands of the operators induced
on the leaf closure spaces by extracting the sections of type $\rho $ from $%
\widetilde{M_{0}}$ and $\widetilde{M_{j}}$.
\end{theorem}

\section{The basic index theorem for foliations given by suspension\label%
{suspensionSection}}

One class of examples of Riemannian foliations are those constructed by
suspensions. Let $X$ be a closed manifold with fundamental group $\pi
_{1}\left( X\right) ,$ which acts on the universal cover $\widetilde{X}$ by
deck transformations. Let $\phi :\pi _{1}\left( X\right) \rightarrow \mathrm{%
Isom}\left( Y\right) $ be a homomorphism to the group of isometries of a
closed Riemannian manifold $Y$. The suspension is defined to be%
\begin{equation*}
X\times _{\phi }Y=\widetilde{X}\times Y\diagup \sim ,
\end{equation*}%
where the equivalence relation is defined by $\left( x,y\right) \sim \left(
x\cdot g^{-1},\phi \left( g\right) y\right) $ for any $g\in \pi _{1}\left(
X\right) $. The foliation $\mathcal{F}$ associated to this suspension is
defined by the $\widetilde{X}$-parameter submanifolds, so that $T\mathcal{F}$
agrees with $T\widetilde{X}$ over each fundamental domain of $X\times _{\phi
}Y$ in $\widetilde{X}\times Y$. This foliation is Riemannian, with
transverse metric given by the metric on $Y$. A transversally-elliptic
operator that preserves the foliation is simply an elliptic operator $D^{Y}$
on $Y$ that is $G$-equivariant, where $G=\phi \left( \pi _{1}\left( X\right)
\right) \subset \mathrm{Isom}\left( Y\right) $. It follows that $D^{Y}$ is
also equivariant with respect to the action of the closure $\overline{G}$, a
compact Lie group. Then we have that the basic index satisfies 
\begin{equation*}
\mathrm{ind}_{b}\left( D_{b}^{Y}\right) =\mathrm{ind}\left( \left(
D^{Y}\right) ^{\overline{G}}\right) .
\end{equation*}%
We wish to apply the basic index theorem to this example. Observe that the
strata of the foliation $\mathcal{F}$ are determined by the strata of the $%
\overline{G}$-action on $Y$. Precisely, if $\Sigma _{\alpha _{1}},...,\Sigma
_{\alpha _{r}}$ are the components of the strata of $Y$ relative to $%
\overline{G}$, then each%
\begin{equation*}
M_{j}=\widetilde{X}\times \Sigma _{\alpha _{j}}\diagup \sim
\end{equation*}%
is a component of a stratum of the foliation $\left( X\times _{\phi }Y,%
\mathcal{F}\right) $. Similarly, the desingularizations of the foliation
correspond exactly to the desingularizations of the group action in the
Equivariant Index Theorem (\cite{BKR}), applied to the $\overline{G}$ action
on $Y$. By the basic index theorem,%
\begin{multline*}
\mathrm{ind}_{b}\left( D_{b}^{Y}\right) =\int_{\widetilde{M_{0}}\diagup 
\overline{\mathcal{F}}}A_{0,b}\left( x\right) ~\widetilde{\left\vert
dx\right\vert }+\sum_{j=1}^{r}\beta \left( M_{j}\right) ~ =\int_{\widetilde{%
Y_{0}}\diagup \overline{G}}A_{0}^{G}\left( x\right) ~\widetilde{\left\vert
dx\right\vert }+\sum_{j=1}^{r}\beta \left( M_{j}\right) ,
\end{multline*}%
where $A_{0}^{G}\left( x\right) $ is the Atiyah-Singer integrand of the
operator $D^{Y}$ on the (blown up) quotient of the principal stratum of the $%
\overline{G}$-action, where the bundle is the space of invariant sections on
the corresponding orbit. Similarly, the singular terms $\beta \left(
M_{j}\right) $ are exactly the same as those in the Equivariant Index
Theorem, applied to the $\overline{G}$ action on $Y$. Thus, the basic index
theorem gives precisely the same formula as the Equivariant Index Theorem
calculating the index $\mathrm{ind}\left( \left( D^{Y}\right) ^{\overline{G}%
}\right) $.

We remark that in this particular case, the basic index may be calculated in
an entirely different way, using the Atiyah-Segal fixed point formula for $%
\overline{G}$-equivariant elliptic operators (see \cite{ASe}). Their formula
is a formula for $\mathrm{ind}_{g}\left( D^{Y}\right) $, the difference of
traces of the action of $g\in \overline{G}$ on $\ker \left( D^{Y}\right) $
and $\ker \left( D^{Y\ast }\right) $, and the answer is an integral $%
\int_{Y^{g}}\alpha _{g}$ of characteristic classes over the fixed point set $%
Y^{g}\subset Y$ of the element $g$. To extract the invariant part of this
index, we would need to calculate%
\begin{equation*}
\mathrm{ind}\left( \left( D^{Y}\right) ^{\overline{G}}\right) =\int_{%
\overline{G}}\mathrm{ind}_{g}\left( D^{Y}\right) dg=\int_{\overline{G}%
}\left( \int_{Y^{g}}\alpha _{g}\right) dg,
\end{equation*}%
where $dg$ is the normalized Haar measure. Since the fixed point set changes
with $g$, the integral above could not be evaluated as above. However, if $%
\overline{G}$ is connected, we could use the Weyl integration formula to
change the integral to an integral over a maximal torus $T$, and we could
replace $Y^{g}$ with the fixed point set $Y^{T}$, since for generic $g\in T$%
, $Y^{g}=Y^{T}$. Moreover, if $G$ is not connected, one may construct a
suspension $Y^{\prime }$ of the manifold on which a larger connected group $%
G^{\prime }$ acts such that $G^{\prime }\diagdown Y^{\prime }=G\diagdown Y$.

\section{An example of transverse signature\label{transverseSignSection}}

In this section we give an example of a transverse signature operator that
arises from an $S^{1}$ action on a $5$-manifold. This is essentially a
modification of an example from \cite[pp. 84ff]{A}, and it illustrates the
fact that the eta invariant term may be nonzero. Let $Z^{4}$ be a closed,
oriented, $4$-dimensional Riemannian manifold on which $\mathbb{Z}_{p}$ ($p$
prime $>2$) acts by isometries with isolated fixed points $x_{i}$, $%
i=1,...,N $. Let $M=Z^{4}\times _{\mathbb{Z}_{p}}S^{1}$, where $\mathbb{Z}%
_{p}$ acts on $S^{1}$ by rotation by multiples of $\frac{2\pi }{p}$. Then $%
S^{1}$ acts on $M$, and $M\diagup S^{1}\cong Z^{4}\diagup \mathbb{Z}_{p}$.

Next, let $D^{+}$ denote the signature operator $d+d^{\ast }$ from self-dual
to anti-self-dual forms on $Z^{4}$; this induces a transversally elliptic
operator (also denoted by $D^{+}$). Then the $S^{1}$-invariant index of $%
D^{+}$ satisfies%
\begin{equation*}
\mathrm{ind}^{\rho _{0}}\left( D^{+}\right) =\mathrm{Sign}\left( M\diagup
S^{1}\right) =\mathrm{Sign}\left( Z^{4}\diagup \mathbb{Z}_{p}\right) .
\end{equation*}%
By the Invariant Index Theorem \cite[Theorem 9.6]{BKR} 
%(Corollary \ref{InvariantIndexTheorem}) 
and the fact that the Atiyah-Singer integrand is the Hirzebruch $L$%
-polynomial $\frac{1}{3}p_{1}$,%
\begin{eqnarray*}
\mathrm{ind}^{\rho _{0}}\left( D\right) &=&\frac{1}{3}\int_{\widetilde{M}%
\diagup S^{1}}p_{1}~ \\
&&+\frac{1}{2}\sum_{j=1}^{N}\left( -\eta \left( D_{j}^{S+,\rho _{0}}\right)
+h\left( D_{j}^{S+,\rho _{0}}\right) \right) ,
\end{eqnarray*}%
where each $D_{i}^{S+,\rho _{0}}$ is two copies of the boundary signature
operator 
\begin{equation*}
B=\left( -1\right) ^{p}\left( \ast d-d\ast \right)
\end{equation*}%
on $2l$-forms ($l=0,1$) on the lens space $S^{3}\diagup \mathbb{Z}_{p}$. We
have $h\left( D_{j}^{S+,\rho _{0}}\right) =2h\left( B\right) =2$
(corresponding to constants), and in \cite{APS2} the eta invariant is
explicitly calculated to be%
\begin{equation*}
\eta \left( D_{j}^{S+,\rho _{0}}\right) =2\eta \left( B\right) =-\frac{2}{p}%
\sum_{k=1}^{p-1}\cot \left( \frac{km_{j}\pi }{p}\right) \cot \left( \frac{%
kn_{j}\pi }{p}\right) ,
\end{equation*}%
where the action of the generator $\zeta $ of $\mathbb{Z}_{p}$ on $S^{3}$ is%
\begin{equation*}
\zeta \cdot \left( z_{1},z_{2}\right) =\left( e^{\frac{2m_{j}\pi i}{p}%
}z_{1},e^{\frac{2n_{j}\pi i}{p}}z_{2}\right) ,
\end{equation*}%
with $\left( m_{j},p\right) =\left( n_{j},p\right) =1$. Thus,%
\begin{equation*}
\mathrm{Sign}\left( M\diagup S^{1}\right) =\frac{1}{3}\int_{\widetilde{Z^{4}}%
\diagup \mathbb{Z}_{p}}p_{1}+\frac{1}{p}\sum_{j=1}^{N}\sum_{k=1}^{p-1}\cot
\left( \frac{km_{j}\pi }{p}\right) \cot \left( \frac{kn_{j}\pi }{p}\right) +N
\end{equation*}%
Note that in \cite[pp. 84ff]{A} it is shown that%
\begin{equation*}
\mathrm{Sign}\left( M\diagup S^{1}\right) =\frac{1}{3}\int_{Z^{4}\diagup 
\mathbb{Z}_{p}}p_{1}+\frac{1}{p}\sum_{j=1}^{N}\sum_{k=1}^{p-1}\cot \left( 
\frac{km_{j}\pi }{p}\right) \cot \left( \frac{kn_{j}\pi }{p}\right) ,
\end{equation*}%
which demonstrates that 
\begin{equation*}
\frac{1}{3}\int_{Z^{4}\diagup \mathbb{Z}_{p}}p_{1}-\frac{1}{3}\int_{%
\widetilde{Z^{4}}\diagup \mathbb{Z}_{p}}p_{1}=N,
\end{equation*}%
illustrating the difference between the blowup $\widetilde{M}$ and the
original $M$.

\section{The Basic Euler characteristic\label{euler}}

\subsection{The Basic Gauss-Bonnet Theorem}

Suppose that a smooth, closed manifold $M$ is endowed with a smooth
foliation $\mathcal{F}$.

In the theorem that follows, we express the basic Euler characteristic in
terms of the ordinary Euler characteristic, which in turn can be expressed
in terms of an integral of curvature. We extend the Euler characteristic
notation $\chi \left( Y\right) $ for $Y$ any open (noncompact without
boundary) or closed (compact without boundary) manifold to mean%
\begin{equation*}
\chi \left( Y\right) =%
\begin{array}{ll}
\chi \left( Y\right) & \text{if }Y\text{ is closed} \\ 
\chi \left( 1\text{-point compactification of }Y\right) -1~ & \text{if }Y%
\text{ is open}%
\end{array}%
\end{equation*}%
Also, if $\mathcal{L}$ is a flat foliated line bundle over a Riemannian
foliation $\left( X,\mathcal{F}\right) $, we define the basic Euler
characteristic $\chi \left( X,\mathcal{F},\mathcal{L}\right) $ as before,
using the basic cohomology groups with coefficients in the line bundle $%
\mathcal{L}$.

\begin{theorem}
(Basic Gauss-Bonnet Theorem, announced in \cite{RiLodz}) \label%
{BasicGaussBonnet}Let $\left( M,\mathcal{F}\right) $ be a Riemannian
foliation. Let $M_{0}$,..., $M_{r}$ be the strata of the Riemannian
foliation $\left( M,\mathcal{F}\right) $, and let $\mathcal{O}_{M_{j}\diagup 
\overline{\mathcal{F}}}$ denote the orientation line bundle of the normal
bundle to $\overline{\mathcal{F}}$ in $M_{j}$. Let $L_{j}$ denote a
representative leaf closure in $M_{j}$. With notation as above, the basic
Euler characteristic satisfies 
\begin{equation*}
\chi \left( M,\mathcal{F}\right) =\sum_{j}\chi \left( M_{j}\diagup \overline{%
\mathcal{F}}\right) \chi \left( L_{j},\mathcal{F},\mathcal{O}_{M_{j}\diagup 
\overline{\mathcal{F}}}\right) .
\end{equation*}
\end{theorem}

\begin{remark}
In \cite[Corollary 1]{GLott}, they show that in special cases the only term
that appears is one corresponding to a most singular stratum.
\end{remark}

\subsubsection{Proof using the basic Hopf index theorem}

In this section, we prove the basic Gauss-Bonnet Theorem using the Hopf
index theorem for Riemannian foliations (\cite{BePaRi}).

To find a topological formula for the basic index, we first construct a
basic, normal, $\mathcal{F}^{\prime }$-nondegenerate vector field $V$ on $%
\left( M,\mathcal{F}\right) $ and then compute the basic Euler
characteristic from this information. The formula from the main theorem in 
\cite{BePaRi} is 
\begin{equation*}
\chi \left( M,\mathcal{F}\right) =\sum_{L\text{ critical}}\mathrm{ind}\left(
V,L\right) \chi \left( L,\mathcal{F},\mathcal{O}_{L}\right) .
\end{equation*}%
We construct the vector field as follows. First, starting with $i=1$ (where
the holonomy is largest, where $M_{i}\diagup \overline{\mathcal{F}}$ is a
closed manifold), we triangulate $\overline{M_{i}}\diagup \overline{\mathcal{%
F}}\cong \overline{\widehat{W}\left( G_{i}\right) }\diagup G$, without
changing the triangulation of $\left( \overline{M_{i}}\diagup \overline{%
\mathcal{F}}\right) \setminus M_{i}\diagup \overline{\mathcal{F}}$ (to
construct the triangulation, we may first apply the exponential map of $%
M_{i} $ to the normal space to a specific leaf closure of $M_{i}$ and extend
the geodesics to the cut locus, and so on). The result is a triangulation of 
$M\diagup \overline{\mathcal{F}}$ that restricts to a triangulation of each $%
\overline{M_{i}}\diagup \overline{\mathcal{F}}$. Next, we assign the value $%
0 $ to each vertex of the triangulation and the value $k$ to a point on the
interior of each $k$-cell, and we smoothly extend this function to a smooth
basic Morse function on all of $M$ whose only critical leaf closures are
each of the points mentioned above. The gradient of this function is a a
basic, normal, $\mathcal{F}^{\prime }$-nondegenerate vector field $V$ on $M$%
. Thus, letting $L_{k}$ denote a leaf closure corresponding to the value $k$%
, 
\begin{eqnarray*}
\chi \left( M,\mathcal{F}\right) &=&\sum_{L\text{ critical}}\mathrm{ind}%
\left( V,L\right) \chi \left( L,\mathcal{F},\mathcal{O}_{L}\right) \\
&=&\sum_{k}\sum_{L_{k}}\left( -1\right) ^{k}\chi \left( L_{k},\mathcal{F},%
\mathcal{O}_{L}\right) \\
&=&\sum_{i}\chi \left( M_{i}\diagup \overline{\mathcal{F}}\right) \chi
\left( L_{i},\mathcal{F},\mathcal{O}_{L_{i}}\right) \\
&=&\sum_{i}\chi \left( M_{i}\diagup \overline{\mathcal{F}}\right) \chi
\left( L_{i},\mathcal{F},\mathcal{O}_{M_{i}\diagup \overline{\mathcal{F}}%
}\right)
\end{eqnarray*}%
where $L_{i}$ denotes a representative leaf closure of $M_{i}$, and $%
\mathcal{O}_{L_{i}}$ denotes its \textquotedblleft negative direction
orientation bundle\textquotedblright , which by the definition of the vector
field is isomorphic to the orientation bundle $\mathcal{O}_{M_{i}\diagup 
\overline{\mathcal{F}}}$ of $T\left( M_{i}\diagup \overline{\mathcal{F}}%
\right) $.

\subsubsection{Proof using the Basic Index Theorem}

In this section, we prove the basic Gauss-Bonnet Theorem using the Basic
Index Theorem (Theorem \ref{basicIndexTheorem}).

As explained in Section \ref{diracex} we wish to compute $\mathrm{ind}%
_{b}\left( D_{b}^{\prime }\right) =\mathrm{ind}_{b}\left( D_{b}\right) $,
with 
\begin{equation*}
D_{b}^{\prime }=d+\delta _{b};~D_{b}=D_{b}^{\prime }-\frac{1}{2}\left(
\kappa _{b}\wedge +\kappa _{b}\lrcorner \right) .
\end{equation*}%
Let $M_{0}$ be the principal stratum of the Riemannian foliation $\left( M,%
\mathcal{F}\right) $, and let $M_{1}$, ... , $M_{r}$ denote all the
components of all singular strata, corresponding to $O\left( q\right) $%
-isotropy types $\left[ G_{1}\right] $, ... ,$\left[ G_{r}\right] $ on the
basic manifold. At each $M_{j}$, we may write the basic de Rham operator (up
to lower order perturbations) as%
\begin{equation*}
D_{b}=D_{N_{j}}\ast D_{M_{j}}~,
\end{equation*}%
where $D_{N_{j}}$ is in fact the de Rham operator on the vertical forms, and 
$D_{M_{j}}$ is the basic de Rham operator on $\left( M_{j}~,\left. \mathcal{F%
}\right\vert _{M_{j}}\right) $. Further, the spherical operator $D_{j}^{S}$
in the main theorem is simply 
\begin{equation*}
D_{j}^{S}=-c\left( \partial _{r}\right) \left( d+d^{\ast }\right)
^{S},~c\left( \partial _{r}\right) =dr\wedge -dr\lrcorner ~,
\end{equation*}%
where $\left( d+d^{\ast }\right) ^{S}$ is a vector-valued de Rham operator
on the sphere (normal to $M_{j}$) and $r$ is the radial distance from $M_{j}$
. We performed a similar calculation in \cite[Section 10.2]{BKR}, and the
results are that $\eta \left( D_{j}^{S+,\sigma }\right) =0$ for all $G_{j}$%
-representation types $\left[ \sigma \right] $ and 
\begin{equation}
h\left( D_{j}^{S+,\sigma }\right) =\left\{ 
\begin{array}{ll}
2 & \text{if }\sigma =\mathbf{1}\text{ and }G_{j}\text{ preserves orientation%
} \\ 
1 & \text{if }\sigma =\mathbf{1}\text{ and }G_{j}\text{ does not preserve
orientation} \\ 
1 & \text{if }\sigma =\xi _{G_{j}}\text{ and }G_{j}\text{ does not preserve
orientation} \\ 
0 & \text{otherwise}%
\end{array}%
\right.  \label{hFormulasEuler}
\end{equation}%
Here, if some elements of $G_{j}$ reverse orientation of the normal bundle,
then $\xi _{G_{j}}$ denotes the relevant one-dimensional representation of $%
G_{j}$ as $\pm 1$. The orientation line bundle $\mathcal{O}_{M_{i}\diagup 
\overline{\mathcal{F}}}\rightarrow M_{j}$ of the normal bundle to $M_{j}$ is
a pointwise representation space for the representation $\xi _{G_{j}}$.
After pulling back to and pushing forward to the basic manifold, it is the
canonical isotropy $G$-bundle $W^{b}$ corresponding to $\left( j,\left[ \xi
_{G_{j}}\right] \right) $. We may also take it to be a representation bundle
for the trivial $G_{j}$-representation $\mathbf{1}$ (although the trivial
line bundle is the canonical one). The Basic Index Theorem takes the form 
\begin{multline*}
\mathrm{ind}_{b}\left( D_{b}^{E}\right) =\int_{\widetilde{M_{0}}\diagup 
\overline{\mathcal{F}}}A_{0,b}\left( x\right) ~\widetilde{\left\vert
dx\right\vert }+\sum_{j=1}^{r}\beta \left( M_{j}\right) ~ \\
\beta \left( M_{j}\right) =\frac{1}{2}\sum_{j}\left( h\left( D_{j}^{S+,\xi
_{G_{j}}}\right) ++h\left( D_{j}^{S+,\mathbf{1}}\right) \right) \int_{%
\widetilde{M_{j}}\diagup \overline{\mathcal{F}}}A_{j,b}\left( x,\mathcal{O}%
_{M_{j}\diagup \overline{\mathcal{F}}}\right) ~\widetilde{\left\vert
dx\right\vert } \\
=\sum_{j}\int_{\widetilde{M_{j}}\diagup \overline{\mathcal{F}}}A_{j,b}\left(
x,\mathcal{O}_{M_{j}\diagup \overline{\mathcal{F}}}\right) ~\widetilde{%
\left\vert dx\right\vert }.
\end{multline*}%
We rewrite $\int_{\widetilde{M_{j}}\diagup \overline{\mathcal{F}}%
}A_{j,b}\left( x,\mathcal{O}_{M_{j}\diagup \overline{\mathcal{F}}}\right) ~%
\widetilde{\left\vert dx\right\vert }$ as $\int_{\widetilde{M_{j}}%
}K_{j}\left( x,\mathcal{O}_{M_{j}\diagup \overline{\mathcal{F}}}\right) ~%
\widetilde{\left\vert dx\right\vert }$ before taking it to the quotient. We
see that $K_{j}\left( x,\mathcal{O}_{M_{j}\diagup \overline{\mathcal{F}}%
}\right) $ is the Gauss-Bonnet integrand on the desingularized stratum $%
\widetilde{M_{j}}$, restricted to $\mathcal{O}_{M_{j}\diagup \overline{%
\mathcal{F}}}$-twisted basic forms. The result is the relative Euler
characteristic$\chi \left( L_{j},\mathcal{F},\mathcal{O}_{M_{j}\diagup 
\overline{\mathcal{F}}}\right) $ 
\begin{equation*}
\int_{\widetilde{M_{j}}}K_{j}\left( x,\mathcal{O}_{M_{j}\diagup \overline{%
\mathcal{F}}}\right) ~\widetilde{\left\vert dx\right\vert }=\chi \left( 
\overline{M_{j}},\text{lower strata},\mathcal{F},\mathcal{O}_{M_{j}\diagup 
\overline{\mathcal{F}}}\right) ,
\end{equation*}%
Here, the relative basic Euler characteristic is defined for $X$ a closed
subset of a manifold $Y$ as $\chi \left( Y,X,\mathcal{F},\mathcal{V}\right)
=\chi \left( Y,\mathcal{F},\mathcal{V}\right) -\chi \left( X,\mathcal{F},%
\mathcal{V}\right) $, which is also the alternating sum of the dimensions of
the relative basic cohomology groups with coefficients in a complex vector
bundle $\mathcal{V}\rightarrow Y$. Since $M_{j}$ is a fiber bundle over $%
M_{j}\diagup \overline{\mathcal{F}}$ with fiber $L_{j}$ (a representative
leaf closure), we have%
\begin{equation*}
\int_{\widetilde{M_{j}}}K_{j}\left( x,\mathcal{O}_{M_{j}\diagup \overline{%
\mathcal{F}}}\right) ~\widetilde{\left\vert dx\right\vert }=\chi \left(
L_{j},\mathcal{F},\mathcal{O}_{M_{j}\diagup \overline{\mathcal{F}}}\right)
\chi \left( \overline{M_{j}}\diagup \overline{\mathcal{F}},\text{lower strata%
}\diagup \overline{\mathcal{F}}\right) ,
\end{equation*}%
by the formula for the Euler characteristic on fiber bundles, which extends
naturally to the current situation. The Basic Gauss-Bonnet Theorem follows.

\subsubsection{The representation-valued basic Euler characteristic}

Using the Representation-valued Basic Index Theorem (Theorem \ref%
{representationValuedTheorem}), we may use the arguments in the previous
section to derive a formula for the basic Euler characteristic of basic
forms twisted by a representation of $O\left( q\right) $. Since the proof is
nearly the same, we simply state the result.

\begin{theorem}
(Representation-valued Basic Gauss-Bonnet Theorem) \label%
{representationValuedGaussBonnetThm}Let $\left( M,\mathcal{F}\right) $ be a
Riemannian foliation. Let $M_{0}$,..., $M_{r}$ be the strata of the
Riemannian foliation $\left( M,\mathcal{F}\right) $, and let $\mathcal{O}%
_{M_{j}\diagup \overline{\mathcal{F}}}$ denote the orientation line bundle
of the normal bundle to $\overline{\mathcal{F}}$ in $M_{j}$. Let $L_{j}$
denote a representative leaf closure in $M_{j}$. For $\left( X,\mathcal{F}%
_{X}\right) $ a Riemannian foliation of codimension $q$, let $\chi ^{\rho
}\left( X,\mathcal{F}_{X},\mathcal{V}\right) $ denote the index of the basic
de Rham operator twisted by a representation $\rho :O\left( q\right)
\rightarrow U\left( V_{\rho }\right) $ with values in the flat line bundle$%
\mathcal{V}$. Then the basic Euler characteristic satisfies 
\begin{equation*}
\chi ^{\rho }\left( M,\mathcal{F}\right) =\sum_{j}\chi \left( M_{j}\diagup 
\overline{\mathcal{F}}\right) \chi ^{\rho }\left( L_{j},\mathcal{F},\mathcal{%
O}_{M_{j}\diagup \overline{\mathcal{F}}}\right) .
\end{equation*}
\end{theorem}

\subsection{Examples of the basic Euler characteristic}

\vspace{1pt}In addition to the examples in this section, we refer the reader
to \cite{HabRi2}, where in some nontaut Riemannian foliations, the basic
Euler characteristic and basic cohomology groups and twisted basic
cohomology groups are computed using the theorems in this paper.

The first example is a codimension $2$ foliation on a 3-manifold. Here, $%
O(2) $ acts on the basic manifold, which is homeomorphic to a sphere. In
this case, the principal orbits have isotropy type $\left( \{e\}\right) $,
and the two fixed points obviously have isotropy type $\left( O(2)\right) $.
In this example, the isotropy types correspond precisely to the
infinitesimal holonomy groups.

\begin{example}
\label{rotation} (This example is taken from \cite{Ri} and \cite{Ri2}.)
Consider the one dimensional foliation obtained by suspending an irrational
rotation on the standard unit sphere $S^{2}$. \ On $S^{2}$ we use the
cylindrical coordinates $\left( z,\theta \right) $, related to the standard
rectangular coordinates by $x^{\prime }=\sqrt{\left( 1-z^{2}\right) }\cos
\theta $, $y^{\prime }=\sqrt{\left( 1-z^{2}\right) }\sin \theta $, $%
z^{\prime }=z$. \ Let $\alpha $ be an irrational multiple of $2\pi $, and
let the three--manifold $M=S^{2}\times \left[ 0,1\right] /\sim $, where $%
\left( z,\theta ,0\right) \sim \left( z,\theta +\alpha ,1\right) $. \ Endow $%
M$ with the product metric on $T_{z,\theta ,t}M\cong T_{z,\theta
}S^{2}\times T_{t}\mathbb{R}$. \ Let the foliation $\mathcal{F}$ be defined
by the immersed submanifolds $L_{z,\theta }=\cup _{n\in \mathbb{Z}}\left\{
z\right\} \times \left\{ \theta +\alpha \right\} \times \left[ 0,1\right] $
(not unique in $\theta $). \ The leaf closures $\overline{L}_{z}$ for $|z|<1$
are two dimensional, and the closures corresponding to the poles ($z=\pm 1$)
are one dimensional.\newline
The stratification of $\left( M,\mathcal{F}\right) $ is $M\left(
H_{1}\right) \coprod M\left( H_{2}\right) $, where $M\left( H_{1}\right) $
is the union of the two ``polar'' leaves ($z=\pm 1$), and $M\left(
H_{2}\right) $ is the complement of $M\left( H_{1}\right) $. Note that each
orientation bundle $\mathcal{O}_{M\left( H_{i}\right) \diagup \overline{%
\mathcal{F}}}$ is trivial. Next, $\chi \left( M\left( H_{2}\right) \diagup 
\overline{\mathcal{F}}\right) =\chi \left( \text{open interval}\right) =-1$,
and $\chi \left( M\left( H_{1}\right) \diagup \overline{\mathcal{F}}\right)
=\chi \left( \text{disjoint union of two points}\right) =2$. Observe that $%
\chi \left( L_{1},\mathcal{F},\mathcal{O}_{M\left( H_{1}\right) \diagup 
\overline{\mathcal{F}}}\right) =\chi \left( L_{1},\mathcal{F}\right) $ $%
=\chi \left( S^{1},S^{1}\right) =1$. However, $\chi \left( L_{2},\mathcal{F},%
\mathcal{O}_{M\left( H_{2}\right) \diagup \overline{\mathcal{F}}}\right)
=\chi \left( L_{2},\mathcal{F}\right) $ $=0$, since every such leaf closure
is a flat torus, on which the foliation restricts to be the irrational flow
and since the vector field $\partial _{\theta }$ is basic, nonsingular, and
orthogonal to the foliation on this torus. By our theorem, we conclude that 
\begin{eqnarray*}
\chi \left( M,\mathcal{F}\right) &=&\sum_{i}\chi \left( M\left( H_{i}\right)
\diagup \overline{\mathcal{F}}\right) \chi \left( L_{i},\mathcal{F},\mathcal{%
\ \ \ \ O}_{M\left( H_{i}\right) \diagup \overline{\mathcal{F}}}\right) \\
&=&2\cdot 1+\left( -1\right) \cdot 0=2.
\end{eqnarray*}

We now directly calculate the Euler characteristic of this foliation. Since
the foliation is taut, the standard Poincare duality works \cite{KT3} \cite%
{KT4} , and $H_{b}^{0}\left( M\right) \cong H_{b}^{2}\left( M\right) \cong 
\mathbb{R}$ . It suffices to check the dimension $h^{1}$ of the cohomology
group $H_{b}^{1}\left( M\right) $. \ Then the basic Euler characteristic is $%
\chi \left( M,\mathcal{F}\right) =1-h^{1}+1=2-h^{1}$. \ Smooth basic
functions are of the form $f\left( z\right) $, where $f\left( z\right) $ is
smooth in $z$ for $-1<z<1$ and is of the form $f\left( z\right) =f_{1}\left(
1-z^{2}\right) $ near $z=1$ for a smooth function $f_{1}$ and is of the form 
$f\left( z\right) =f_{2}\left( 1-z^{2}\right) $\ near $z=-1$ for a smooth
function $f_{2}$. Smooth basic one forms are of the form $\alpha =g\left(
z\right) dz+k\left( z\right) d\theta $, where $g\left( z\right) $ and $\
k\left( z\right) $ are smooth functions for $-1<z<1$ and satisfy 
\begin{eqnarray}
g\left( z\right) &=&g_{1}\left( 1-z^{2}\right) \,\,\text{and}  \notag \\
k\left( z\right) &=&\left( 1-z^{2}\right) k_{1}\left( 1-z^{2}\right)
\label{condition}
\end{eqnarray}%
near $z=1$ and 
\begin{eqnarray*}
g\left( z\right) &=&g_{2}\left( 1-z^{2}\right) \,\,\text{and} \\
k\left( z\right) &=&\left( 1-z^{2}\right) k_{2}\left( 1-z^{2}\right)
\end{eqnarray*}%
near $z=-1$ for smooth functions $g_{1},g_{2},k_{1},k_{2}$~. A simple
calculation shows that %\linebreak 
$\ker d^{1}=\mathrm{im}\,d^{0}$, so that $h^{1}=0$. Thus, $\chi \left( M,%
\mathcal{F}\right) =2$. This example shows that the orbit space can be
dimension 1 (odd) and yet have nontrivial index.
\end{example}

The next example is a codimension $3$ Riemannian foliation for which all of
the infinitesimal holonomy groups are trivial; moreover, the leaves are all
simply connected. There are leaf closures of codimension 2 and codimension
1. The codimension 1 leaf closures correspond to isotropy type $(e)$ on the
basic manifold, and the codimension 2 leaf closures correspond to an
isotropy type $(O(2))$ on the basic manifold. In some sense, the isotropy
type measures the holonomy of the leaf closure in this case.

\begin{example}
This foliation is a suspension of an irrational rotation of $S^{1}$ composed
with an irrational rotation of $S^{2}$ on the manifold $S^{1}\times S^{2}$.
As in Example~\ref{rotation}, on $S^{2}$ we use the cylindrical coordinates $%
\left( z,\theta \right) $, related to the standard rectangular coordinates
by $x^{\prime }=\sqrt{\left( 1-z^{2}\right) }\cos \theta $, $y^{\prime }=%
\sqrt{\left( 1-z^{2}\right) }\sin \theta $, $z^{\prime }=z$. \ Let $\alpha $
be an irrational multiple of $2\pi $, and let $\beta $ be any irrational
number. We consider the four--manifold $M=S^{2}\times \left[ 0,1\right]
\times \left[ 0,1\right] /\sim $, where $\left( z,\theta ,0,t\right) \sim
\left( z,\theta ,1,t\right) $, $\left( z,\theta ,s,0\right) \sim \left(
z,\theta +\alpha ,s+\beta \mod1,1\right) $. Endow $M$ with the product
metric on $T_{z,\theta ,s,t}M\cong T_{z,\theta }S^{2}\times T_{s}\mathbb{R}%
\times T_{t}\mathbb{R}$. Let the foliation $\mathcal{F}$ be defined by the
immersed submanifolds $L_{z,\theta ,s}=\cup _{n\in \mathbb{Z}}\left\{
z\right\} \times \left\{ \theta +\alpha \right\} \times \left\{ s+\beta
\right\} \times \left[ 0,1\right] $ (not unique in $\theta $ or $s$). The
leaf closures $\overline{L}_{z}$ for $|z|<1$ are three--dimensional, and the
closures corresponding to the poles ($z=\pm 1$) are two--dimensional. The
basic forms in the various dimensions are: 
\begin{eqnarray*}
\Omega _{b}^{0} &=&\left\{ f\left( z\right) \right\} \\
\Omega _{b}^{1} &=&\left\{ g_{1}\left( z\right) dz+\left( 1-z^{2}\right)
g_{2}(z)d\theta +g_{3}\left( z\right) ds\right\} \\
\Omega _{b}^{2} &=&\left\{ h_{1}\left( z\right) dz\wedge d\theta +\left(
1-z^{2}\right) h_{2}(z)d\theta \wedge ds+h_{3}\left( z\right) dz\wedge
ds\right\} \\
\Omega _{b}^{3} &=&\left\{ k\left( z\right) dz\wedge d\theta \wedge
ds\right\} ,
\end{eqnarray*}
where all of the functions above are smooth in a neighborhood of $\left[ 0,1%
\right] $. An elementary calculation shows that $h^{0}=h^{1}=h^{2}=h^{3}=1 $%
, so that $\chi \left( M,\mathcal{F}\right) =0$.

We now compute the basic Euler characteristic using our theorem. The
stratification of $\left( M,\mathcal{F}\right) $ is $M\left( H_{1}\right)
\coprod M\left( H_{2}\right) $, where $M\left( H_{1}\right) $ is the union
of the two ``polar'' leaf closures ($z=\pm 1$) , and $M\left( H_{2}\right) $
is the complement of $M\left( H_{1}\right) $. Note that each orientation
bundle $\mathcal{O}_{M\left( H_{i}\right) \diagup \overline{\mathcal{F}}}$
is trivial. Next, $\chi \left( M\left( H_{2}\right) \diagup \overline{%
\mathcal{F}}\right) =\chi \left( \text{open interval}\right) =-1$, and $\chi
\left( M\left( H_{1}\right) \diagup \overline{\mathcal{F}}\right) =\chi
\left( \text{disjoint union of two points}\right) =2$.

\noindent Observe that $\chi \left( L_{1},\mathcal{F},\mathcal{O}_{M\left(
H_{1}\right) \diagup \overline{\mathcal{F}}}\right) =\chi \left( L_{1},%
\mathcal{F}\right)=0,$ since this is a taut, codimension-$1$ foliation.
Also, $\chi \left( L_{2},\mathcal{F}, \mathcal{O}_{M\left( H_{2}\right)
\diagup \overline{\mathcal{F}}}\right) =\chi \left( L_{2},\mathcal{F}\right) 
$ $=1-2+1=0$, since the basic forms restricted to $L_{2}$ consist of the
span of the set of closed forms $\left\{ 1,d\theta ,ds,d\theta \wedge
ds\right\} $. Thus, 
\begin{eqnarray*}
\chi \left( M,\mathcal{F}\right) &=&\sum_{i}\chi \left( M\left( H_{i}\right)
\diagup \overline{\mathcal{F}}\right) \chi \left( L_{i},\mathcal{F}, 
\mathcal{O}_{M\left( H_{i}\right) \diagup \overline{\mathcal{F}}}\right) \\
&=&2\cdot 0+\left( -1\right) \cdot 0=0,
\end{eqnarray*}
as we have already seen.

Note that taut foliations of odd codimension will always have a zero Euler
characteristic, by Poincare duality. Open Question: will these foliations
always have a zero basic index?
\end{example}

The following example is a codimension two transversally oriented Riemannian
foliation in which all the leaf closures have codimension one. The leaf
closure foliation is not transversally orientable, and the basic manifold is
a flat Klein bottle with an $O(2)$--action. The two leaf closures with $%
\mathbb{Z}_{2}$ holonomy correspond to the two orbits of type $\left( 
\mathbb{Z}_{2}\right) $, and the other orbits have trivial isotropy.

\begin{example}
This foliation is the suspension of an irrational rotation of the flat torus
and a $\mathbb{Z}_{2}$--action. Let $X$ be any closed Riemannian manifold
such that $\pi _{1}(X)=\mathbb{Z}*\mathbb{Z}$~, the free group on two
generators $\{\alpha ,\beta \}$. We normalize the volume of $X$ to be 1. Let 
$\widetilde{X}$ be the universal cover. We define $M=\widetilde{X}\times
S^{1}\times S^{1}\diagup \pi _{1}(X)$, where $\pi _{1}(X)$ acts by deck
transformations on $\widetilde{X}$ and by $\alpha \left( \theta ,\phi
\right) =\left( 2\pi -\theta ,2\pi -\phi \right) $ and $\beta \left( \theta
,\phi \right) =\left( \theta ,\phi +\sqrt{2}\pi \right) $ on $S^{1}\times
S^{1}$. We use the standard product--type metric. The leaves of $\mathcal{F}$
are defined to be sets of the form $\left\{ (x,\theta ,\phi )_{\sim
}\,|\,x\in \widetilde{X} \right\} $. Note that the foliation is
transversally oriented. The leaf closures are sets of the form 
\begin{equation*}
\overline{L}_{\theta }=\left\{ (x,\theta ,\phi )_{\sim }\,|\,x\in \widetilde{%
X},\phi \in [0,2\pi ]\right\} \bigcup \left\{ (x,2\pi -\theta ,\phi )_{\sim
}\,|\,x\in \widetilde{X},\phi \in [0,2\pi ]\right\}
\end{equation*}

The basic forms are: 
\begin{eqnarray*}
\Omega _{b}^{0} &=&\left\{ f\left( \theta \right) \right\} \\
\Omega _{b}^{1} &=&\left\{ g_{1}\left( \theta \right) d\theta +g_{2}(\theta
)d\phi \right\} \\
\Omega _{b}^{2} &=&\left\{ h(\theta )d\theta \wedge d\phi \right\} ,
\end{eqnarray*}%
where the functions are smooth and satisfy 
\begin{eqnarray*}
f\left( 2\pi -\theta \right) &=&f\left( \theta \right) \\
g_{i}\left( 2\pi -\theta \right) &=&-g_{i}\left( \theta \right) \\
h\left( 2\pi -\theta \right) &=&h\left( \theta \right) .
\end{eqnarray*}%
A simple argument shows that $h^{0}=h^{2}=1$ and $h^{1}=0$. Thus, $\chi
\left( M,\mathcal{F}\right) =2$. The basic manifold $\widehat{W}$ is an $%
O(2) $--manifold, defined by $\widehat{W}=[0,\pi ]\times S^{1}\diagup \sim $
, where the circle has length $1$ and $\left( \theta =0\text{ or }\pi
,\gamma \right) \sim \left( \theta =0\text{ or }\pi ,-\gamma \right) $. This
is a Klein bottle, since it is the connected sum of two projective planes. $%
O(2)$ acts on $\widehat{W}$ via the usual action on $S^{1}$.

Next, we compute the basic Euler characteristic using our theorem. The
stratification of $\left( M,\mathcal{F}\right) $ is $M\left( H_{1}\right)
\coprod M\left( H_{2}\right) $, where $M\left( H_{1}\right) $ is the union
of the two leaf closures $\theta _{2}=0$ and $\theta _{2}=\pi $, and $%
M\left( H_{2}\right) $ is the complement of $M\left( H_{1}\right) $. Note
that the orientation bundle $\mathcal{O}_{M\left( H_{2}\right) \diagup 
\overline{\mathcal{F}}}$ is trivial since an interval is orientable, and $%
\mathcal{O}_{M\left( H_{1}\right) \diagup \overline{\mathcal{F}}}$ is
trivial even though those leaf closures are not transversally oriented
(since the points are oriented!). Next, 
\begin{equation*}
\chi \left( M\left( H_{2}\right) \diagup \overline{\mathcal{F}}\right) =\chi
\left( \text{open interval}\right) =-1,
\end{equation*}%
and 
\begin{equation*}
\chi \left( M\left( H_{1}\right) \diagup \overline{\mathcal{F}}\right) =\chi
\left( \text{disjoint union of two points}\right) =2.
\end{equation*}%
Observe that $\chi \left( L_{2},\mathcal{F},\mathcal{O}_{M\left(
H_{2}\right) \diagup \overline{\mathcal{F}}}\right) =\chi \left( L_{2},%
\mathcal{F}\right) $ $=0$, since each representative leaf $L_{2}$ is a taut
(since it is a suspension), codimension $1$ foliation, and thus ordinary
Poincare duality holds (\cite{T},\cite{PaRi}): $\dim H_{B}^{0}\left( L_{2},%
\mathcal{F}\right) =\dim H_{B}^{1}\left( L_{2},\mathcal{F}\right) =1$. On
the other hand, $\chi \left( L_{1},\mathcal{F},\mathcal{O}_{M\left(
H_{1}\right) \diagup \overline{\mathcal{F}}}\right) =\chi \left( L_{1},%
\mathcal{F}\right) =1$, since each such leaf closure has $\dim
H_{B}^{0}\left( L_{1},\mathcal{F}\right) =1$ but $\dim H_{B}^{1}\left( L_{1},%
\mathcal{F}\right) =0$ since there are no basic one-forms. By our theorem,
we conclude that 
\begin{eqnarray*}
\chi \left( Y,\mathcal{F}\right) &=&\sum_{i}\chi \left( M\left( H_{i}\right)
\diagup \overline{\mathcal{F}}\right) \chi \left( L_{i},\mathcal{F},\mathcal{%
\ O}_{M\left( H_{i}\right) \diagup \overline{\mathcal{F}}}\right) \\
&=&2\cdot 1+\left( -1\right) \cdot 0=2,
\end{eqnarray*}%
as we found before by direct calculation.
\end{example}

The next example is a codimension two Riemannian foliation with dense
leaves, such that some leaves have holonomy but most do not. The basic
manifold is a point, the fixed point set of the $O\left( 2\right) $ action.
The isotropy group $O(2)$ measures the holonomy of some of the leaves
contained in the leaf closure.

\begin{example}
This Riemannian foliation is a suspension of a pair of rotations of the
sphere $S^{2}$. Let $X$ be any closed Riemannian manifold such that $\pi
_{1}(X)=\mathbb{Z}\ast \mathbb{Z}$~, that is the free group on two
generators $\{\alpha ,\beta \}$. We normalize the volume of $X$ to be 1. Let 
$\widetilde{X}$ be the universal cover. We define $M=\widetilde{X}\times
S^{2}\diagup \pi _{1}(X)$. The group $\pi _{1}(X)$ acts by deck
transformations on $\widetilde{X}$ and by rotations on $S^{2}$ in the
following ways. Thinking of $S^{2}$ as imbedded in $\mathbb{R}^{3}$, let $%
\alpha $ act by an irrational rotation around the $z$--axis, and let $\beta $
act by an irrational rotation around the $x$--axis. We use the standard
product--type metric. As usual, the leaves of $\mathcal{F}$ are defined to
be sets of the form $\left\{ (x,v)_{\sim }\,|\,x\in \widetilde{X}\right\} $.
Note that the foliation is transversally oriented, and a generic leaf is
simply connected and thus has trivial holonomy. Also, the every leaf is
dense. The leaves $\{\left( x,(1,0,0)\right) _{\sim }\}$ and $\{\left(
x,(0,0,1)\right) _{\sim }\}$ have nontrivial holonomy; the closures of their
infinitesimal holonomy groups are copies of $SO(2)$. Thus, a leaf closure in 
$\widehat{M}$ covering the leaf closure $M$ has structure group $SO(2)$ and
is thus all of $\widehat{M}$, so that $\widehat{W}$ is a point. The only
basic forms are constants and $2$ forms of the form $CdV$, where $C$ is a
constant and $dV$ is the volume form on $S^{2}$. Thus $h^{0}=h^{2}=1$ and $%
h^{1}=0$, so that $\chi \left( M,\mathcal{F}\right) =2$.

Our theorem in this case, since there is only one stratum, is 
\begin{eqnarray*}
\chi \left( M,\mathcal{F}\right) &=&\sum_{i}\chi \left( M\left( H_{i}\right)
\diagup \overline{\mathcal{F}}\right) \chi \left( L_{i},\mathcal{F},\mathcal{%
\ \ \ O}_{M\left( H_{i}\right) \diagup \overline{\mathcal{F}}}\right) \\
&=&\chi \left( \text{point}\right) \chi \left( M,\mathcal{F}\right) \\
&=&\chi \left( M,\mathcal{F}\right) ,
\end{eqnarray*}
which is perhaps not very enlightening.
\end{example}

The following example is a codimension two Riemannian foliation that is not
taut. This example is in \cite{Car}.

\begin{example}
Consider the flat torus $T^{2}=\mathbb{R}^{2}\diagup \mathbb{Z}^{2}$.
Consider the map $F:T^{2}\rightarrow T^{2}$ defined by 
\begin{equation*}
F\left( 
\begin{array}{c}
x \\ 
y%
\end{array}
\right) =\left( 
\begin{array}{cc}
2 & 1 \\ 
1 & 1%
\end{array}
\right) \left( 
\begin{array}{c}
x \\ 
y%
\end{array}
\right) \,\mod1
\end{equation*}
Let $M=[0,1]\times T^{2}\diagup \sim $, where $\left( 0,a\right) \sim \left(
1,F(a)\right) $. Let $v$, $v^{\prime }$ be orthonormal eigenvectors of the
matrix above, corresponding to the eigenvalues $\frac{3+\sqrt{5}}{2}$, $%
\frac{3-\sqrt{5}}{2}$, respectively. Let the linear foliation $\mathcal{F}$
be defined by the vector $v^{\prime }$ on each copy of $T^{2}$. Notice that
every leaf is simply connected and that the leaf closures are of the form $%
\{t\}\times T^{2}$, and this foliation is Riemannian if we choose a suitable
metric. For example, we choose the metric along $[0,1]$ to be standard and
require each torus to be orthogonal to this direction. Then we define the
vectors $v$ and $v^{\prime }$ to be orthogonal in this metric and let the
lengths of $v$ and $v^{\prime }$ vary smoothly over $[0,1]$ so that $\Vert
v\Vert (0)=\frac{3+\sqrt{5}}{2}\Vert v\Vert (1)$ and $\Vert v^{\prime }\Vert
(0)=\frac{3-\sqrt{5}}{2}\Vert v^{\prime }\Vert (1)$. Let $\overline{v}%
=a\left( t\right) v$, $\overline{v^{\prime }}=b\left( t\right) v^{\prime }$
be the resulting renormalized vector fields. The basic manifold is a torus,
and the isotropy groups are all trivial. We use coordinates $(t,x,y)\in
[0,1]\times T^{2}$ to describe points of $M$. The basic forms are: 
\begin{eqnarray*}
\Omega _{b}^{0} &=&\left\{ f\left( t\right) \right\} \\
\Omega _{b}^{1} &=&\left\{ g_{1}\left( t\right) dt+g_{2}(t)\overline{v}%
^{*}\right\} \\
\Omega _{b}^{2} &=&\left\{ h(t)dt\wedge \overline{v}^{*}\right\} ,
\end{eqnarray*}
where all the functions are smooth. Note that $d\overline{v}^{*}=-\,\frac{
a^{\prime }\left( t\right) }{a\left( t\right) }\,dt\wedge \overline{v}^{*}$
By computing the cohomology groups, we get $h^{0}=h^{1}=1$, $h^{2}=0$. Thus,
the basic Euler characteristic is zero.

We now compute the basic Euler characteristic using our theorem. There is
only one stratum, and the leaf closure space is $S^{1}$. The foliation
restricted to each leaf closure is an irrational flow on the torus. Thus, 
\begin{eqnarray*}
\chi \left( M,\mathcal{F}\right) &=&\sum_{i}\chi \left( M\left( H_{i}\right)
\diagup \overline{\mathcal{F}}\right) \chi \left( L_{i},\mathcal{F},\mathcal{%
\ O}_{M\left( H_{i}\right) \diagup \overline{\mathcal{F}}}\right) \\
&=&\chi \left( S^{1}\right) \chi \left( \{t\}\times T^{2},\mathcal{F}\right)
\\
&=&0\cdot 0=0,
\end{eqnarray*}%
as we have already seen.
\end{example}

Following is an example of using the representation-valued basic index
theorem, in this case applied to the Euler characteristic (Theorem \ref%
{representationValuedGaussBonnetThm}).

\begin{example}
Let $M=\mathbb{R}\times _{\phi }T^{2}$ be the suspension of the torus $T^{2}=%
\mathbb{R}^{2}\diagup \mathbb{Z}^{2}$, constructed as follows. The action $%
\phi :\mathbb{Z}\rightarrow \mathrm{Isom}\left( T^{2}\right) $ is generated
by a $\frac{\pi }{2}$ rotation. The Riemannian foliation $\mathcal{F}$ is
given by the $\mathbb{R}$-parameter curves. Explicitly, $k\in \mathbb{Z}$
acts on $\left( 
\begin{array}{c}
y_{1} \\ 
y_{2}%
\end{array}%
\right) $ by 
\begin{equation*}
\phi \left( k\right) \left( 
\begin{array}{c}
y_{1} \\ 
y_{2}%
\end{array}%
\right) =\left( 
\begin{array}{cc}
0 & -1 \\ 
1 & 0%
\end{array}%
\right) ^{k}\left( 
\begin{array}{c}
y_{1} \\ 
y_{2}%
\end{array}%
\right) .
\end{equation*}%
Endow $T^{2}$ with the standard flat metric. The basic harmonic forms have
basis $\left\{ 1,dy_{1},dy_{2},dy_{1}\wedge dy_{2}\right\} $. Let $\rho _{j}$
be the irreducible character defined by $k\in \mathbb{Z}\mapsto e^{ikj\pi
/2} $. Then the basic de Rham operator $\left( d+\delta _{b}\right) ^{\rho
_{0}}$ on $\mathbb{Z}$-invariant basic forms has kernel $\left\{
c_{0}+c_{1}dy_{1}\wedge dy_{2}:c_{0},c_{1}\in \mathbb{C}\right\} $. One also
sees that $\ker \left( d+\delta _{b}\right) ^{\rho _{1}}=\mathrm{span}%
\left\{ idy_{1}+dy_{2}\right\} $, $\ker \left( d+\delta _{b}\right) ^{\rho
_{2}}=\left\{ 0\right\} $, and $\ker \left( d+\delta _{b}\right) ^{\rho
_{3}}=\mathrm{span}\left\{ -idy_{1}+dy_{2}\right\} $. Then 
\begin{equation*}
\chi ^{\rho _{0}}\left( M,\mathcal{F}\right) =2,\chi ^{\rho _{1}}\left( M,%
\mathcal{F}\right) =\chi ^{\rho _{3}}\left( M,\mathcal{F}\right) =-1,\chi
^{\rho _{2}}\left( M,\mathcal{F}\right) =0.
\end{equation*}%
This illustrates the point that it is not possible to use the Atiyah-Singer
integrand on the quotient of the principal stratum to compute even the
invariant index alone. Indeed, the Atiyah-Singer integrand would be a
constant times the Gauss curvature, which is identically zero. In these
cases, the three singular points $a_{1}=\left( 
\begin{array}{c}
0 \\ 
0%
\end{array}%
\right) ~,a_{2}=\left( 
\begin{array}{c}
0 \\ 
\frac{1}{2}%
\end{array}%
\right) ~,a_{3}=\left( 
\begin{array}{c}
\frac{1}{2} \\ 
\frac{1}{2}%
\end{array}%
\right) $ certainly contribute to the index. The quotient $M\diagup 
\overline{\mathcal{F}}$ is an orbifold homeomorphic to a sphere.

We now compute the Euler characteristics $\chi ^{\rho }\left( M,\mathcal{F}%
\right) $ using Theorem \ref{representationValuedGaussBonnetThm}. The strata
of the foliation are as follows. The leaves corresponding to $a_{1}$ and $%
a_{3}$ comprise the most singular stratum $M_{s}$ with isotropy $\mathbb{Z}%
_{4}$, and the leaf correspondng to $a_{2}$ is its own stratum $M_{l}$ with
isotropy isomorphic to $\mathbb{Z}_{2}$. Then 
\begin{eqnarray*}
\chi \left( M_{s}\diagup \overline{\mathcal{F}}\right) &=&2, \\
\chi \left( M_{l}\diagup \overline{\mathcal{F}}\right) &=&1, \\
\chi \left( M_{0}\diagup \overline{\mathcal{F}}\right) &=&\chi \left(
S^{2}\smallsetminus \left\{ \text{3 points}\right\} \right) =-1.
\end{eqnarray*}

In each stratum ($M_{0}$, $M_{l}$, or $M_{s}$), the representative leaf
closure is a circle, a single leaf, and each stratum is transversally
oriented. The Euler characteristic $\chi ^{\rho }\left( L_{j},\mathcal{F}%
\right) $ is one if there exists a locally constant section of the line
bundle associated to $\rho $ over $L_{j}$, and otherwise it is zero. We see
that 
\begin{equation*}
\chi ^{\rho }\left( L_{j},\mathcal{F},\mathcal{O}_{M_{j}\diagup \overline{%
\mathcal{F}}}\right) =\chi ^{\rho }\left( L_{j},\mathcal{F}\right) =\left\{ 
\begin{array}{ll}
1 & \text{if }M_{j}=M_{0}\text{ and }\rho =\rho _{0},\rho _{1},\rho _{2},%
\text{ or }\rho _{3} \\ 
1 & \text{if }M_{j}=M_{s}\text{ and }\rho =\rho _{0} \\ 
1 & \text{if }M_{j}=M_{l}\text{ and }\rho =\rho _{0}\text{ or }\rho _{2} \\ 
0~ & \text{otherwise}%
\end{array}%
\right. .
\end{equation*}%
Then Theorem \ref{representationValuedGaussBonnetThm} implies%
\begin{eqnarray*}
\chi ^{\rho }\left( M\right) &=&\left( -1\right) \left\{ 
\begin{array}{ll}
1 & \text{if }\rho =\rho _{0},\rho _{1},\rho _{2},\text{ or }\rho _{3} \\ 
0~ & \text{otherwise}%
\end{array}%
\right. +\left( 1\right) \left\{ 
\begin{array}{ll}
1 & \text{if }\rho =\rho _{0}\text{ or }\rho _{2} \\ 
0~ & \text{otherwise}%
\end{array}%
\right. +\left( 2\right) \left\{ 
\begin{array}{ll}
1 & \text{if }\rho =\rho _{0} \\ 
0~ & \text{otherwise}%
\end{array}%
\right. \\
&=&\left\{ 
\begin{array}{ll}
-1 & \text{if }M_{j}=M_{0}\text{ and }\rho =\rho _{1}\text{ or }\rho _{3} \\ 
2 & \text{if }M_{j}=M_{s}\text{ and }\rho =\rho _{0} \\ 
0 & \text{if }M_{j}=M_{l}\text{ and }\rho =\rho _{2} \\ 
0~ & \text{otherwise}%
\end{array}%
\right. ,
\end{eqnarray*}%
which agrees with the previous direct calculation.
\end{example}

%%\subsection{Foliations constructed by suspension}\label{92}

%%\subsection{Codimensions 1 and 2}\label{93}

%%\section{Vanishing Theorems}\label{10}

\end{document}